\documentclass[12pt]{amsart} 
\usepackage{fullpage}
\usepackage{setspace}
\usepackage{amsmath} 
\usepackage{amsfonts} 
\usepackage{graphicx} 
\usepackage{hyperref}
\usepackage{courier}
\usepackage{subfigure} 
\usepackage{color} 
\usepackage[english]{babel}
\usepackage{mathbbol}
\usepackage{multirow}
\usepackage{subfigure}
\usepackage{latexsym}
\usepackage{pdfsync}
\usepackage{epsfig}
\usepackage{subfigure}

\definecolor{Blue}{rgb}{0.3,0.3,0.9}
\usepackage{amsmath}
\usepackage{amsthm,amsbsy}
\usepackage{mathcomp}
\usepackage{textcomp}

\usepackage{amssymb}
\usepackage{graphicx}
\usepackage{graphicx}
\usepackage{dcolumn}
\usepackage{bm}
\usepackage{amsfonts}
\usepackage{latexsym}
\usepackage{pdfsync}
    \usepackage{fullpage} 
\usepackage{colordvi}
\usepackage{color}
\usepackage{booktabs}
\usepackage{graphicx}
\usepackage{subfigure}
\usepackage{stmaryrd}
\usepackage{cite}
\input xy
\xyoption{all}

\newtheorem{thm}{Theorem}[section]
\newtheorem{cor}[thm]{Corollary}
\newtheorem{prop}[thm]{Proposition}
\newtheorem{rmk}[thm]{Remark}

\newcommand{\commentout}[1]{}

\newcommand{\nwc}{\newcommand}
\newcommand{\bz}{{\mathbf z}}

\newcommand{\lt}{\left}

\newcommand{\rt}{\right}

\nwc{\bR}{\mb R}
\nwc{\bH}{{\mb H}}
\nwc{\bxp}{{{\mathbf x}}}
\nwc{\bap}{{{\mathbf y}}}

\nwc{\bPhi}{\mathbf{\Phi}}
\nwc{\bPsi}{\mathbf{\Psi}}
\nwc{\bh}{\mathbf h}

\nwc{\bI}{\mathbf I}
\nwc{\bP}{\mathbf P}
\nwc{\bs}{\mathbf s}
\nwc{\bd}{\mathbf{d}}
\nwc{\bX}{\mathbf X}

\nwc{\om}{\omega}

\nwc{\nwt}{\newtheorem}
\nwc{\xp}{{x^{\perp}}}
\nwc{\yp}{{y^{\perp}}}
\nwt{remark}{Remark}
\nwt{definition}{Definition} 
\nwt{corollary}{Corollary} 

\nwc{\ba}{{\mb a}}
\nwc{\bal}{\begin{align}}
\nwc{\ben}{\begin{equation*}}
\nwc{\bea}{\begin{eqnarray}}
\nwc{\beq}{\begin{eqnarray}}
\nwc{\bean}{\begin{eqnarray*}}
\nwc{\beqn}{\begin{eqnarray*}}
\nwc{\beqast}{\begin{eqnarray*}}

\nwc{\eal}{\end{align}}
\nwc{\een}{\end{equation*}}
\nwc{\eea}{\end{eqnarray}}
\nwc{\eeq}{\end{eqnarray}}
\nwc{\eean}{\end{eqnarray*}}
\nwc{\eeqn}{\end{eqnarray*}}
\nwc{\eeqast}{\end{eqnarray*}}

\nwc{\vep}{\varepsilon}
\nwc{\ep}{\epsilon}
\nwc{\ept}{\epsilon}
\nwc{\vrho}{\varrho}
\nwc{\orho}{\bar\varrho}
\nwc{\ou}{\bar u}
\nwc{\vpsi}{\varpsi}
\nwc{\lamb}{\lambda}
\nwc{\Var}{{\rm Var}}

\nwt{proposition}{Proposition}
\nwt{theorem}{Theorem}
\nwt{summary}{Summary}
\nwt{lemma}{Lemma}
\nwc{\nn}{\nonumber}
\nwc{\mf}{\mathbf}
\nwc{\mb}{\mathbf}
\nwc{\ml}{\mathcal}

\nwc{\IA}{\mathbb{A}} 
\nwc{\bi}{\mathbf i}
\nwc{\bo}{\mathbf o}
\nwc{\IB}{\mathbb{B}}
\nwc{\IC}{\mathbb{C}} 
\nwc{\ID}{\mathbb{D}} 
\nwc{\IM}{\mathbb{M}} 
\nwc{\IP}{\mathbb{P}} 
\nwc{\II}{\mathbb{I}} 
\nwc{\IE}{\mathbb{E}} 
\nwc{\IF}{\mathbb{F}} 
\nwc{\IG}{\mathbb{G}} 
\nwc{\IN}{\mathbb{N}} 
\nwc{\IQ}{\mathbb{Q}} 
\nwc{\IR}{\mathbb{R}} 
\nwc{\IT}{\mathbb{T}} 
\nwc{\IZ}{\mathbb{Z}} 

\nwc{\cE}{{\ml E}}
\nwc{\cP}{{\ml P}}
\nwc{\cQ}{{\ml Q}}
\nwc{\cL}{{\ml L}}
\nwc{\cX}{{\ml X}}
\nwc{\cW}{{\ml W}}
\nwc{\cZ}{{\ml Z}}
\nwc{\cR}{{\ml R}}
\nwc{\cV}{{\ml V}}
\nwc{\cT}{{\ml T}}
\nwc{\crV}{{\ml L}_{(\delta,\rho)}}
\nwc{\cC}{{\ml C}}
\nwc{\cO}{{\ml O}}
\nwc{\cA}{{\ml A}}
\nwc{\cK}{{\ml K}}
\nwc{\cB}{{\ml B}}
\nwc{\cD}{{\ml D}}
\nwc{\cF}{{\ml F}}
\nwc{\cS}{{\ml S}}
\nwc{\cM}{{\ml M}}
\nwc{\cG}{{\ml G}}
\nwc{\cH}{{\ml H}}
\nwc{\bk}{{\mb k}}
\nwc{\bn}{{\mb n}}
\nwc{\cbz}{\overline{\cB}_z}
\nwc{\supp}{{\hbox{supp}}}
\nwc{\fR}{\Re}
\nwc{\bY}{\mathbf Y}

\nwc{\pft}{\cF^{-1}_2}
\nwc{\bU}{{\mb U}}
\nwc{\bG}{{\mb G}}
\nwc{\bg}{\mathbf{g}}
\nwc{\mbf}{\mathbf{f}}
\nwc{\mbe}{\mathbf{e}}
\nwc{\be}{\mathbf{e}}
\nwc{\Om}{\Omega}
\nwc{\ind}{\operatorname{I}}
\nwc{\mbx}{\mathbf{f}}
\nwc{\bb}{\mathbf{g}}
\nwc{\xmax}{f_{\rm max}}
\nwc{\xmin}{f_{\rm min}}
\nwc{\suppx}{\hbox{\rm supp} (\mbf)}
\nwc{\by}{\mathbf{h}}
\nwc{\bZ}{\mathbf{Z}}
\nwc{\bF}{\mathbf{F}}
\nwc{\bE}{\mathbf{E}}
\nwc{\bV}{\mathbf{V}}
\nwc{\cI}{\IZ^2_N}
\nwc{\chis}{{\chi^{\rm s}}}
\nwc{\chii}{{\chi^{\rm i}}}
\nwc{\pdfi}{{f^{\rm i}}}
\nwc{\pdfs}{{f^{\rm s}}}
\nwc{\pdfii}{{f_1^{\rm i}}}
\nwc{\pdfsi}{{f_1^{\rm s}}}
\nwc{\thetatil}{{\tilde\theta}}
\nwc{\red}{\color{red}}
\nwc{\prox}{\hbox{prox}}
\nwc{\diag}{\hbox{\rm diag}}
\nwc{\sloc}{J_{\rm f}}
\nwc{\bu}{\xi}
\nwc{\bv}{\eta}
\nwc{\cU}{\mathcal{U}}
\nwc{\cN}{\mathcal{N}}
\nwc{\bN}{\mathbf{N}}
\nwc{\mbm}{\mathbf{m}}
\nwc{\bw}{\mathbf{w}}
\nwc{\im}{i}
\nwc{\bom}{\mathbf{w}}
\nwc{\bt}{\mathbf{t}}
\nwc{\z}{y}
\nwc{\cY}{\mathcal{Y}}
\nwc{\bM}{\mathbf{M}}
\nwc{\half}{{1\over 2}}
\begin{document}
\title{Fourier Phase Retrieval with a Single Mask by Douglas-Rachford Algorithm}

\author{  Pengwen Chen
\address{ Applied Mathematics, National Chung Hsing University, Taichung 402, Taiwan.  } \and Albert Fannjiang 
 \address{
Department of Mathematics, University of California, Davis, California  95616, USA.
}
 }
\date{}%
  \maketitle
 \begin{abstract} 
 Douglas-Rachford (DR) algorithm is analyzed for Fourier 
 phase retrieval with a single random phase mask.
 Local, geometric convergence  to a unique fixed point  is proved with numerical demonstration of
 global convergence.

\bigskip

\end{abstract}


\thispagestyle{plain}

\section{Introduction}\label{intro}
X-ray crystallography has been the preferred technology for determining the  structure of a biological molecule over the past hundred years. The method, however, is limited by crystal quality, radiation damage and phase
determination \cite{MC}.
 The first two problems call for large crystals that yield sufficient diffraction intensities while reducing the dose to individual molecules in the crystal. 
The difficulty of growing  large, well-diffracting crystals is thus the  major bottleneck of X-ray crystallography - a necessary experimental step that can range from merely challenging to pretty much impossible, particularly for large macromolecular assemblies and membrane proteins.

By boosting the brightness of available X-rays by 10 orders of magnitude and producing pulses well below 100 fs duration, X-ray free electron lasers (XFEL)  offer the possibility of  extending  structural studies to {\em single, non-crystalline} particles or molecules  by using short intense pulses that out-run radiation damage, thus circumventing the first two aforementioned  problems \cite{Hajdu0}.  In the so-called {\em diffract-before-destruct} approach \cite{Chapman14, Chapman11, Hajdu},  
a stream of particles is flowed across the XFEL beam and randomly hit by a single X-ray pulse, forming a single diffraction pattern before being vaporized as a nano-plasma burst.
Each diffraction pattern contains certain information about the planar  projection of the scattering potential of the object along  the direction of the beam which is to be recovered by phase retrieval techniques \cite{BW}. 


The modern approach to phase retrieval for non-periodic objects
roughly starts with the Gerchberg-Saxton algorithm \cite{GS72}, followed by its
variant, Error Reduction (ER), and the more powerful  Hybrid-Input-Output (HIO) algorithm\cite{Fie82,Fie13}. These form the cornerstones of the standard  {\em iterative transform
algorithms} (ITA) \cite{BCL02, Mar07}. 

However, the standard ITA tend to stagnate and 
do not perform well without 
additional  prior information, such as tight support and positivity.
The reason is that 
the {\em plain} diffraction pattern alone  does not guarantee uniqueness of
solution (see \cite{Vetterli}, however, for uniqueness under additional prior information).
 On the contrary, many phasing solutions exist for a given diffraction pattern, resulting in
what is called the {\em phase} problem \cite{Hau91}. 

\begin{figure}[h]
\centerline{\includegraphics[width=12cm]{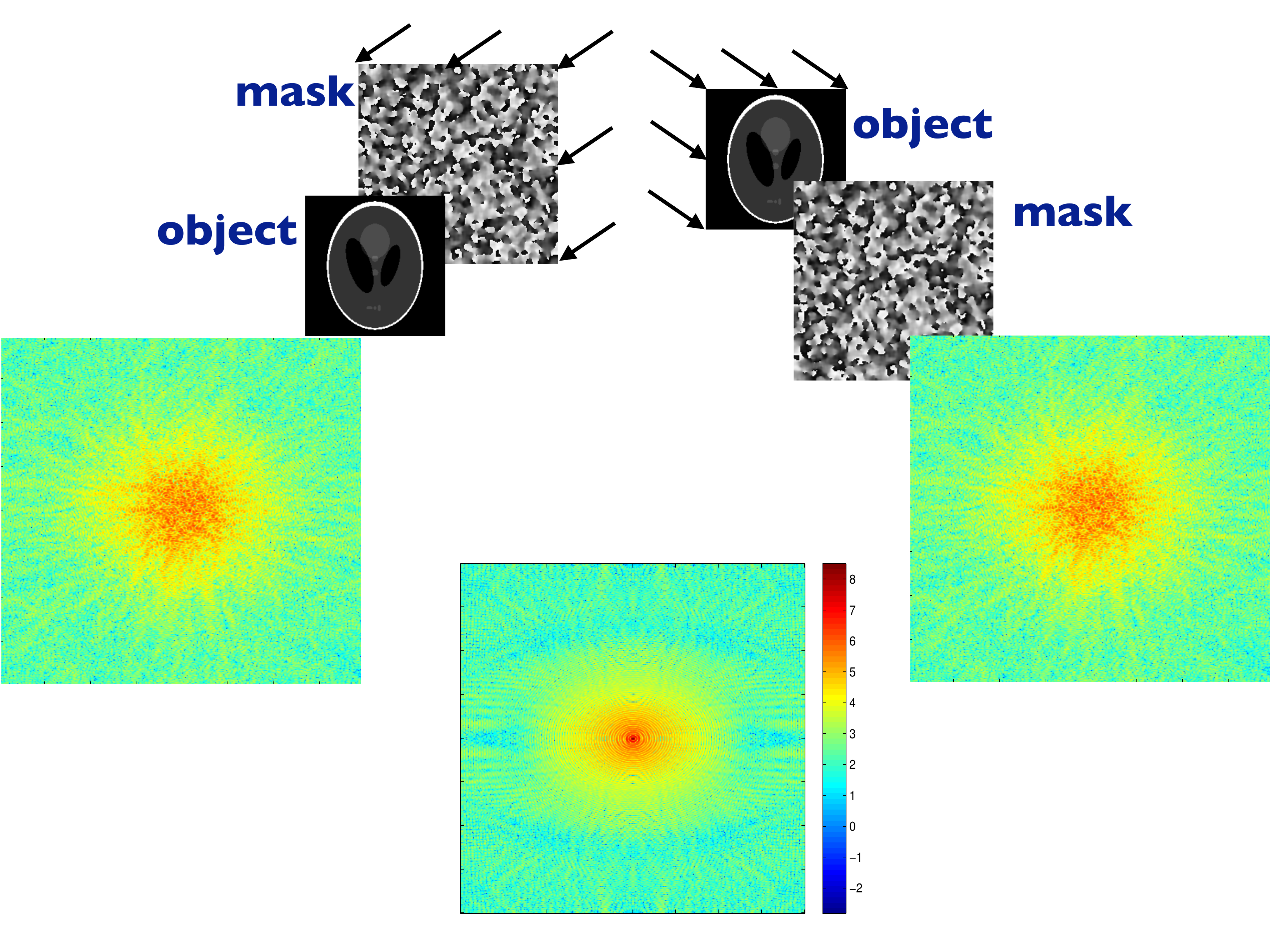}}
  \caption{Conceptual layout of coherent lensless imaging with a random mask (left) {\em before} (for random illumination) or (right) {\em behind} (for wavefront sensing) the object. 
 (middle) The diffraction pattern measured without a mask has a larger dynamic range. The color bar is on a logarithmic scale.} 
 \label{fig:mask}
\end{figure}

To this end, a promising approach is to measure the diffraction pattern with
a {\em single} random mask and use the  {  coded diffraction pattern} as the data.
As shown in \cite{unique}, the uniqueness of solution  is restored with a high probability
given any  scattering potential  whose value is 
restricted to a known sector (say, the upper half plane) of the complex plane (see Proposition \ref{prop6.1}). 

Indeed, the sector constraint is a  practical, realistic condition to impose
on almost all materials as the imaginary part of the scattering potential is proportional to
the (positive) extinction  coefficient with the upper half plane
as the sector constraint \cite{BW}. For X-ray regime, the real part of the scattering potential is typically slightly negative 
which with a nonnegative imaginary part gives to the second quadrant as the sector constraint
 \cite{X}.

What happens if the sector condition is not met and consequently one coded diffraction pattern
is not enough to ensure uniqueness? This question is particularly pertinent to 
the  diffract-before-destruct approach as the particle  can not withstand the radiation damage
from more than one XFEL pulses. 

A plausible measurement scheme  is to guide the transmitted field (the transmission function  \cite{BW}) from a {\em planar} illumination through a beam splitter \cite{splitter}, generating  two copies of
the transmitted field which  are then measured separately 
as a coded diffraction pattern and a plain diffraction pattern.  In this set-up,
the object function  is the transmitted field behind the particle and the phase retrieval problem becomes the wave-front reconstruction problem \cite{BW,Hardy}. In practice beam splitters and the masks (or any measurement devices)  should be used as sparingly as possible to avoid introducing
excessive  measurement noises.

  With two diffraction patterns, the uniqueness of solution (in the  above sense)
is restored almost surely  without the sector constraint (see Proposition \ref{prop6.1} and 
Remark \ref{rmk6.3}). 

With the uniqueness-ensuring  sampling schemes (Section \ref{sec:coded}), {\em ad hoc}  combinations of members of ITA (such as HIO and ER) can be devised  to recover the true solution \cite{rpi,pum}.  There is, however, no convergence proof for these algorithms.

The main goal of the paper is to prove the {\em local, geometric convergence}  
of the Douglas-Rachford (DR) algorithm  to a unique fixed point in the case of  one or two oversampled diffraction patterns (Theorems \ref{thm1}, \ref{cor5.2} and \ref{thm4}) and 
demonstrate {\em global convergence} numerically  (Section \ref{sec:num}).
 
DR  has the following general form: 
 Let $P_1$ and $P_2$ be  the projections  onto the two
constraint sets, respectively. For phase retrieval, $P_1$ describes  the projection onto
the set of diffracted fields (instead of diffraction patterns) and
$P_2$ the data fitting. 
 The Douglas-Rachford (DR) algorithm is defined by
the iteration scheme \cite{EB92,LM79}
\beq\label{dr}
y^{(k+1)}=y^{(k)}+P_1 (2P_2-I)y^{(k)}-P_2 y^{(k)},\quad k=1,2,3...
\eeq
Closely
related to HIO, DR also belongs to the ITA family  (Section \ref{sec:ITA}). ITA are computationally efficient
thanks to  the fast Fourier transform (FFT) and explicit nature of $P_1, P_2$ 
(see \eqref{proj} below). 

\subsection{Oversampled diffraction patterns}\label{sec:coded}
Next we describe our sampling schemes before we can properly introduce 
$P_1, P_2$ and the Douglas-Rachford algorithm for phase retrieval (Section \ref{sec:ITA}).

 Let $f(\bn)$ be a discrete  object function with $\bn = (n_1,n_2,...,n_d) \in \IZ^d$. 
Consider  the {object space} consisting  of all functions  supported in 
$  \cM = \{ 0\le m_1\le M_1, 0\le m_2\le M_2,..., 0\leq m_d\leq M_d\}$. 
We assume $d\geq 2$. 

With a {coherent illumination} under the Fraunhofer 
approximation, the free-space propagation between the object plane and the sensor  plane 
can be described by the Fourier transform \cite{BW} (with the proper coordinates and normalization). However, only the {\em intensities} of the Fourier transform
are measured on the sensor plane and constitute  the so called {\em diffraction pattern} given by 
 \beq
   \sum_{\bn =-\bM}^{\bM}\sum_{\mbm+\bn\in \cM} f(\mbm+\bn)\overline{f(\mbm)}
   e^{-\im 2\pi \bn\cdot \bom},\quad \bom=(w_1,...,w_d)\in [0,1]^d,\quad \bM = (M_1,...,M_d)\nn
   \eeq
   which is the Fourier transform of the autocorrelation
   \beqn
	  R_{ f}(\bn)=\sum_{\mbm\in \cM} f(\mbm+\bn)\overline{f(\mbm)}.
	  \eeqn
Here and below the over-line notation means
complex conjugacy. 

Note that
$R_f$ is defined on the enlarged  grid
 \begin{equation*}
 \widetilde \cM = \{ (m_1,..., m_d)\in \IZ^d: -M_1 \le m_1 \le M_1,..., -M_d\le m_d\leq M_d \} 
 \end{equation*}
whose cardinality is roughly $2^d$ times that of $\cM$.
Hence by sampling  the diffraction pattern
 on the grid 
\beqn
\cL = \Big\{(w_1,...,w_d)\ | \ w_j = 0,\frac{1}{2 M_j + 1},\frac{2}{2M_j + 1},...,\frac{2M_j}{2M_j + 1}\Big\}
\eeqn
we can recover the autocorrelation function by the inverse Fourier transform. This is the {\em standard oversampling} with which  the diffraction pattern and the autocorrelation function become equivalent via the Fourier transform \cite{Miao00,MSC}.


A coded diffraction pattern is measured with a mask
whose effect is multiplicative and results in  
a {\em masked object}  of the form $
g(\bn) =f(\bn) \mu(\bn)$ 
where $\{\mu(\bn)\}$ is an array of random variables representing the mask.   
In other words, a coded diffraction pattern is just the plain diffraction pattern of
a masked object. 

We will focus on the effect of {\em random phases} $\phi(\bn)$ in the mask function 
$
\mu(\bn)=|\mu|(\bn)e^{\im \phi(\bn)}
$
where  $\phi(\bn)$ are independent, continuous real-valued random variables and $|\mu|(\bn)\neq 0,\forall \bn\in \cM$ (i.e. the mask is transparent). 

Without loss of generality, we assume $|\mu(\bn)|=1,\forall\bn$ which gives rise to
a {\em phase} mask and an {\em isometric}  propagation matrix 
\beq
\label{one}
\hbox{\rm (1-mask )}\quad A^*= c\Phi\,\, \diag\{\mu\},
\eeq
i.e. $AA^*=I$ (with a proper choice of the normalizing constant  $c$), where $\Phi$ is the {\em oversampled}  $d$-dimensional discrete Fourier transform (DFT). Specifically  $\Phi \in \IC^{|\tilde \cM|, |\cM|}$ is the sub-column matrix of
the standard DFT on  the extended grid $\tilde \cM$ where $|\cM|$ is
the cardinality of $\cM$.  

When two phase masks $\mu_1, \mu_2$ are deployed, 
the propagation matrix $A^*$ is the stacked coded DFTs, i.e.   
\beq \label{two}\hbox{(2-mask case)}\quad 
A^*=c \lt[\begin{matrix}
\Phi\,\, \diag\{\mu_1\}\\
\Phi\,\, \diag\{\mu_2\}
\end{matrix}\rt]
\eeq
 With proper normalization, $A^*$ is
isometric. 

In line with the spirit of simplifying measurement complexity discussed above, 
we remove the second mask (i.e. $\mu_2\equiv 1$) and consider 
\commentout{
we shall make an observation to reduce the number of coded patterns by one, while keeping
the total number of diffraction patterns unchanged.

In terms of the more elaborate notation of Section \ref{sec:coded}, we introduce
the new unknown masked object function $g(\bn)=\mu^{-1}_2(\bn) f(\bn)$, a new phase mask $\mu=  \mu_1\mu_2^{-1}$ and
}
the  propagation matrix \cite{unique,rpi,pum}
\beq \label{two'}
\hbox{(1$\half$-mask case)}\quad A^*=c \lt[\begin{matrix}
\Phi\,\, \diag\{\mu\}\\
\Phi\,\, 
\end{matrix}\rt]
\eeq
normalized to be isometric, where $\mu$ is independently and continuously
distributed over the unit circle. In other words,  one oversampled coded pattern and one oversampled plain pattern are used for reconstruction. 

For convenience, we shall refer to this set-up as the $1\half$-mask case to distinguish it
from the one- and two-mask cases. This and its extension to the multi-mask cases would be
the set-up for all our numerical simulations in Section \ref{sec:num}.


\commentout{
\begin{figure}[h]
\centerline{\includegraphics[width=6cm]{sector.pdf}}
  \caption{Orthogonal projection $P$  onto the convex sector with $a+b\leq 1$ for $\bn\in \cM$. The shaded area is projected to 0 (the origin). } 
 \label{fig:proj}
\end{figure}
}

\subsection{Related literature}
For the optical spectrum,  experiments with coded diffraction patterns  are not new
and can be  implemented 
 by computer generated holograms \cite{BWW}, random phase plates \cite{AH1} and
liquid crystal phase-only panels \cite{FAK}.  Recently, a phase mask with
randomly distributed pinholes  has been
implemented  for soft X-ray \cite{ptycho-rpi}. 

Coded-aperture phase retrieval was formulated as a convex trace-norm
 minimization problem in \cite{phaselift0,phaselift1,Papa} whose uniqueness
was proved in \cite{CLS1} under the assumption that the number of independently
coded diffraction patterns is sufficiently large (polylogarithmic in
$|\cM|$). 

The convex formulation   \cite{phaselift0,phaselift1, CLS1} certainly has a tremendous appeal as global convergence can be expected for any proper numerical implementations. 
However, 
due to the lift to much higher dimensions, the convex  program may be computationally
expensive (see also \cite{BM1,Gross}). 

Alternative non-convex minimization formulations 
were proposed and solved by
various gradient methods \cite{CLS2,Mig11}. In practice, these 
algorithms are locally convergent with a comparatively large number ($\geq 6$) of coded diffraction patterns.  

An important difference between
the measurement schemes in 
these papers and the present work (as well as \cite{unique,rpi,pum}) is
that their coded diffraction patterns are {\em not} oversampled. Another distinctive  feature of the present setting is that the dimension $d\geq 2$
is required for the spectral gap (Theorem \ref{cor5.2}) and the uniqueness of fixed point (Theorem \ref{thm4}), but not for the structure theorem (Theorem \ref{thm1}). 

In this connection, we emphasize that
reducing the number of coded diffraction patterns
is crucial for the diffract-before-destruct approach
and, in our view,  oversampling is a small price to pay with  current sensor technologies. 

A bigger price may be  that we lose the robust injectivity property
prized in these works (also see  \cite{BY, BM2}). In other words,  just one or two random masks,
the phase retrieval map $F$ defined as $F(x)\equiv |A^* x|$ with $A^* $ given by \eqref{two'} or \eqref{two}  is injective
only after  a certain finite set is excluded from $\IC^{|\cM|}$. 
On the other hand, for {any}  given $f$,  with probability one 
the selection of the mask(s) is such that no other object, modulo a phase factor,  produces the same data as
$|A^* f|$ with  $A^*$ given by \eqref{two'} or \eqref{two}. Our numerical results show
that this notion of uniqueness appears suffices for most practical purposes when
DR is implemented for phase retrieval.

However, to the best of our knowledge, even local convergence is not known 
for any ITA for Fourier phase retrieval. The  present paper aims to fill
this gap. 
While we can not prove {\em global} convergence of DR as for the convex setting,  we will present strong numerical evidence for
global convergence. In \cite{ER}, we prove the local, geometric convergence of the error reduction algorithm
with the same measurement scheme described in  Section \ref{sec:coded}. 
Local convergence of ER in the case of Gaussian matrix $A^*$ was recently proved in \cite{NJS}. 

There is much more literature on phase retrieval with generic frames and  independent random matrices  \cite{Balan, BY, Balan1, Balan2, BM2,CC,  DH12,phaselift1, Eldar, Xin,Mallat} which have
a somewhat different flavor from that with  Fourier phase retrieval  as the latter
is neither of the former two. Consequently, some of our results (Theorems \ref{cor5.2} and
\ref{thm4}) can not extend to these other cases without major revision. The structure theorem (Theorem \ref{thm1})  remains valid, however (see Remark \ref{rmk:generic}).

There also is a growing body of work  on phase retrieval under sparsity assumptions, see \cite{Oh,LV,Hassibi, Hassibi2} and  the references therein. 

The rest of the paper is organized as follows. In Section \ref{sec:not}, we simplify
the notation for presenting the main results and the proof of local convergence.
In Section \ref{sec:ITA}, we describe the DR algorithm widely used
in convex optimization problems and formulate it for the non-convex problem of phase retrieval. In Section \ref{sec:proof}, we prove  local convergence of FDR under the spectral gap assumption (Theorem \ref{thm1}). In Section \ref{sec:gap}, we prove the spectral gap condition for any number of oversampled diffraction patterns (Theorem \ref{cor5.2}). In Section \ref{sec:fixed}, we prove
that the fixed point of DR is unique for one or two oversampled diffraction patterns (Theorem \ref{thm4}). In Section \ref{sec:num} we give 
numerical examples and demonstrate global convergence of
the DR scheme. 

\section{Set-up and notation} \label{sec:not}

For the main text below, we shall simplify  the notion as follows. The more elaborate notion of Section \ref{sec:coded}  will be useful again in the appendix. 

First, we convert the $d$-dimensional grid into an ordered set of index.  The unknown  object will now be  denoted by  $x_0\in \IC^n, n=|\cM|$. In other words, $x_0$ is the vectorized version of the object function $f$ supported in $\cM\subset \IZ^d, d\geq 2$ (Section \ref{sec:coded}). \\

{\bf Rank-2 property:}  $x_0$ is rank-2 (or higher)  if the convex hull of  $\supp\{f\}$ is two (or higher) dimensional. \\

Let $\cX$ be a nonempty closed convex  set in $\IC^n$ and 
\begin{equation}
[x]_\cX=\hbox{\rm arg}\min_{x'\in \cX} \|x'-x\|
\end{equation}
the projection onto $\cX$. \\

{\bf Sector constraint:}  $x_0$ satisfies the sector constraint
if the principal value of $\arg{x_0(j)},\forall j$ is
restricted  to a {\em sector}  $[-\alpha\pi,\beta\pi] \subsetneq (-\pi,\pi], \forall \bn$. As mentioned above 
 almost all scattering potentials $f$ 
have a nonnegative  imaginary part and hence satisfy the sector constraint with $\alpha=0,\beta=1$. The sector constraint serves as transition between the standard  positivity constraint ($\alpha=\beta=0$) and the null constraint ($\alpha=\beta=1$).

The sector projection is explicitly given as follows: 
For $j\leq n$  
 \beq
\label{5'}[x]_\cX(j)&=& \left \{ \begin{array}{ll} x(j) & \text{ if }  \measuredangle{x(j)} \in [-\alpha\pi,\beta\pi] \\
 \Re[x(j)e^{-i\beta\pi}] e^{i\beta\pi} & \text{ if } \measuredangle{x(j)} \in [\beta\pi, (\beta+1/2)\pi]\\
\Re[x(j)e^{i\alpha\pi}] e^{-i\alpha\pi}& \text{ if }\measuredangle{x(j)} \in [-(\alpha+1/2)\pi, -\alpha\pi]\\ 
0 & \text{ else}
\end{array} \right.
\eeq
and $[x]_\cX(j)=0, j>n+1$. \\

{\bf Phase retrieval problem}.  For a given unknown object $x_0$ of rank $\geq 2$, let $A^*=[a^*_j]\in \IC^{N\times n}$ be the propagation matrix given by \eqref{one}, \eqref{two} or
\eqref{two'} where $A^*$ is normalized to be isometric and $
b=|A^*x_0|\in \IR^N$ be the data vector. Phase retrieval is to find a solution $x$ to 
the equation 
\beq
\label{0}
 b=|A^* x|,\quad x\in \cX.
\eeq
We focus on two cases. \\

{{\bf 1) One-pattern case:} $A^*$ is given by \eqref{one}, $[x]_\cX$  is given by \eqref{5'}}. \\

{{\bf 2) Two-pattern case:} $A^*$ is given by \eqref{two} or \eqref{two'},  $\cX=\IC^n$}.  \\

Phasing solution is unique only up to a constant of modulus one
 no matter how many coded diffraction patterns are measured.
 Thus the proper error metric for an estimate $\hat x$ of the true solution $x_0$  is given by
  \begin{equation}
 \min_{\theta\in \IR}\|e^{-i\theta} x_0-\hat x \|=\min_{\theta\in \IR}\|e^{i\theta} \hat x - x_0 \|. 
 \end{equation}

Throughout the paper, we assume 
the  canonical  embedding 
\[
\IC^n\subseteq \IC^{\tilde n}\subseteq \IC^N,\quad n\leq \tilde n\leq N.
\]
For example,  if $x\in \IC^n$, then the embedded vector in $\IC^{\tilde n}$ or $\IC^N$, still denoted by $x$,
has zero components $x(j)=0$ for $j\geq n+1$. This is  referred to as {\em zero padding} and
$\tilde n/n$ is the {\em padding ratio}.
Conversely, if $x\in \IC^{\tilde n}$ or $ \IC^N$, then $[x]_n \in \IC^n$ denotes 
 the projected vector onto $\IC^n$. Clearly, $[x]_{\IC^n}=[x]_n$. 
 
The vector space $\IC^N=\IR^N\oplus_\IR i\IR^N$ is 
isomorphic to
$\IR^{2N}$ via the map 
\begin{equation}\label{51} G(v):=\left[
\begin{array}{c}
\Re(v)     \\
 \Im(v)  
\end{array}
\right],\quad \forall v \in \IC^{N}\end{equation} 
and endowed with the real inner product
\[
\langle u, v\rangle :=\Re(u^*v)=G(u)^\top G(v),\quad u,v\in \IC^N.
\]

With a slight abuse of notation, we will use $G(u)$
to denote the conversion of a complex-valued vector $u$ in $\IC^n, \IC^{\tilde n} $ or
$\IC^N$  to its real-valued version. \\

{\bf Phase factor:} Let $ y\odot y'$ and $y/y'$ be the component-wise multiplication and division between  two vectors $y,y'$, respectively. For any $y\in \IC^N$ define the phase vector $\om\in \IC^N$ with  $\om(j)=\z(j)/|\z(j)|$ where $|\z(j)|\neq 0$.
When $|\z(j)|=0$ the phase can be assigned arbitrarily and we set $\om(j)=1$  unless otherwise specified.

\section{Douglas-Rachford  algorithm}\label{sec:ITA}

Phase retrieval   can be formulated   as the following feasibility problem 
in the Fourier domain 
 \beq
 \label{feas}
\hbox{Find}\quad  \hat y\in  A^*\cX \cap \cY,\quad \cY:=  \{y\in \IC^N: |y|=b\}. 
 \eeq
 
 
Let $P_1$ be the projection onto $A^*\cX$ and $P_2$ the projection onto $\cY$:
\beq
\label{proj}
P_1 y=A^*[Ay]_\cX,\quad P_2 y=b\odot{y\over |y|}
\eeq
Then  DR \eqref{dr} becomes $y^{(k+1)}=S_{\rm f}(y^{(k)})$ with
\begin{eqnarray}
S_{\rm f}(y)&=&\z+A^*\left[A \lt(2b\odot \frac{\z}{|\z|}-\z\rt)\right]_\cX-b\odot \frac{\z}{|\z|}  \label{fdr}
\end{eqnarray}
which we call the {\em Fourier-domain DR}  (FDR) to contrast with the following
object domain version.

  \commentout{
To this end, let 
 $x=[\tilde A z]_n _{X}$. 
 The isometric property  $\tilde A\tilde A^*=I_{\tilde n\times \tilde n}$ yields 
\begin{equation}\label{21}
x=\hbox{\rm arg}\min_{x'\in  X} \|x'-\tilde A z\|^2=\hbox{\rm arg}\min_{x'\in  X} \|\tilde A^* x'- z\|^2.
\end{equation}
Since $ \tilde A^* x' =A^* x',\forall x'\in X,$ we have by \eqref{21} and the isometric property  $AA^*=I_n$
\begin{equation}x=\hbox{\rm arg}\min_{x'\in X} \|A^*x'- \z\|^2=\hbox{\rm arg}\min_{x\in X} \|x'- A \z\|^2= [A\z]_n_\cX. 
\end{equation}
Hence $[\tilde A \z]_n _{X}=[A\z]_n_\cX$. 
}

Let  $\tilde A^*=[A^*, A_\perp^*]\in \IC^{N, \tilde n}$ be a complex isometric  extension of $A^*$,
implying  that $A_\perp A^*_\perp=I, AA^*_\perp=A_\perp A^*=0$.  
Then the phase retrieval problem can be more generally  formulated  as $|\tilde A^* x|=b, x\in \cX$. Consider the feasibility problem
 \beq
 \label{feas2}
 \hbox{Find}\quad \hat x\in \cX\cap \tilde  \cX,\quad \tilde \cX := \Big\{x\in \IC^{\tilde n}: |\tilde A^*  x|=b\Big\}.
 \eeq
Let  $P_1$ be the projection onto $\cX$, i.e. $P_1 x=[x]_\cX$, and $P_2$ the projection onto $\tilde \cX$.
When $\tilde n=N$ (hence $\tilde A$ is unitary), 
\beq
\label{10} P_2 x=\tilde A\lt(b\odot {\tilde A^* x\over |\tilde A^* x|}\rt)
\eeq
and \eqref{fdr} is equivalent to  
 \begin{equation}\label{odr}
 S(x)=x+\left[\tilde A \left(2b\odot \frac{\tilde A^*x}{| \tilde A^*x|}\right)-x\right]_\cX-\tilde A\left(b\odot \frac{ \tilde A^*x}{|\tilde A^*x|}\right).  
 \end{equation} 
In this case, we have 
\beq
\label{2.2}
\tilde A^*S\tilde A=S_{\rm f},\quad\hbox{for}\quad  \tilde n=N.
\eeq

In the 1-pattern case with the standard oversampling  $N=\tilde n\approx 4n$, $\tilde A=A$ is unitary and \eqref{odr} is also known as the Hybrid-Input-Output (HIO)  algorithm (with the HIO parameter set to one) \,\cite{Fie82, BCL02}. 

For $\tilde n<N$ (as with two oversampled patterns $N\approx 8n$ with the standard padding
$\tilde n\approx 4n$), the precise form of $P_2$ is not known explicitly. For the purpose of contrasting  with \eqref{fdr} and for lack of a better term we shall call
\eqref{odr}  (with $\tilde n\leq N$) the generalized Object-domain Douglas-Rachford  algorithm (ODR for short).  The ODR family is an interpolation between the
HIO and FDR. 

While ODR depends explicitly on $\tilde n$, FDR is independent of $\tilde n$ in the sense that
 \begin{equation} \label{19} S_{\rm f}(y) =\z+\tilde A^*\left[\tilde A \lt(2b\odot \frac{\z}{|\z|}-\z\rt)\right]_{ \cX}-b\odot \frac{\z}{|\z|} \end{equation}
 since $ [\tilde A\z]_{ \cX}=[A\z]_\cX\in \IC^n$ and $  \tilde A^*[\tilde A \z]_{\cX}=A^* [A \z]_\cX.$
\commentout{%
In the 2-mask case, it is natural to separate the propagation matrix and the data  into
two components as $A=[A_1, A_2]$ and $b=[b_1^\top, b_2^\top]^\top$  corresponding to the two different masks. Note that $AA^*=A_1A_1^*+A_2A_2^*=I$ by assumption
and hence $A_1A_1^*=A_2A_2^*=I/{2}$. 
Again, the phase retrieval problem is equivalent to the feasibility problem
\beq
\hbox{Find}  \quad \hat y\in \cX_1\cap \cX_2,\quad
\cX_j=\big\{x\in \IC^{n}: |A_j^*x |=b_j\big\},\quad j=1,2.
\eeq
Let $P_j$ be the projection onto $\cX_j, j=1,2,$.
}
  
\section{Local convergence} \label{sec:proof}

For simplicity,  we shall analyze FDR \eqref{fdr} {\em without} the sector condition: 
\begin{equation} \label{DRnoSector}\; S_{\rm f}(\z):=\z+ A^*A \lt(2b\odot \frac{\z}{|\z|}-\z\rt)-b\odot \frac{\z}{|\z|}
\end{equation}
whereas ODR \eqref{odr} becomes
  \begin{equation}\label{odr2}
 S(x)=x+\left[\tilde A \left(2b\odot \frac{\tilde A^*x}{| \tilde A^*x|}\right)-x\right]_n-\tilde A\left(b\odot \frac{ \tilde A^*x}{|\tilde A^*x|}\right).
 \end{equation} 
 However, we make no assumption about the number of diffraction patterns which can well
 be one.
 

The main result is  local, geometric convergence of the FDR algorithm.

\begin{thm} \label{thm1} Let $x_0\in \IC^n$ and $A^*$ any isometric $N\times n$ matrix.   Suppose $N\geq 2n$ and 
\beq\label{lam2}
 \max_{u\in \IC^n\atop iu\perp x_0}  \|u\|^{-1} \|\Im(B^* u)\| <1,\quad B:=A\,\,\diag\lt\{{A^*x_0\over|A^* x_0|}\rt\}.
 \eeq
 
Let $y^{(k)}$ be an FDR sequence  and 
 $x^{(k)}:=Ay^{(k)}, k=1,2,3...$. If $x^{(1)}$ is sufficient close to $x_0$, then for some constant $ \gamma<1$ 
\begin{equation}\label{23}
\| \alpha^{(k)}x^{(k)}-x_0\|\le \gamma^{k-1} \| x^{(1)}-x_0\|, \quad\forall k, \end{equation}
where
 \begin{equation}\label{24'}
\alpha^{(k)}:=\hbox{\rm arg}\min_{\alpha} \{ \|\alpha x^{(k)}- x_0\|: |\alpha|=1, \alpha\in \IC\}.
\end{equation}\end{thm}\commentout{
 \begin{rmk}
 In view of \ref{19}, the error bound \eqref{23} remains valid for any sequence
 generated by \eqref{19} with initial condition sufficiently close to $x_0$. 
Here $A$, rather than $\tilde A$, should be in \eqref{23}. 
 \end{rmk}
 }

\begin{rmk}
In view of \eqref{2.2} and Theorem \ref{thm1}, the following error bound holds for ODR with $\tilde n=N$
\begin{equation*}
\| \alpha^{(k)}[x^{(k)}]_n-x_0\|\le \gamma^{k-1}  \| \alpha^{(1)} [x^{(1)}]_n-x_0\|, \end{equation*}
where  $x^{(k)}=\tilde A \z^{(k)}$ and
\beq
\label{71'}
\alpha^{(k)}=\hbox{\rm arg}\min_{\alpha\in \IC, |\alpha|=1} \|\alpha x^{(k)} - x_0\|. 
\eeq
\end{rmk}

\begin{rmk}  \label{rmk4.3}
 Theorem \ref{thm1} is about the algebraic structure of FDR and  does not assume  oversampled  diffraction patterns. For example, one oversampled 
diffraction pattern ($N\approx 4n$) or two unoversampled diffraction patterns ($N=2n$) are  sufficient. 

However, as shown in Theorem \ref{cor5.2}, the proof of 
\eqref{lam2} requires one (and only one) oversampled coded  diffraction pattern. 

\end{rmk}

\begin{rmk}\label{rmk:generic} 
When the propagation matrix $A^*$ is not isometric, we apply QR-decomposition
to obtain $A^*=QR$, where $Q$ is isometric, and treat $Q$ as the new propagation
matrix and $Rx_0$ as the unknown. 

\end{rmk}

To this end, we derive and analyze the  local approximation of FDR as follows.

\subsection{Local analysis}

First note that
  \begin{equation}\langle A  (\alpha^{(k)} \z^{(k)}-\z_0),x_0\rangle = \langle  \alpha^{(k)} \z^{(k)}-\z_0 ,\z_0\rangle= \langle v^{(k)}, |\z_0|\rangle \end{equation}
  with
  \[
  v^{(k)}= \Om_0^*(\alpha^{(k)}\z^{(k)}-\z_0),\quad\Om_0=\hbox{diag}(\om_0),\quad \om_0={y_0\over |y_0|}.
  \]

This motivates the following analysis of the Jacobian operator $\sloc$.

 \begin{prop}\label{Linear} Let $\z\in\IC^N$, 
 $
  \om=\z/|\z|$ and 
  $\Om=\diag(\om)$.  
  
Then  \begin{equation} 
S_{\rm f}(\z+\ep \eta)-S_{\rm f}(\z)=\ep \Om \sloc v+o(\ep)\end{equation}  
where 
\begin{equation} \sloc v=  (I-B^* B) v + i (2 B^* B-I)\; \diag\lt[{b\over |\z|}\rt]\Im(v),
\label{Sloc1}
 \end{equation}
 with 
 \beq
 \label{42}
B=A \; \Om,\quad v=\Om^*\eta . 
\eeq
In particular,  if  $|\z|=b$, then  \eqref{Sloc1} becomes  
 \begin{equation} \sloc v=  (I-B^* B)  \Re(v) +  i B^* B \Im(v)\label{49}
 \end{equation}
  \end{prop}
\begin{proof} Let 
\[
\om_\ep={\z+\ep \eta\over |\z+\ep\eta|},\quad \Om_\ep=\diag(\om_\ep). 
\]
Reorganizing \eqref{DRnoSector}, we have 
\begin{equation}
S_{\rm f}(\z)=\z-A^*A \z+(2A^* A-I)\Om b.
\end{equation}
and hence
\begin{eqnarray} 
\nn S_{\rm f}(\z+\ep\eta)-S_{\rm f}(\z)&=&\ep (I-A^* A)\eta  +(2A^*A-I)(\Om_\ep-\Om)b\\
&=&\ep (I-\Om B^*B\Om^*)\eta +(2 \Om B^* B \Om^*-I) (\Om_\ep-\Om)  b\label{44}
\eeq
We next give a first order approximation
to $(\Om_\ep -\Om)b$ in terms of $\eta$.

\commentout{
Decomposing  $\ep\eta$  into the (nearly) radial and angular components
\[
\ep\eta=\Om_\ep (|\z+\ep\eta|-|\z|)+(\Om_\ep-\Om) |\z|
\]
and extracting the angular component by forming the real inner product with ${i\om}$,
we obtain
\beq
\label{45}
\Re\lt(-i \overline{\om}\odot {\eta\over |\z|} \rt)&=&
\Im\lt(\Om^*{\eta\over |\z|} \rt)=\ep^{-1} \Im\lt(\Om^*(\om_\ep -\om)\rt)
+o(1). 
\eeq
}
Using   the first order Taylor expansion we have
\beq
\nn \om_\ep-\om&=&i\Om\Im\Big[\Om^*(\om_\ep-\om)\Big]+o(\ep)= {i\ep}\Om\Im\lt[\Om^*{\eta \over |\z|}\rt]+o(\ep)\nn,
\eeq
 and hence 
\beq
(\Om_\ep-\Om)b&=&i\ep \Om\,\diag\lt[ {b\over |\z|}\rt] \Im(\Om^*\eta )+o(\ep). \label{43}
\eeq
Finally, substituting \eqref{43} into  \eqref{44} we obtain 
\beqn
S_{\rm f}(\z+\ep\eta)-S_{\rm f}(\z)&=&
\ep (I-\Om B^*B\Om^*)\eta +i\ep (2 \Om B^* B-\Om)\,\diag( {b}/{|\z|}) \Im(\Om^*\eta)+o(\ep).
\end{eqnarray*}
Multiplying $\Om^*$  on both sides and using the definition of  $v$ we
 complete the proof.
\end{proof}
Note that $\sloc$ is a {\em real}, but {\em not} {\em complex},  linear map since
$\sloc (cv)\neq c\sloc v,  c\in \IC$ in general. 

 Define the real form of the matrix $B$: 
 \begin{equation} \label{Bv}
  \mathcal{B}:=  
\left[
\begin{matrix}
 \Re[B]\\
  \Im[B] 
\end{matrix}
\right]\in \IR^{2n,N}.
\end{equation}
Note that 
\[
\lt[\begin{matrix}
\Re[B^\top]& \Im[B^\top]\\
-\Im[B^\top]& \Re[B^\top]
\end{matrix}\rt]
\]
is real isometric because $B^*$ is complex isometric.

From \eqref{51} we  have
\beq\label{53}
G(B^*u)=\lt[\begin{matrix}
\cB^\top G(u)\\
           \cB^\top G(-i u)
           \end{matrix}\rt],\quad u\in \IC^n. 
           \eeq

For the rest of the paper, $B$ denotes the matrix \eqref{42} with $\Omega=\Omega_0$, i.e. 
\beq
B=A\Omega_0\label{B},\quad\Omega_0=\hbox{diag}[\omega_0],\quad \omega_0={\z_0\over |\z_0|}
\eeq
unless otherwise specified. 

\subsection{Eigen structure}

 Let $\lambda_1\ge \lambda_2\ge \ldots\ge \lambda_{2n}\ge \lambda_{2n+1}=...=\lambda_{N}=0$ be the singular values of $\mathcal{B}$ with the corresponding right singular vectors
$\{\bv_k\in \IR^{N}\}_{k=1}^{N} $ and left singular vectors $\{{\bu}_k\in \IR^{2n}\}_{k=1}^{2n}$.
By definition, for  $k=1,...,2n$, 
\beq
\label{53'}
B \bv_k&=&\lambda_k G^{-1} (\bu_k),\\
\Re[B^* G^{-1}( \bu_k)]&=& \lambda_k \bv_k. 
\eeq 
\begin{prop}\label{B*bound} 
We have $\bu_{1}=G(x_0)$,  $\bu_{2n}=G(-ix_0)$, $\lambda_1=1, \lambda_{2n}=0$ as well
as 
  $\bv_{1}= |\z_0|$.
\end{prop}
\begin{proof} Since
 \beq\nn
&& B^* x_0=\Om_0^* A^* x_0=|\z_0|
\eeq
we have by  \eqref{53}
\beq
& &\Re[B^* x_0]=\mathcal{B}^\top \xi_{1}=|\z_0|, \quad \Im[B^*x_0]=\mathcal{B}^\top \xi_{2n}=0  \label{56}\eeq
 and hence the results. 
\end{proof}

\begin{cor}\label{cor5.5}
 \beq 
\label{63} \lambda_2= \max_{u\in \IC^n\atop iu\perp x_0}  \|u\|^{-1} \|\Im[B^* u]\|=\max_{u\in \IR^{2n}\atop  u\perp \xi_1} \|u\|^{-1} \|\cB^\top u\| 
 \eeq
\end{cor}
\begin{proof}
By \eqref{53}, 
\[
\Im[B^* u]=\cB^\top G(-i u).
\]
The orthogonality condition $iu\perp x_0$ is equivalent to
\[
G(x_0)\perp G(-iu).
\]
Hence, by 
 Proposition \ref{B*bound} $\bu_2 $ is the maximizer of the right hand side of \eqref{63}, yielding the desired value $\lambda_2$. 

\end{proof}

 \begin{prop}\label{Buu}   For $k=1,\ldots, n$, 
  \beq\label{57}
  \lambda_k^2+\lambda_{2n+1-k}^2=1
  \eeq
\beq
\label{58}
\bu_{2n+1-k}&=&G( -i G^{-1}(\bu_k) )\\
\bu_{k}&=&G(i G^{-1}(\bu_{2n+1-k}) ).\label{59}
\eeq
\end{prop}
\begin{proof}
Since $B^*$ is an isometry, we have $\|w\|=\|B^* w\|,\forall w\in \IC^n$. On the other hand, 
we have
\commentout{
\beq \label{Key}
\|B^*u\|^2=\|\Re(B^* u)\|^2+\|\Im[B^* u]\|^2=\|\Re[B^* u]\|^2+\|\Re[i B^*u]\|^2
\eeq
and 
}
\beqn
\|B^*w\|^2=\|G(B^*w)\|^2=\|\mathcal{B}^\top G(w)\|^2+\|\mathcal{B}^\top G(-iw)\|^2
\eeqn
and hence
\beq
\|G(w)\|^2=\|\mathcal{B}^\top G(w)\|^2+\|\mathcal{B}^\top G(-iw)\|^2. \label{Key}
\eeq

Now we prove \eqref{57}, \eqref{58} and \eqref{59} by induction. 

Recall the variational characterization of the singular values/vectors
\beq \label{60'}
\lambda_j=\max\| \mathcal{B}^\top {u}\|,& \bu_j=\hbox{\rm arg}\max
 \| \mathcal{B}^\top  {u}\|, & \hbox{s.t.}\,\, {u}\perp {\bu}_1,\ldots,  {\bu}_{j-1},\quad \|u\|=1\\
 \lambda_{2n+1-j}=\min\| \mathcal{B}^\top {u}\|,& \bu_{2n+1-j}=\hbox{\rm arg}\min
 \| \mathcal{B}^\top  {u}\|, & \hbox{s.t.}\,\, {u}\perp {\bu}_{2n},\ldots,  {\bu}_{2n+2-j}, \|u\|=1.\label{61'}
\eeq
By Proposition \ref{B*bound}, \eqref{57}, \eqref{58} and \eqref{59} hold for $k=1$. 
Suppose  \eqref{57}, \eqref{58} and \eqref{59} hold for $k=1,...,j-1$ and 
we now show that they also hold for $k=j$. 

 Hence by \eqref{Key}
 \[
 \lambda^2_j=\max_{\|u\|=1}\| \mathcal{B}^\top {u}\|^2=1-\min_{\|v\|=1}\| \mathcal{B}^\top {v}\|^2,\quad  \hbox{s.t.}\,\, {u}\perp {\bu}_1,\ldots,  {\bu}_{j-1},\quad v=G(-iG^{-1}(u)).
 \]
 The condition $ {u}\perp {\bu}_1,\ldots,  {\bu}_{j-1}$ implies $v\perp {\bu}_{2n},\ldots,  {\bu}_{2n+2-j}$ and vice versa. By \eqref{61'}, we have  $\lambda_j^2=1-\lambda_{2n+1-j}^2$ and  $G(-iG^{-1}(\bu_j))$ is the minimizer, i.e. $\bu_{2n+1-j}=G(-iG^{-1}(\bu_j)).$ 

\end{proof}

The relation \eqref{57} suggests the following parametrization  
of singular values of $\cB$:
 \[ \lambda_{2n+1-k}:=\cos \theta_k, 
\quad \lambda_k:=\sin \theta_k,\quad \theta_k\in [0,2\pi].\]

\begin{prop}\label{Srate}
For each $k=1,\ldots, n$, 
\begin{eqnarray}\label{B5}
&&B^* B \bv_k=\lambda_k(\lambda_k \bv_k+i\lambda_{2n+1-k} \bv_{2n+1-k}),\\
&&B^* B \bv_{2n+1-k}=\lambda_{2n+1-k}(\lambda_{2n+1-k} \bv_{2n+1-k}-i\lambda_k \bv_k).\label{B51}
\end{eqnarray}

\end{prop}

\begin{proof} By definition, $\mathcal{B}\bv_k= \lambda_k {\bu}_k.$ Hence 
\[B\bv_k=(\Re[B]+i\Im[B]) \bv_k=\lambda_k (\bu_k^{\rm R}+i\bu_k^{\rm I})
\]
where 
\[
\bu_k=\lt[\begin{matrix} \bu_k^{\rm R}\\
\bu_k^{\rm I}
\end{matrix}
\rt],\quad \bu_k^{\rm R},\xi_k^{\rm I}\in \IR^n.
\]

On the other hand, $\cB^\top \bu_k=\lambda_k \bv_k$ and hence 
\beq
\label{59'}
\Re[B^\top]\bu_k^{\rm R}+\Im[B^\top] \bu_k^{\rm I}=\lambda_k \bv_k. 
\eeq

Now we compute $B^*B \bv_k$ as follows.
\beq
B^* B \bv_k&=& \lambda_k B^*(\bu_k^{\rm R}+i\bu_k^{\rm I})\label{60}\\
&=& \lambda_k (\Re[B^\top]-i\Im[B^\top])(\bu_k^{\rm R}+i\bu_k^{\rm I})\nn\\
&=&\lambda_k( \Re[B^\top]\bu_k^{\rm R}+\Im[B^\top]\bu_k^{\rm I})+i\lambda_k (\Re[B^\top]\bu_k^{\rm I}-\Im[B^\top] \bu_k^{\rm R})\nn\\
&=&\lambda_k^2\bv_k+i\lambda_k (\Re[B^\top]\bu_k^{\rm I}-\Im[B^\top] \bu_k^{\rm R})\nn
\eeq
by \eqref{59'}.

Notice that 
\beq
\nn\Re(B^\top)\bu_k^{\rm I}-\Im(B^\top) \bu_k^{\rm R}&=&\cB^\top \lt[
\begin{matrix}
\Re(-i G^{-1}(\bu_k))\\
\Im(-iG^{-1}(\bu_k))\end{matrix}\rt]\\
&=&\cB^\top G(-i G^{-1}(\bu_k))\nn\\
&=&\cB^\top \bu_{2n+1-k}\nn\\
&=&\lambda_{2n+1-k} \bv_{2n+1-k} \label{61}
\eeq
by Proposition \ref{Buu}. 

Putting \eqref{60} and \eqref{61} together, we have
\eqref{B5}. \eqref{B51} follows from the similar calculation. 
\end{proof}

\begin{cor}\label{Srate2}
For $k=1,2,\ldots, 2n$, $\sloc$ leaves invariant the subspace  spanned by $\{\bv_k, i\bv_{2n+1-k}\}$ and has the matrix form 
\begin{equation} \label{Sloc2}\sloc=\lambda_{2n+1-k} \left[
\begin{array}{cc}
\cos \theta_k  & \sin \theta_k\\
 -\sin \theta_k  & \cos \theta_k
\end{array}
\right],\quad \lambda_{2n+1-k}:=\cos \theta_k, 
\quad \lambda_k:=\sin \theta_k
\end{equation} 
   in the basis of $\{\bv_k, i \bv_{2n+1-k}\}$.  In particular, 
   \beq
   \label{69}
   \sloc \bv_1&=&0,\quad \sloc i\bv_{1}= i\bv_1\\
   \sloc \bv_{2n}&=& \bv_{2n},\quad \sloc i\bv_{2n}=0.
   \eeq
   where $ \bv_1=|\z_0|$. 
\end{cor}
\begin{proof}

By Proposition \ref{Srate},  the span of $\bv_k$ and $i\bv_{2n+1-k}$ is invariant under $B^*B$ and hence under $\sloc$  for $k=1,...,2n.$ Moreover, \eqref{B5} and \eqref{B51} imply  
\beqn
B^*B&=&\lt[ \begin{matrix}
\lambda_k^2& \lambda_k\lambda_{2n+1-k}\\
\lambda_k\lambda_{2n+1-k}&\lambda_{2n+1-k}^2
\end{matrix}\rt]
\eeqn
in  the basis of $\bv_k, i \bv_{2n+1-k}$. Hence by the definition \eqref{49} and Proposition \ref{Buu}, 
\begin{equation*} \sloc=\lambda_{2n+1-k}  
\left[
\begin{array}{cc}
 \lambda_{2n+1-k} & \lambda_k\\
  -\lambda_k  & \lambda_{2n+1-k}  
\end{array}
\right]
=\lambda_{2n+1-k} \left[
\begin{array}{cc}
\cos \theta_k  & \sin \theta_k\\
 -\sin \theta_k  & \cos \theta_k
\end{array}
\right],\quad \theta_k\in \IR.
\end{equation*} 
  Hence $\lambda_{2n+1-k}(\lambda_{2n+1-k}\pm i \lambda_k)$ are eigenvalues of $\sloc$. 

\end{proof}

In the next two propositions, we give a complete characterization of the
eigenstructure of $\sloc$.
 \begin{prop}\label{Seig}  If $v^*\bv_k=0, k=1,2,...,2n-1$, then  
\beq
Bv=0,\quad \sloc v=\Re(v).\label{70}
\eeq
  \end{prop}
 \commentout{
 \begin{rmk}\label{rmk1}
 By simple dimension count, the subspace 
 \[
 \{v\in \IR^N: Bv=0\}
 \]
 has  $N-n$ (real) dimension.
 \end{rmk}
 }
  \begin{proof}
The condition  $v^*\bv_k=0$ is equivalent to  $\bv_k^\top\Re(v)=\bv_k^\top\Im(v)=0$.
So  we have
 \begin{equation*} 
G(B\Re(v))=\left[
\begin{array}{c} \Re(B\Re(v)) \\
 \Im(B\Re(v)) 
 \end{array}\right]=
\left[
\begin{array}{c} \Re(B)\Re(v) \\
 \Im(B)\Re(v) 
 \end{array}\right]=
  \mathcal{B}\Re(v)=0
 \end{equation*} 
 implying $B\Re(v)=0$. 
  Likewise, $B\Im(v)=0$. So, we have $Bv=0$. 
  
 By the definition of $\sloc$ and $Bv=0$, 
 \[
 \sloc v=(I-B^*B)\Re(v)+iB^*B \Im(v)=\Re(v).
 \]

\end{proof}

\begin{cor}\label{cor5.8}
 The fixed point set of $\sloc$ contains 
the subspace
\[
 E_1=\hbox{\rm null}_\IR (\cB)\subset \IR^N
 \]
 and the null space of $\sloc$ contains the subspace
 \[
E_0= iE_1. 
 \]
 \commentout{
 as well as the $0$-real-dimensional subspace
  \[
E_3 =\{ v=B^*  u: \Im(B^* u)=0\}.
 \]
 }
 Moreover, if $\lambda_2<1$, then 
 \[
E_2^\perp=E_0\oplus_\IR  E_1
\]  where
\[
 E_2=\hbox{\rm span}_\IR\{\eta_k, i\eta_{k}\}_{k=1}^{2n-1}.
 \]
\end{cor}
\begin{proof}
Note that $\eta_{2n} $ and $i\eta_{2n}$ are excluded from $E_2$ 
because $\eta_{2n}\in E_1, i\eta_{2n}\in E_0$.
On the other hand the null vector $\eta_1$ does not belong in $E_0$
and the eigenvector $i\eta_1$ for eigenvalue 1 does not belong in $E_1$
for an obvious reason. 

For any $v\in \IC^N$, we can write $v=\Re(v)+i\Im(v)$. 
By Proposition \ref{Seig}, if $\Re(v), \Im(v) \in E_2^\perp$, then 
\beqn
\cB\Re(v)&=&0,\quad \sloc(\Re(v))=\Re(v)\\
\cB\Im(v)&=& 0,\quad \sloc (\Im(v))=0. 
\eeqn
In other words, $\Re(v)\in E_1$ and $\Im(v)\in E_0$. 

On the other hand, if $\lambda_2<1$, then $\lambda_{2n-1}>0$ and $E_2$ has no nontrivial intersection with
either $E_0$ or $E_1$. Hence, $
E_2^\perp=E_0\oplus_\IR  E_1.
$

\end{proof}

\subsection{Proof of Theorem \ref{thm1}}

With $ v^{(k)}=\Om_0^* (\alpha^{(k)}\z^{(k)}-\z_0)$,  we can express   FDR as
\begin{equation}\label{viter}
v^{(k+1)}=\sloc v^{(k)} +o(\|v^{(k)}\|).
\end{equation}
by  Proposition ~\ref{Linear}. Moreover,\beq
\label{70'}
B \sloc v &= &(I-BB^*) B\Re(v)+iBB^* B\Im(v)\\
&=&iB\Im(v),\quad \forall v\in \IC^N,\nn
\eeq
 by \eqref{49} and the isometry property $BB^*=I$.

At the optimal phase $\alpha^{(k)}$ adjustment for $\z^{(k)}$,  we have  
\[
\Im(\z_0^*\alpha^{(k)} \z^{(k)})=0
\]
and hence  \begin{equation} \langle v^{(k)}, i |\z_0|\rangle = \langle  \alpha^{(k)} \z^{(k)}-\z_0 ,i  \z_0\rangle=\langle  \alpha^{(k)} \z^{(k)}, i  \z_0\rangle =\Re(({\alpha^{(k)} \z^{(k)} })^*i \z_0)=0\end{equation}
which implies  that  $\Im(v^{(k)})$ is orthogonal to the leading right singular vector $\bv_1=|\z_0|$ of $\cB$:
\beq
\bv_1^T\Im(v^{(k)})=0\label{74}
\eeq
cf. Proposition \ref{B*bound}.

By \eqref{70'} and \eqref{74}, 
\beqn
\|B \sloc v^{(k)}\|&=&\|B\Im(v^{(k)})\|= \|\cB \Im(v^{(k)})\|\leq\lambda_2\|\Im(v^{(k)})\|.\\
\eeqn
By induction for $k=1,2,\ldots$ 
\begin{equation}\| A(\alpha^{(k+1)} \z^{(k+1)}-\z_0)\|=
 \|B v^{(k+1)}\|\le \gamma^{k}\|v^{(1)}\|, 
\end{equation}  
for some $\gamma\in [\lambda_2, 1)$, if $\|v^{(1)}\|$ is sufficiently small. 

Therefore for  $x^{(k)}=A\z^{(k)}$,
\beq
\alpha^{(k)}x^{(k)}-x_0&=&B\lt(\Om_0^* \alpha^{(k)} \z^{(k)}-\Om_0^* \z_0\rt)=B v^{(k)},\quad k=1,2,...
\eeq
which tends to zero geometrically with a rate at most $\gamma<1$.

 \section{Spectral gap}\label{sec:gap}
 
In this section, we  prove the spectral gap condition
 \eqref{lam2} with at least {\em one}  {\em oversampled  coded} diffraction pattern.
This is the immediate consequence of the following  two results.

\begin{prop}\label{prop5.4} Let $A^*$ be isometric and $B=A\Om_0$.  Then  $\|\Im(B^*x)\|=1$ holds for some unit vector $x $ if and only if 
\beq\label{95}
\Re(a_j^* x)\Re(a_j^*x_0)+\Im( a_j^*x) \Im( a_j^*x_0)=0, \; \forall j=1,... N,
\eeq
where $a_j$ are the columns of $A$, 
or equivalently
\beq\label{95'}
{A^*x\over |A^*x|}= \sigma\odot \om_0 
\eeq
where the components of $\sigma$ are either 1 or -1, i.e.
\[
 \sigma(j)\in \{1, -1\}, \quad \forall j=1,... N.
 \]
 \end{prop}
\begin{proof} We have 
\beq
\nn\Im(B^*x)
&=&\Im \lt({\overline{A^* x_0}\over |A^* x_0|} \odot A^*x\rt)\\
&=&\sum_{j=1}^N {\Re(a_j^* x_0) \Im( a_j^* x)-\Im( a_j^*x_0) \Re(a_j^* x)\over
(\Re^2(a_j^*x_0)+\Im^2( a_j^*x_0))^{1/2}}\label{57'}
\eeq
and hence 
\beq\nn
\|\Im(B^* x)\|^2
&\le &
 \sum_{j=1}^N  \Re^2( a_j^* x)+ \Im^2(a_j^*x)=\sum_{j=1}^N |a_j^*x|^2=\|A^* x\|^2=\|x\|^2
\eeq
by the Cauchy-Schwartz inequality and isometry.

In view of \eqref{57'}, the inequality becomes an equality if and only  if  \eqref{95} or \eqref{95'} holds. 
 \end{proof}
 
 \begin{prop}\label{thm3}  (Uniqueness of Fourier magnitude retrieval)
 Let $x_0$ be a given rank $\geq 2$ object  and $\mu$ be continuously and  independently
distributed on the unit circle. Let 
\[
B:=A\,\,\diag\lt\{\om_0\rt\},\quad \om_0:={A^*x_0\over|A^* x_0|}
\]
where $A^*$ is the given by \eqref{one}.
If 
\beq
\label{mag}
\measuredangle A^*\hat x=\pm \measuredangle A^* x_0
\eeq
 (after proper
 adjustment of the angles wherever the coded diffraction patterns vanish) where
 the $\pm$ sign may be  pixel-dependent, then almost surely $\hat x= c x_0$ for some constant $c\in \IR$. 
\end{prop}

The proof  of Proposition \ref{thm3} is  given in Appendix A.

Now we can prove the spectral gap theorem needed for geometric  convergence of FDR.

  \begin{thm}\label{cor5.2} Let $\Phi$ be the oversampled discrete Fourier transform. Let $x_0$ be a rank $\geq 2$ object and at least one of $\mu_j, j=1,...,\ell\geq 2$, be independently and continuously 
distributed on the unit circle. Let 
 \beq
 \label{1half}
A^*=c \lt[\begin{matrix}
\Phi\,\, \diag\{\mu_1\}\\
...\\
\Phi\,\, \diag\{\mu_\ell\}
\end{matrix}\right]
\eeq
 be  isometric.
Let
\[
B:=A\,\,\diag\lt\{\om_0\rt\},\quad \om_0:={A^*x_0\over|A^* x_0|}. 
\]

Then with probability one
\beq\label{iff}
\|\Im(B^* u)\|=1,\quad \|u\|=1\quad \hbox{iff}\quad 
u=\pm i x_0/\|x_0\| 
\eeq
and hence 
\beq\label{lam2'}
\lambda_2= \max_{u\in \IC^n\atop u\perp ix_0}  \|u\|^{-1} \|\Im(B^* u)\| <1.
 \eeq
 \end{thm}
\begin{proof}
Note that the proof of Proposition \ref{prop5.4} depends only on the fact that $A^*$ is isometric and hence holds for at least one coded diffraction pattern, oversampled or not.

Also, the uniqueness theorem, Proposition \ref{thm3},  clearly holds as long as there is at least one oversampled coded 
diffraction pattern.

Now Proposition \ref{prop5.4} says that  \eqref{iff} holds if \eqref{mag} has a unique solution  up to a real constant and Proposition \ref{thm3} says that \eqref{mag} indeed has a unique solution up to a real constant. The proof is complete. \end{proof}

We have the following corollary from Theorems \ref{thm1} and 
\ref{cor5.2}.

\begin{cor} Let the assumptions of Theorem \ref{cor5.2} be satisfied.
 Then the local, geometric convergence \eqref{23}-\eqref{24'}
holds for phase retrieval with  \eqref{1half} as the propagation matrix. 

\end{cor}

 \section{Uniqueness of fixed point}\label{sec:fixed}

Let $\z_\infty$ be a fixed point of FDR \eqref{fdr}, i.e. 
\[
S_{\rm f}(\z_\infty)=\z_\infty,\quad x_\infty=Ay_\infty. 
\]
Let $
\om_\infty= {\z_\infty/|\z_\infty|}
$
be the phase factor of the fixed point. 
Let 
\beq
\hat x&=&\left[A\left(2b\odot \om_\infty-\z_\infty\right)\right]_\cX=\left[2 A (b\odot \om_\infty)-x_\infty\right]_\cX,
\label{14}\eeq
where $\cX$ represents the sector condition in the 1-pattern case and $\cX=\IC^n$ in the 2-pattern case.  We have from \eqref{fdr}
\beq
A^* \hat x = b\odot \om_\infty
\eeq
which implies
\beq
|A^* \hat x|&=& |A^* x_0|\label{15}\\
\measuredangle A^* \hat x &=& \measuredangle \z_\infty.\label{16}
\eeq

Eq. \eqref{15} is related to phase retrieval and eq.  \eqref{16}  magnitude
retrieval problem.

Now we state the uniqueness for Fourier phase retrieval.
\begin{prop} \label{prop6.1} \cite{unique} (Uniqueness of Fourier phase retrieval) Let the
assumptions of Theorem \ref{cor5.2} hold.    Let $x$ be a solution of  of the  phase retrieval problem \eqref{0}.   \\

\noindent {\bf 1) One-mask case $\ell=1$.} Suppose, in addition, that $\measuredangle x_0(j) \in [-\alpha\pi,\beta\pi],\,\forall j$. Then   $x=e^{i\theta} x_0$ for some constant $\theta\in \IR$ with a high probability which has a simple, lower bound 
\beq
\label{prob}
1-n\lt|{\beta+\alpha\over 2}\rt|^{\llfloor S/2\rrfloor} 
\eeq
if  $\mu$ is  uniformly distributed on the unit circle, where $S$ is the sparsity of the image and $\llfloor S/2\rrfloor$ the greatest integer less than or equal to $S/2$. \\

\noindent  {\bf 2) Two-mask case $\ell=2$.} Suppose, in addition, that both masks are
independently and continuously distributed on the unit circle and independent from each other. 
 Then   $x=e^{i\theta} x_0$ for some constant $\theta\in \IR$ with  probability one. 
\end{prop}
The proof of Proposition \ref{prop6.1} is given in \cite{unique} where
more general uniqueness theorems can be found, including the $1\half$-mask case. 

\begin{thm}\label{thm4} (Uniqueness of fixed point)
Under the set-up of Proposition \ref{prop6.1}, the following statements hold for the phase retrieval problem \eqref{0}.\\

{\bf 1) One-mask case.} With probability at least given in  \eqref{prob}, $\hat x =e^{i\theta}x_0$ for some $\theta\in \IR$. \\

{\bf 2) Two-mask case.} Almost surely  $\hat x=x_\infty=e^{i\theta}x_0$ for some constant $\theta\in \IR$. 

\end{thm}
\begin{rmk}\label{rmk6.3}
With a slightly stronger assumption (\cite{unique}, Theorem 6),  the two-mask
case in Proposition \ref{prop6.1} and hence Theorem \ref{thm4} hold for the
$1\half$-mask case \eqref{two'}.
\end{rmk}

\begin{proof}
By Proposition \ref{prop6.1} (\ref{15}) implies that
$\hat x=e^{i\theta}x_0$ for some constant $\theta\in [0,2\pi]$, with
the only difference between case 1) and case 2) being the probability with which this statement holds. To complete the proof, we only need to consider case 2) and show $\hat x=x_\infty$.\\

By  \eqref{14}  $\hat x\in \cX$  and by Proposition \ref{prop6.1}  (\ref{15}) implies that with probability no less
than \eqref{prob}  $\hat x=e^{i\theta} x_0$ for some constant $\theta\in\IR$. Hence, by \eqref{16}, we have
\beq
\label{100}
e^{i\theta}\om_0=\om_\infty.
\eeq
Substituting \eqref{100} into \eqref{14} we obtain
\[
2e^{i\theta} x_0=\hat x+x_\infty=e^{i\theta} x_0+x_\infty
\]
and hence $e^{i\theta} x_0=x_\infty$. In other words, 
\[
x_\infty=\hat x=e^{i\theta}x_0. 
\]
\end{proof}

On the other hand, 
for case 1),   $\hat x=e^{i\theta}x_0$ and  \eqref{16} imply 
$e^{i\theta} \om_0=\om_\infty$ and hence
\[
e^{i\theta} y_0=e^{i\theta} b\odot \om_0=b\odot \om_\infty.
\]
 This and \eqref{14} only lead to 
 \[
 \hat x=[2e^{i\theta}x_0-x_\infty]_\cX=[2\hat x-x_\infty]_\cX.
 \] \\

\section{Numerical simulations} \label{sec:num}

\subsection{Test images}

 \begin{figure}[tbp]
\begin{center}
\includegraphics[width=5cm]{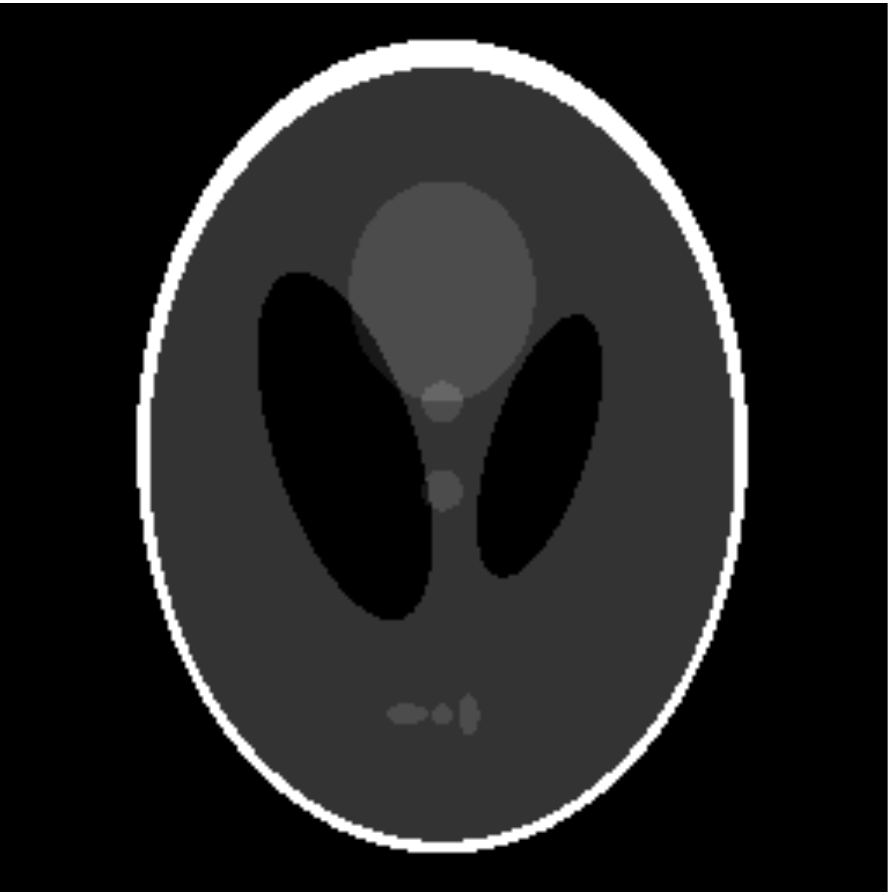}\quad
\includegraphics[width=5cm]{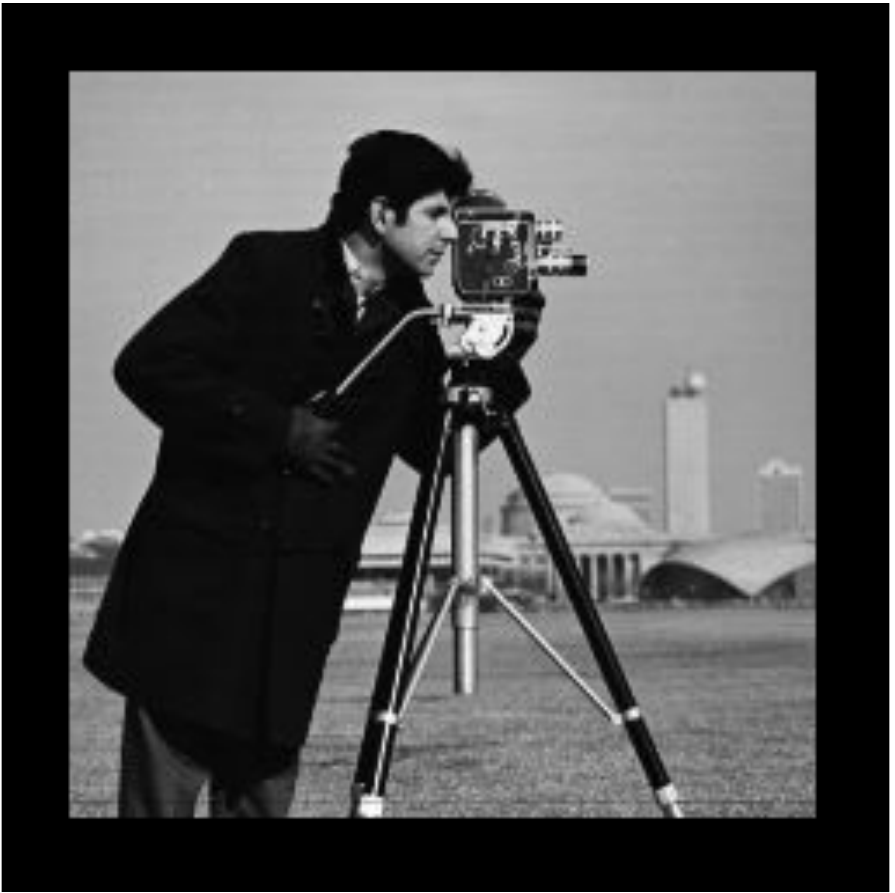}
\includegraphics[width=5cm]{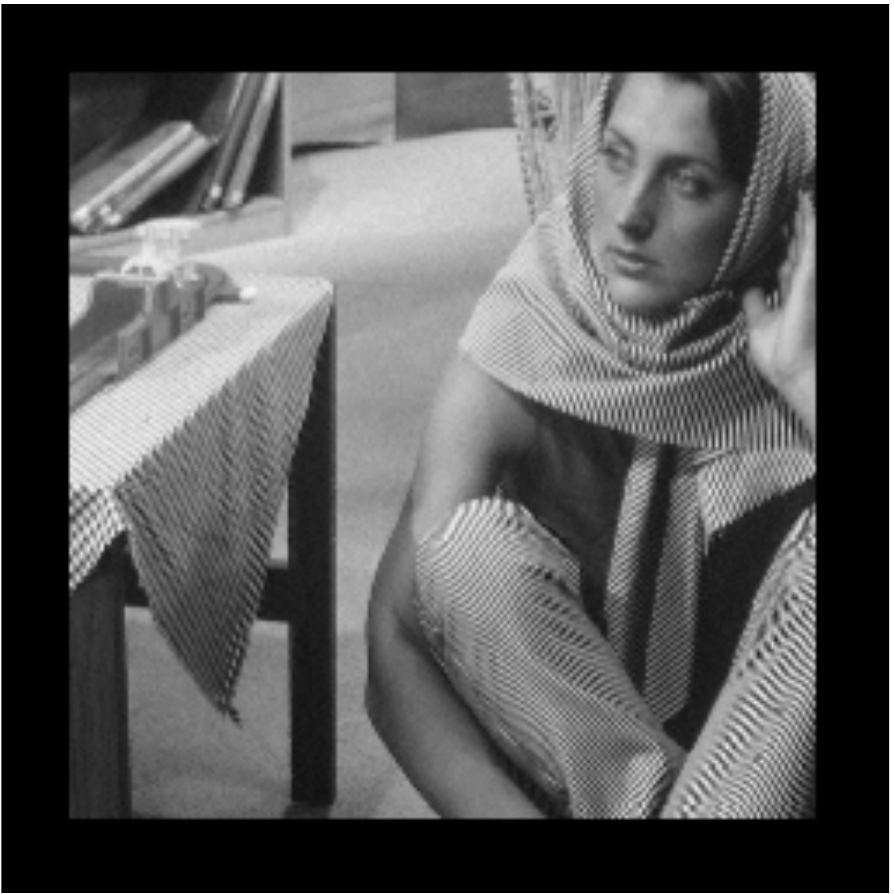}

\caption{The original phantom without phase randomization  (left), the truncated cameraman (middle) and the truncated Barbara (right).
}
\label{HIOimagesAC}
\label{fig2}
\end{center}
\end{figure}

For test images $x_0$ we consider the Randomly Phased Phantom  (RPP) Fig. \ref{fig2} (left) and the deterministic image,
hereby called the Truncated Cameraman-Barbara (TCB), whose real part is the truncated cameraman,  Fig. \ref{fig2} (middle) and whose imaginary part is the truncated Barbara, Fig. \ref{fig2} (right). The purpose of  truncation is to
create an unknown,  loose support (dark margins) which makes the image more difficult to recover. RPP has a loose support without additional  truncation. Likewise, we randomize
the original phantom in order to make its reconstruction more challenging. In general, a random object such as RPP is more difficult to recover than a deterministic object such as
TCB (see, e.g. Fig. \ref{fig:two} and \ref{fig:four}). The  size $n$ of both images is $256 \times 256$, including the margins. 

The propagation matrix is primarily based on either \eqref{one} or \eqref{two'} unless specified
otherwise. 

\subsection{Local convergence rate}
First we simulate the local convergence rate of the $1\half$-mask case  and compare them with $\lambda_2$.

 The initial condition $x^{(1)}$ is chosen sufficiently close to the true object $x_0$, which is a unit vector. Fig. \ref{fig:rate}  shows the error $\|\alpha^{(k)}[x^{(k)}]_n-x_0\|$  on the log scale versus  the iteration counter in the case of two oversampled diffraction patterns. 
The  oscillation in the blue curve (FDR) is due to  the complex eigenvalues of $\sloc$. 
 The magenta line shows the geometric sequence $\{\lambda_2^k\}_{k=1}^{100}$. The $\lambda_2$ value is computed via the power method, $\lambda_2=0.9505$ for TCB and $\lambda_2=0.9533$ for RPP. Note that the FDR curve (red) decays slightly faster
 than the $\lambda_2$-curve (magenta), which decays still faster than the black curve (ODR with $\tilde n\approx 4n$).

   \begin{figure}[tbp]
\begin{center}
\subfigure[TCB]{
\includegraphics[width=0.4\textwidth]{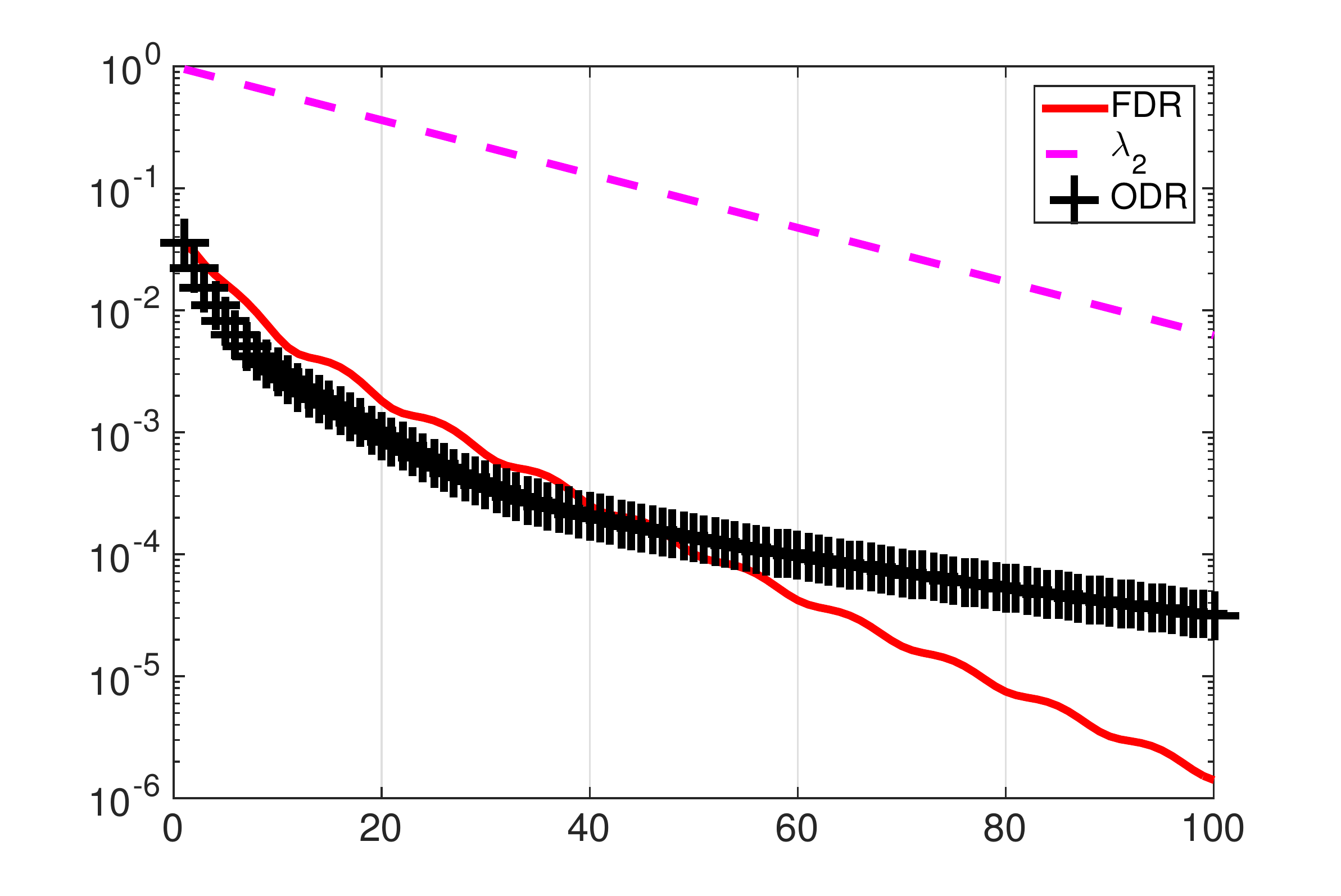}
}
\subfigure[RPP]{
\includegraphics[width=0.4\textwidth]{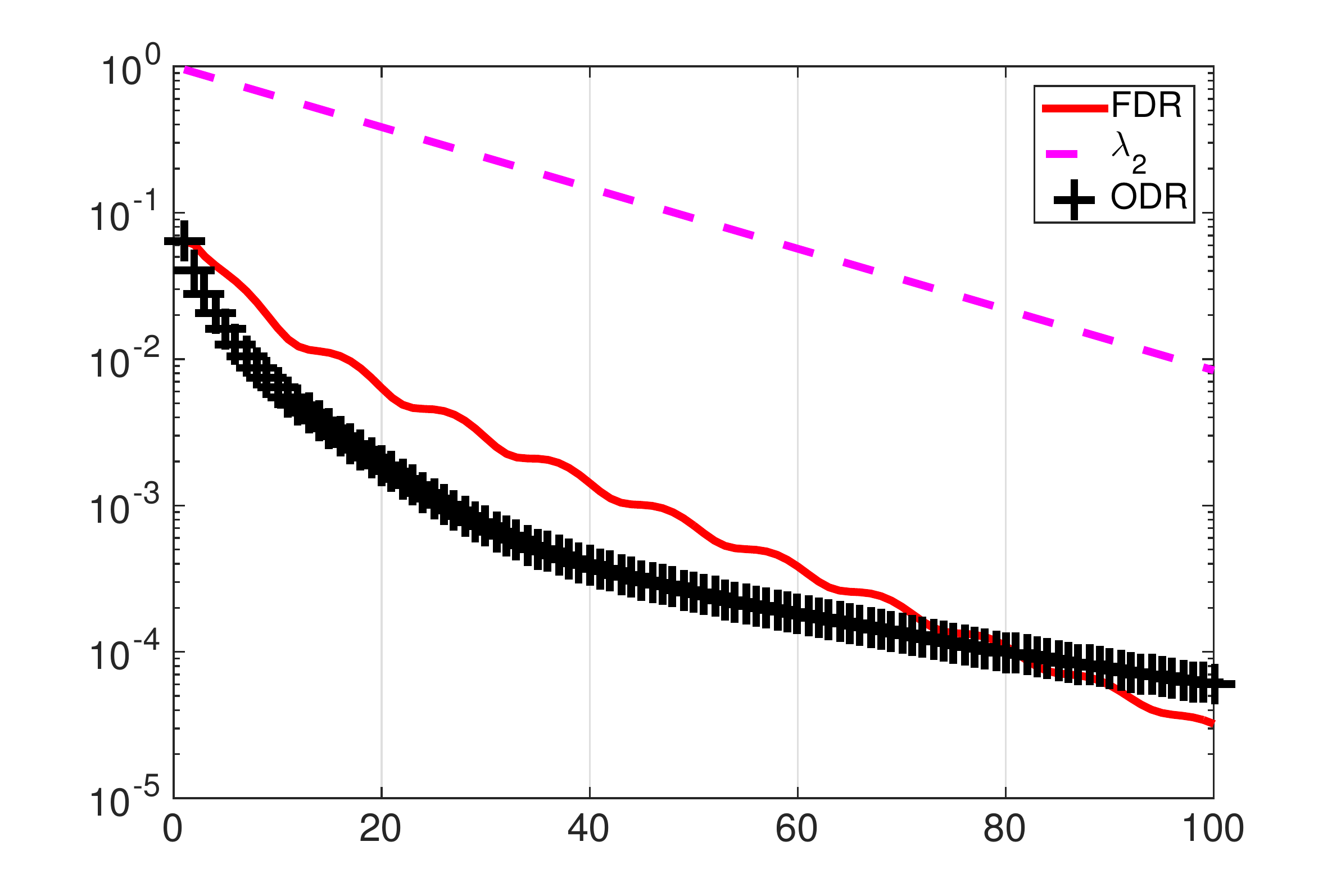}
}
\caption{Error (on the log-scale) versus iteration  for TCB (left) and RPP (right).   
}
\label{fig:rate}
\end{center}
\end{figure}

\commentout{  \begin{figure}[tbp]
\begin{center}
\includegraphics[width=5cm]{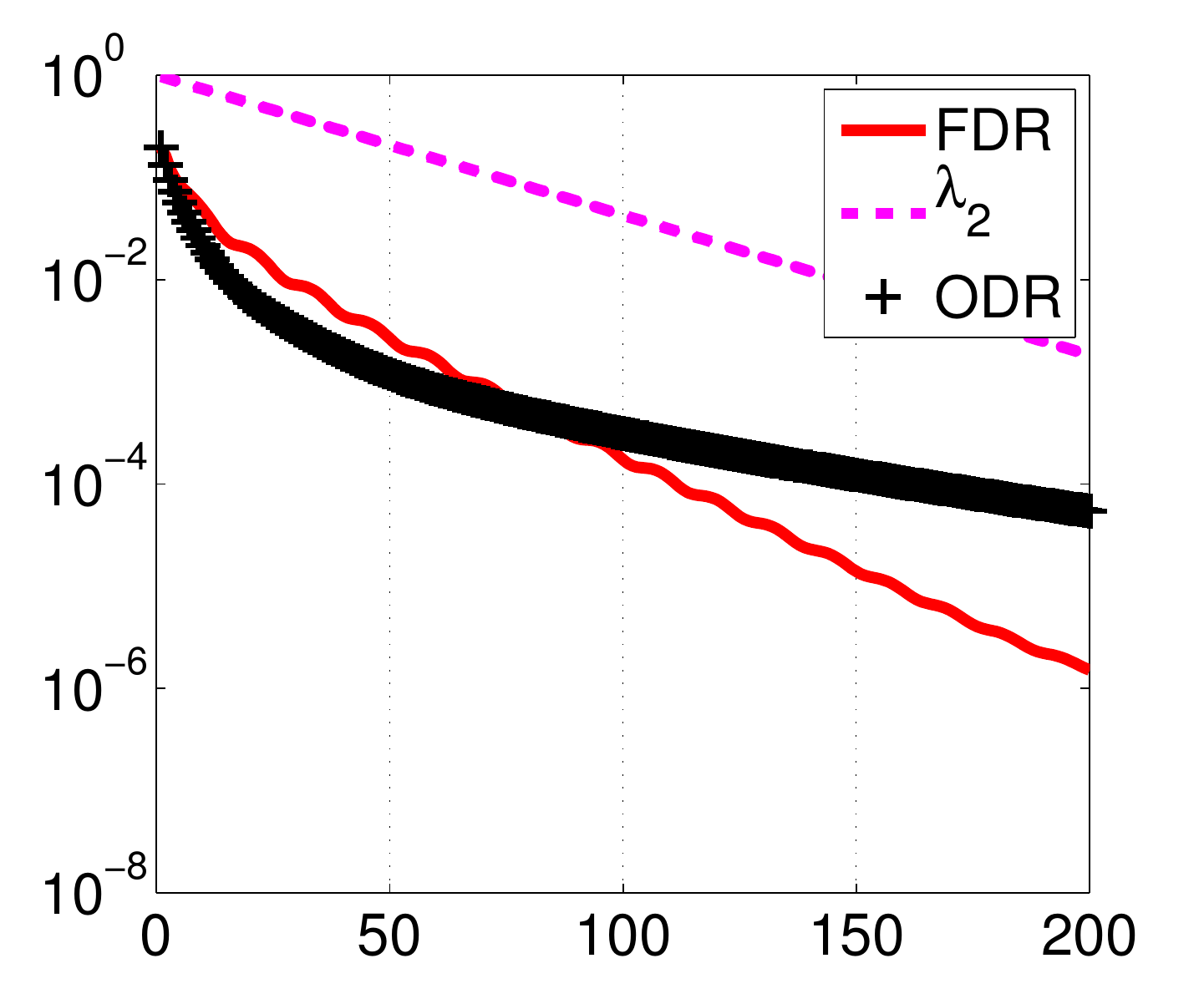}\quad 
\includegraphics[width=5cm]{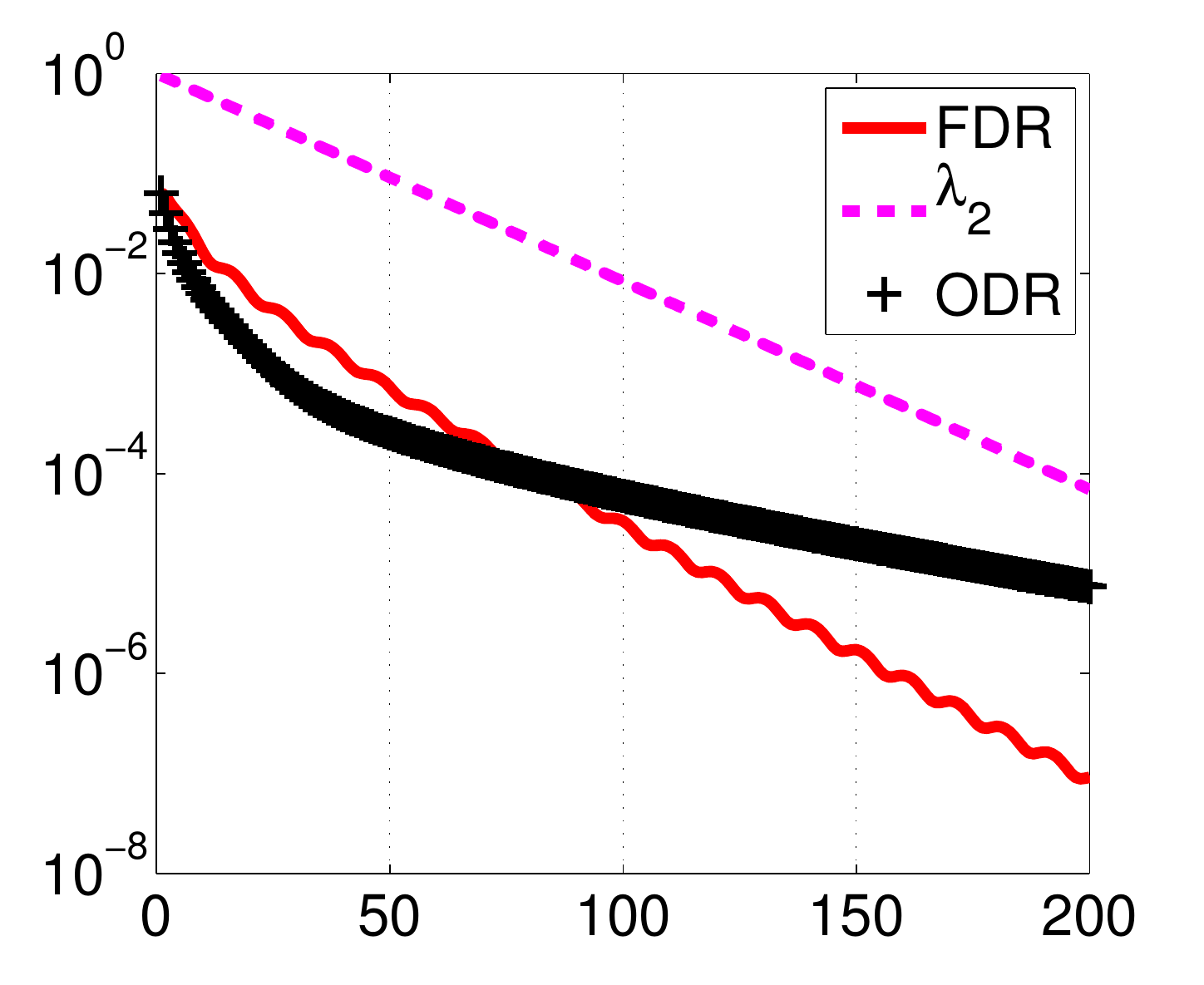}
\caption{
Error (on the log-scale) versus iteration  for TCB (left) and RPP (right). 
}
\label{fig:rate}
\end{center}
\end{figure}
}

\commentout{
\subsubsection{The Gaussian case}

Let $A^*\in \IC^{80\times 20}$ be an isometric matrix obtained by QR decomposition of a $80 \times 20$ Gaussian random  matrix.

 
The black curve shows the error for  ODR with $\tilde n=2n$ and the red line represents
the geometric series $\{\gamma^k\}_{k=1}^{200}$  where $\gamma=0.9724>\lambda_2=0.9545$.

  \begin{figure}
\begin{center}
\includegraphics[width=6cm]{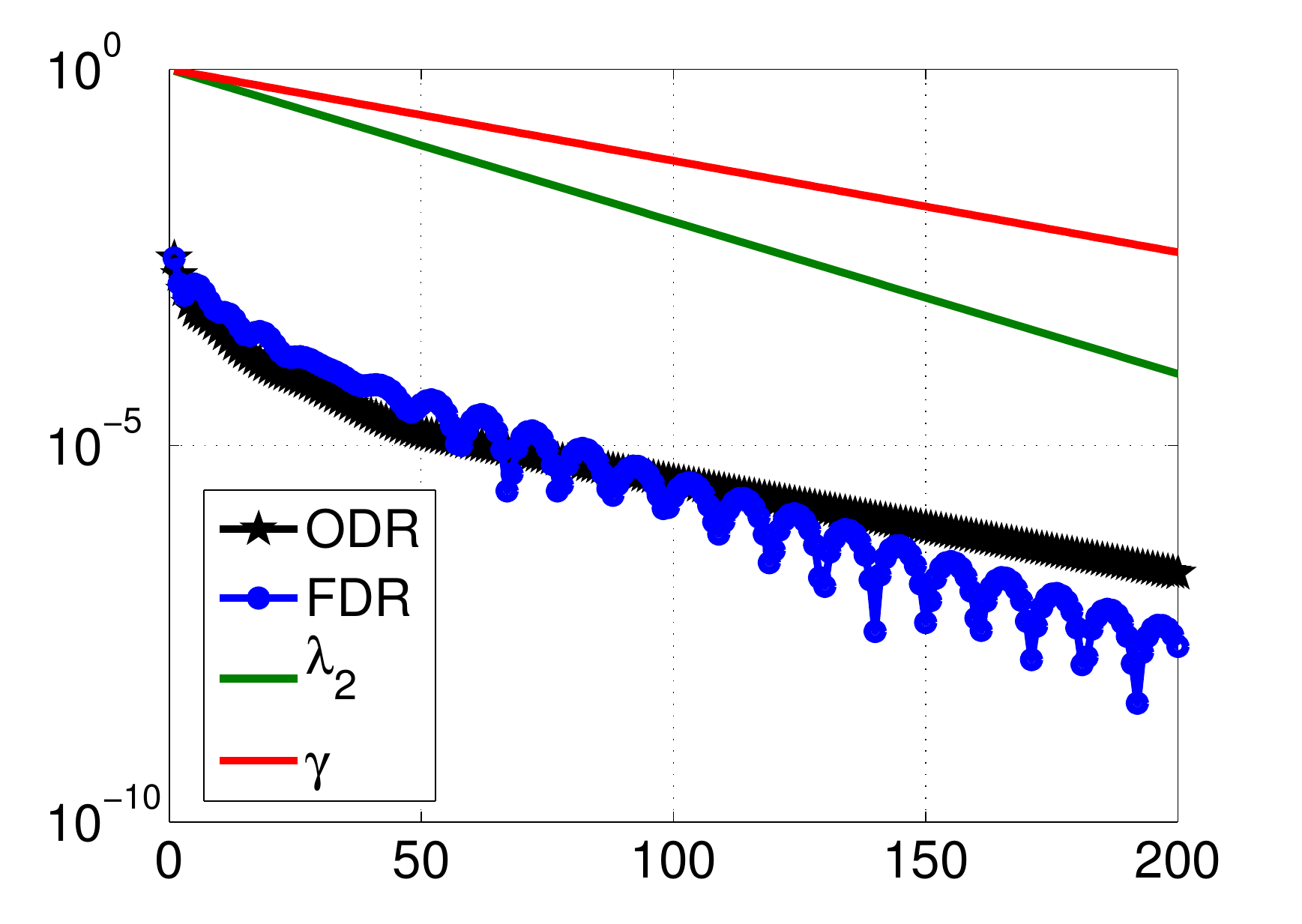}
\caption{
Convergence rates of FDR and ODR. 
}
\label{LinearRate}
\end{center}
\end{figure}
}

\subsection{Initiailization}
For {\em global } convergence behaviors, we test two different initializations: 
the Random Initialization (RI), where each pixel value is selected randomly and 
independently,   and the Constant Initialization (CI), where each pixel value is set to unity.

 The relative  error of the estimate $\hat x$ with
the optimal phase adjustment   
is given by 
\begin{equation}
\frac{\|e^{i\hat\theta} \hat x-x_0\|}{\|x_0\|},\quad \hat \theta=\hbox{\rm arg}\min_{\theta\in\IR} {\|e^{i\theta}  \hat x-x_0\|}.
\end{equation}

\subsection{One-pattern case}

 \begin{figure}[t]
\begin{center}
\subfigure[RPP + RI]
{
\includegraphics[width=4.2cm]{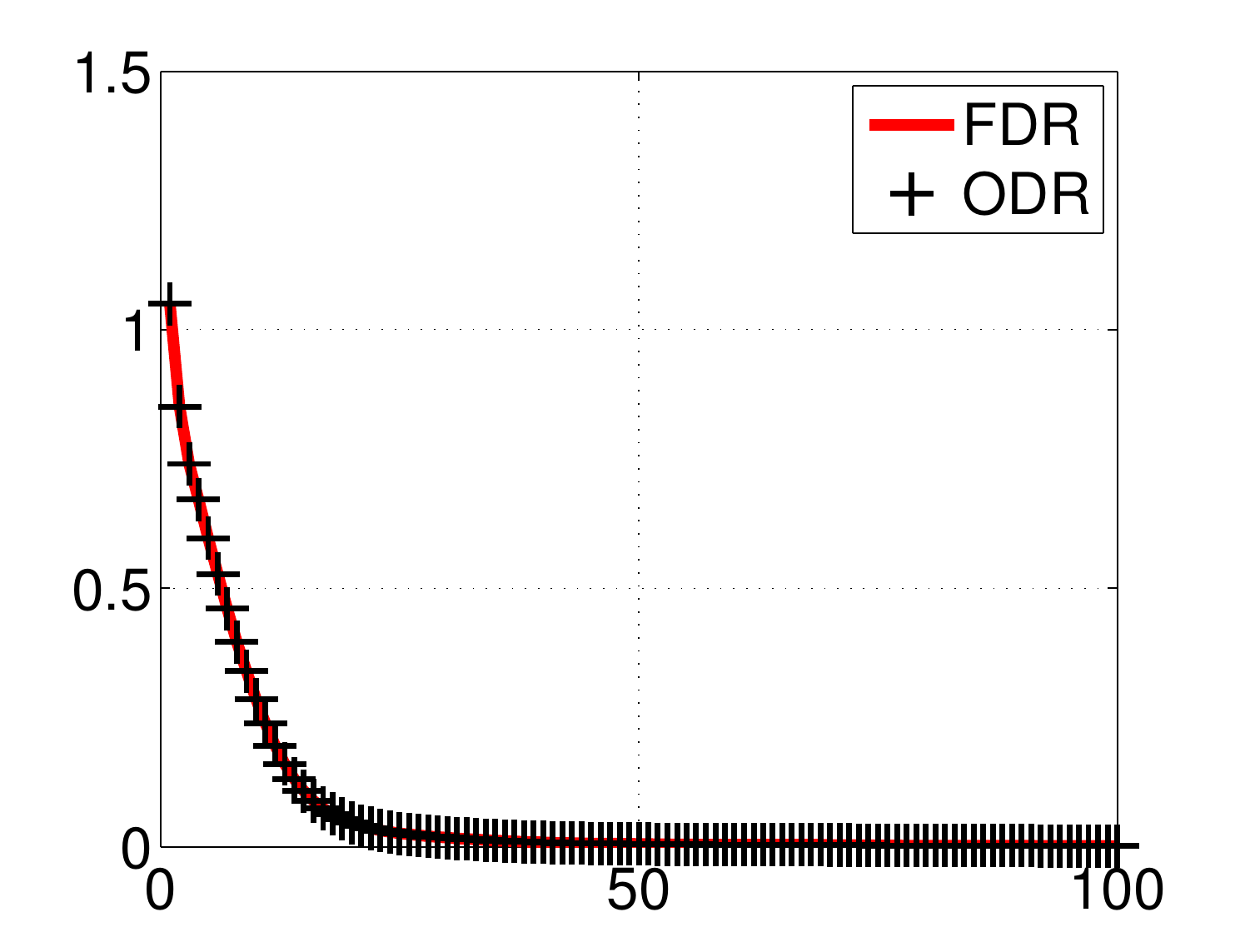}
}\hspace{-0.6cm}
\subfigure[RPP + CI ]
{
\includegraphics[width=4.2cm]{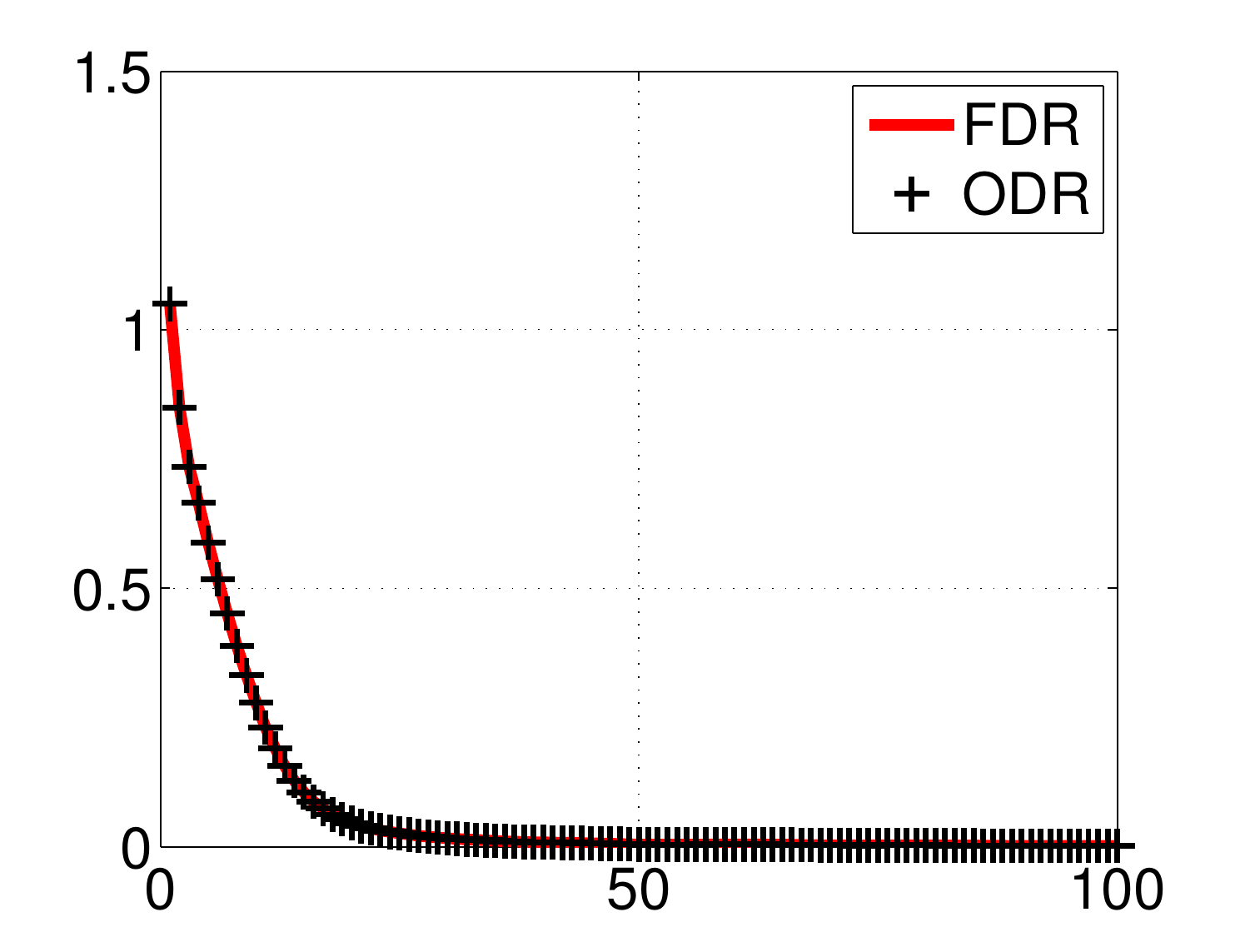}
}\hspace{-0.6cm}\subfigure[TCB + RI ]
{
\includegraphics[width=4.4cm]{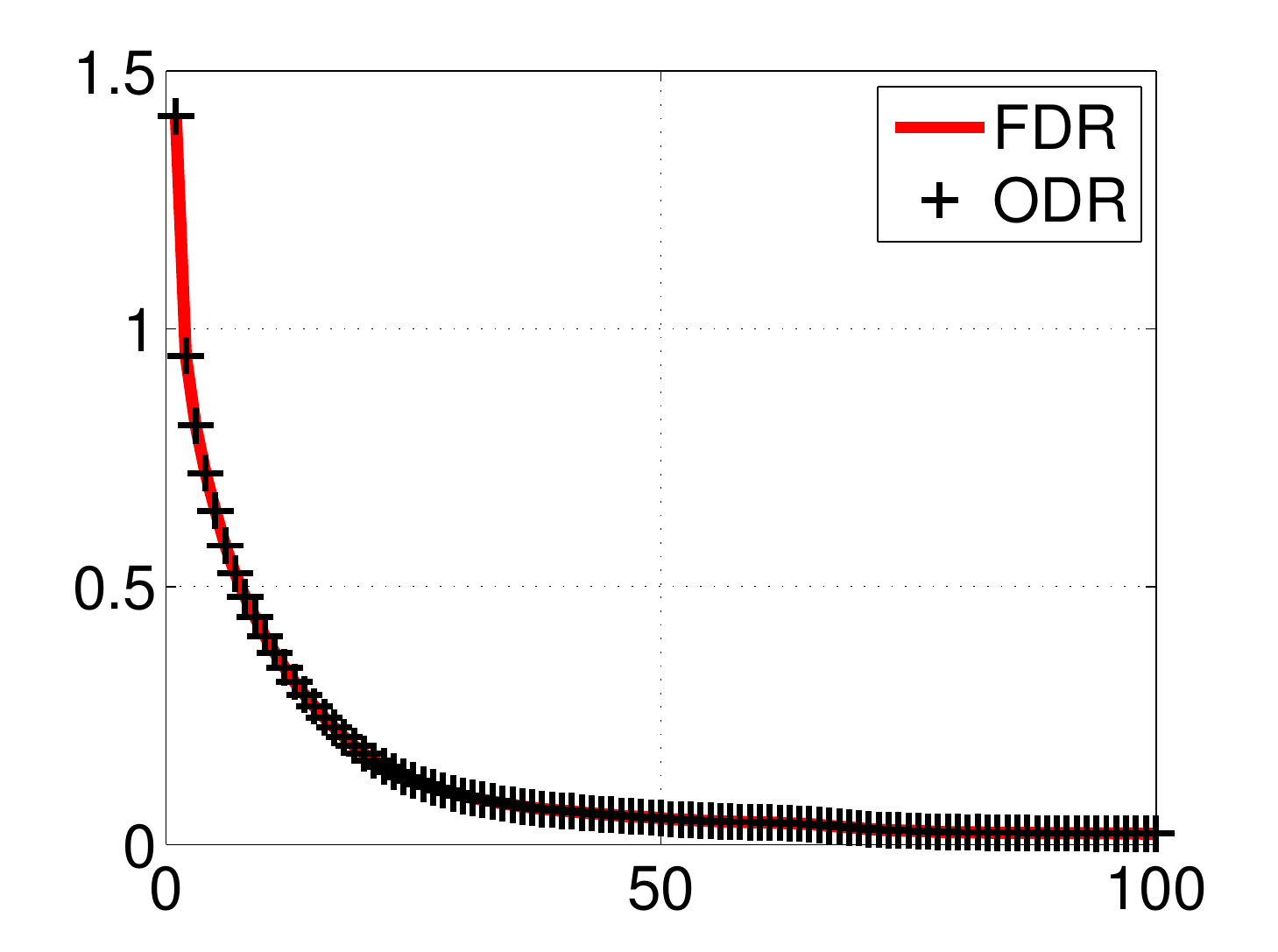}
}\hspace{-0.6cm}
\subfigure[TCB + CI]
{
\includegraphics[width=4.1cm]{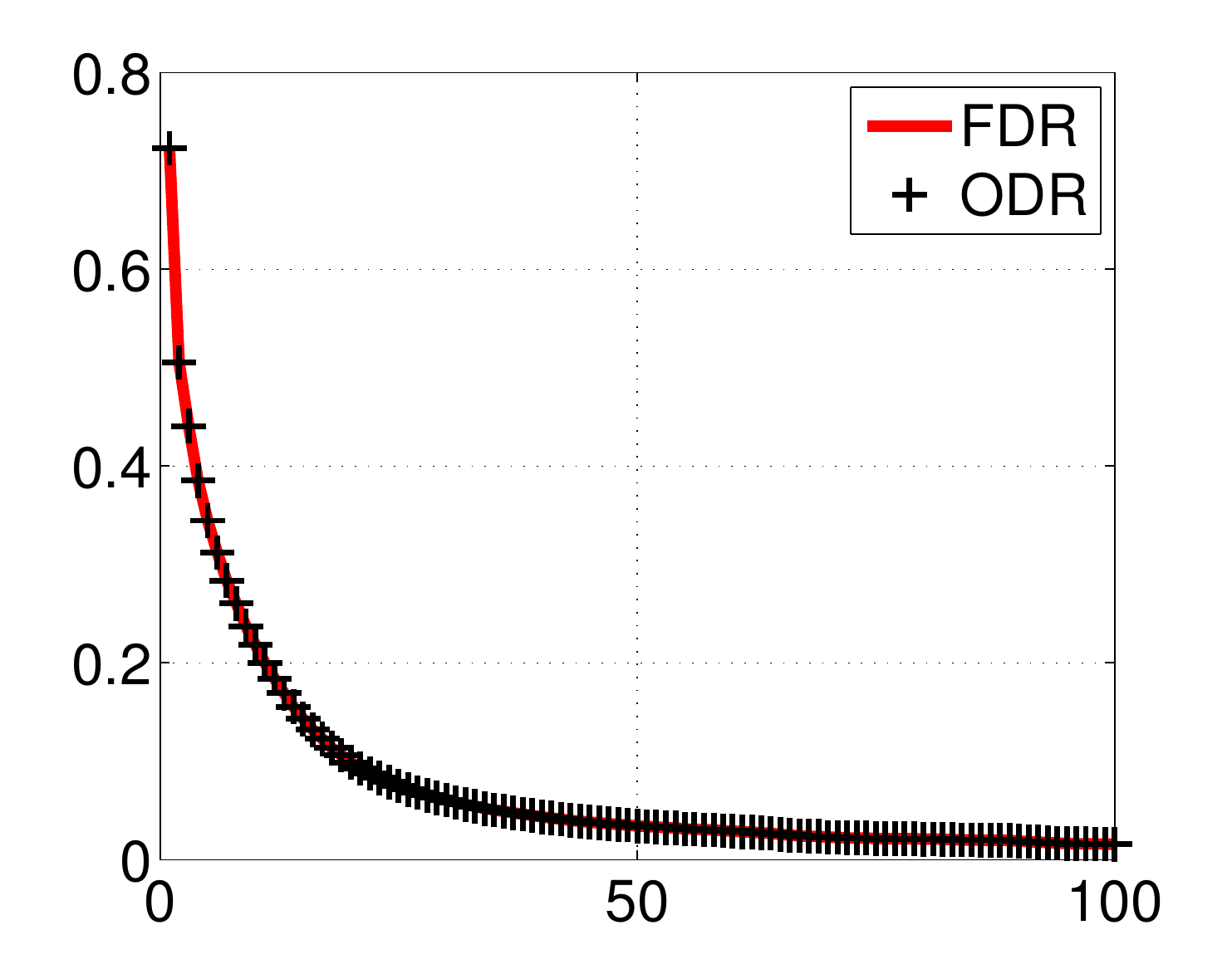}
}

\subfigure[RPP + RI]
{
\includegraphics[width=4.2cm]{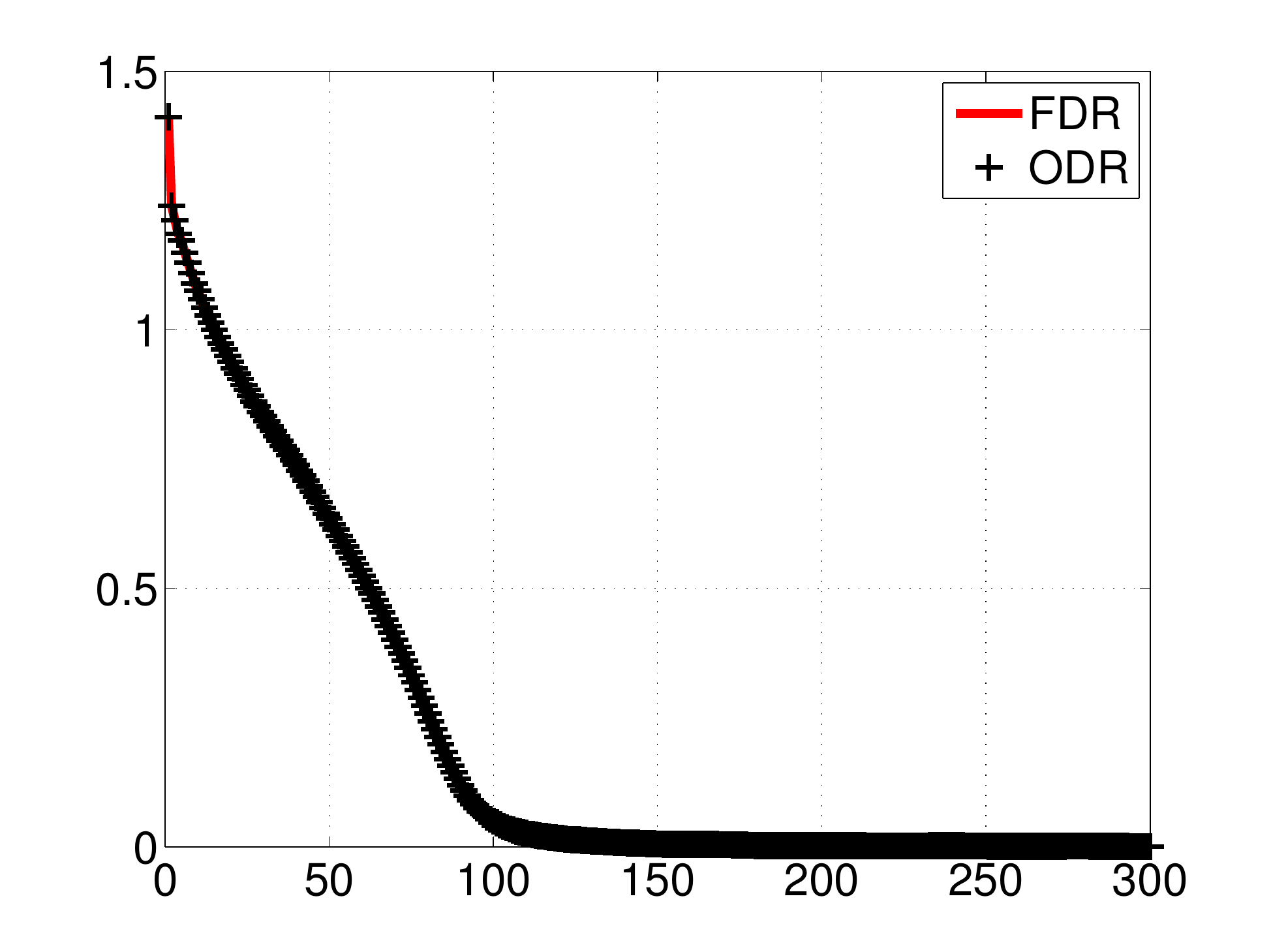}
}\hspace{-0.6cm}
\subfigure[RPP + CI]
{
\includegraphics[width=4.2cm]{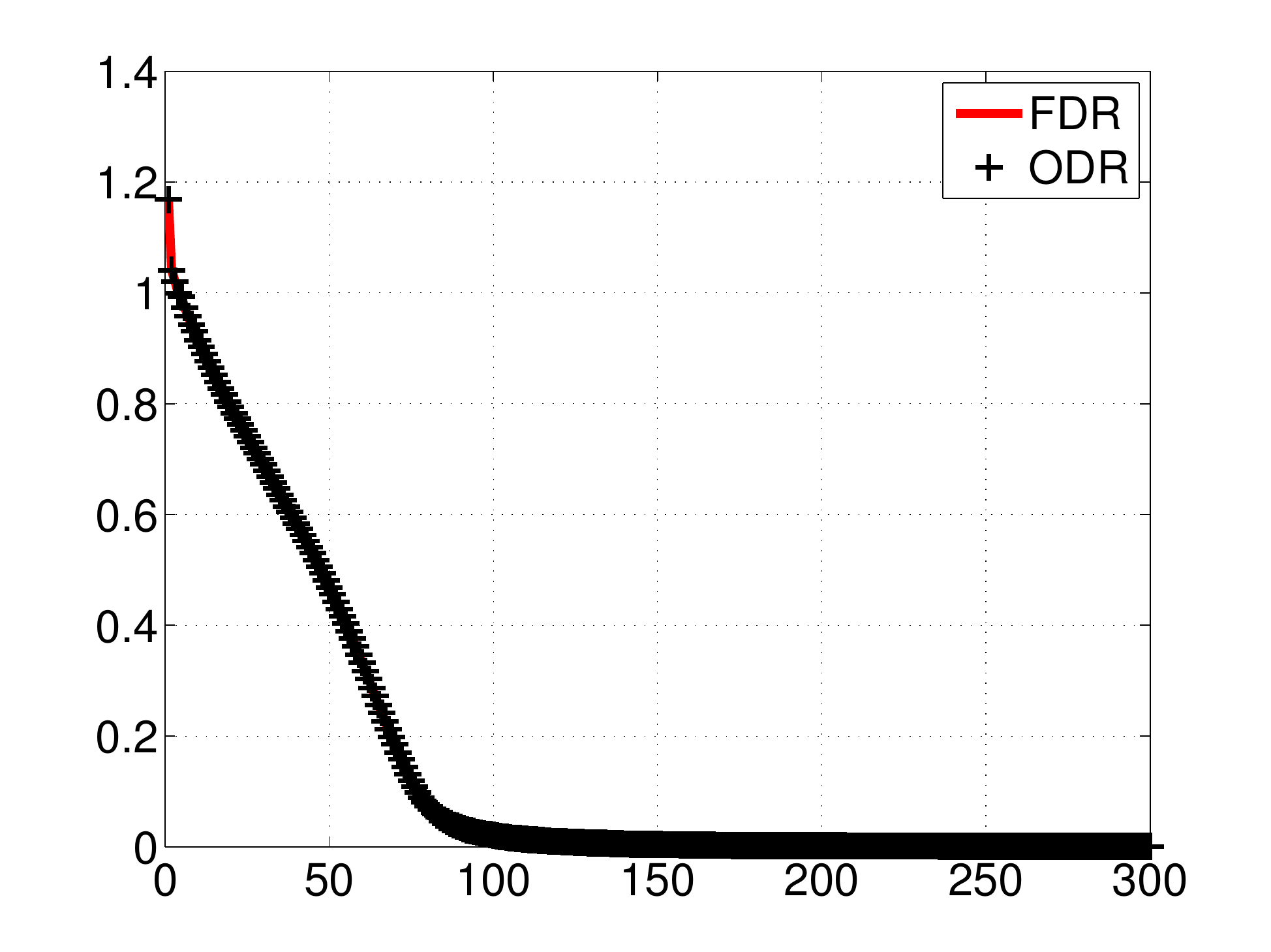}
}\hspace{-0.5cm}\subfigure[TCB + RI]
{
\includegraphics[width=4.2cm]{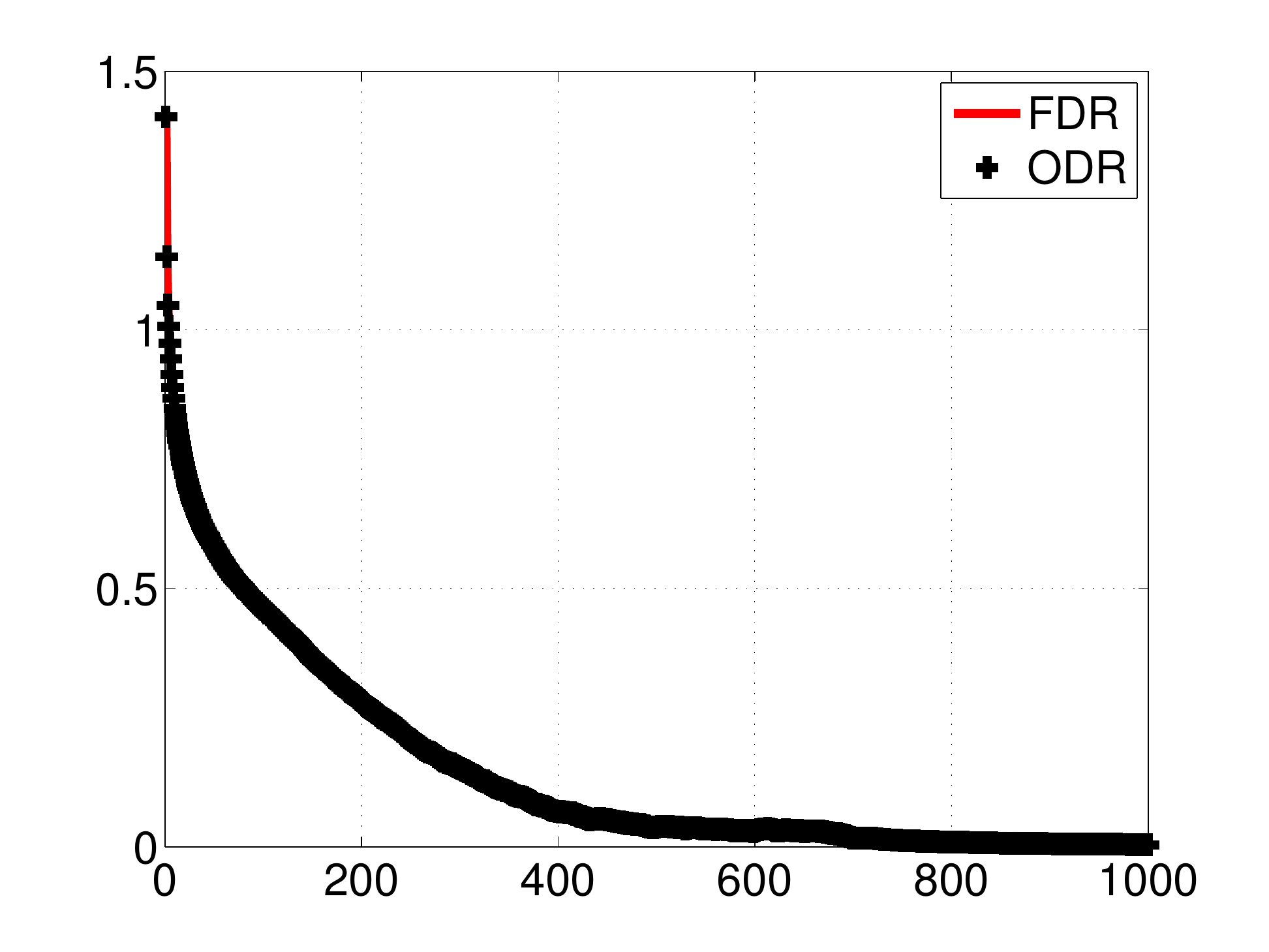}
}\hspace{-0.5cm}\subfigure[TCB + CI]
{
\includegraphics[width=4.2cm]{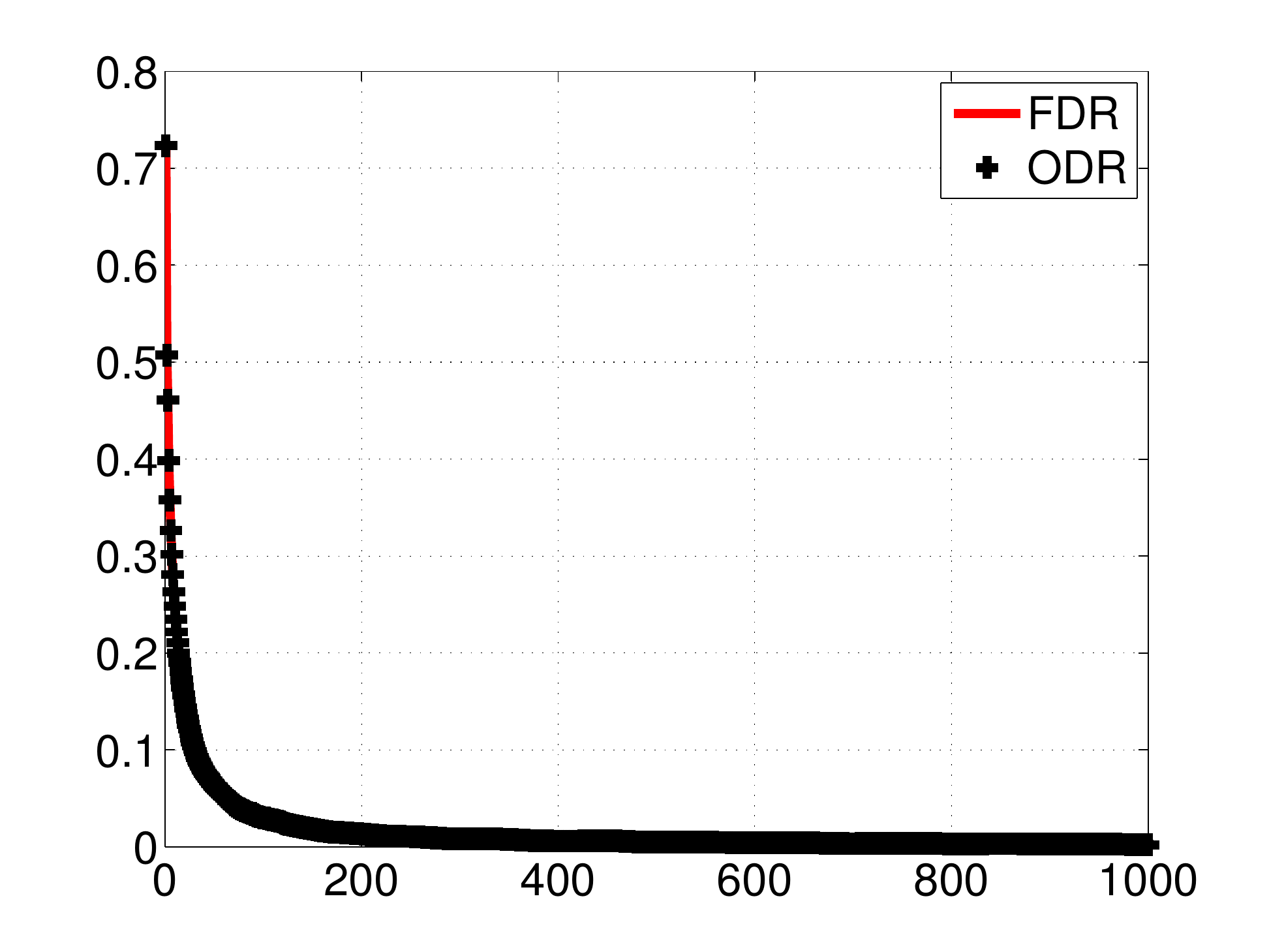}
}
\commentout{
\subfigure[RPP + RI]
{
\includegraphics[width=4.2cm]{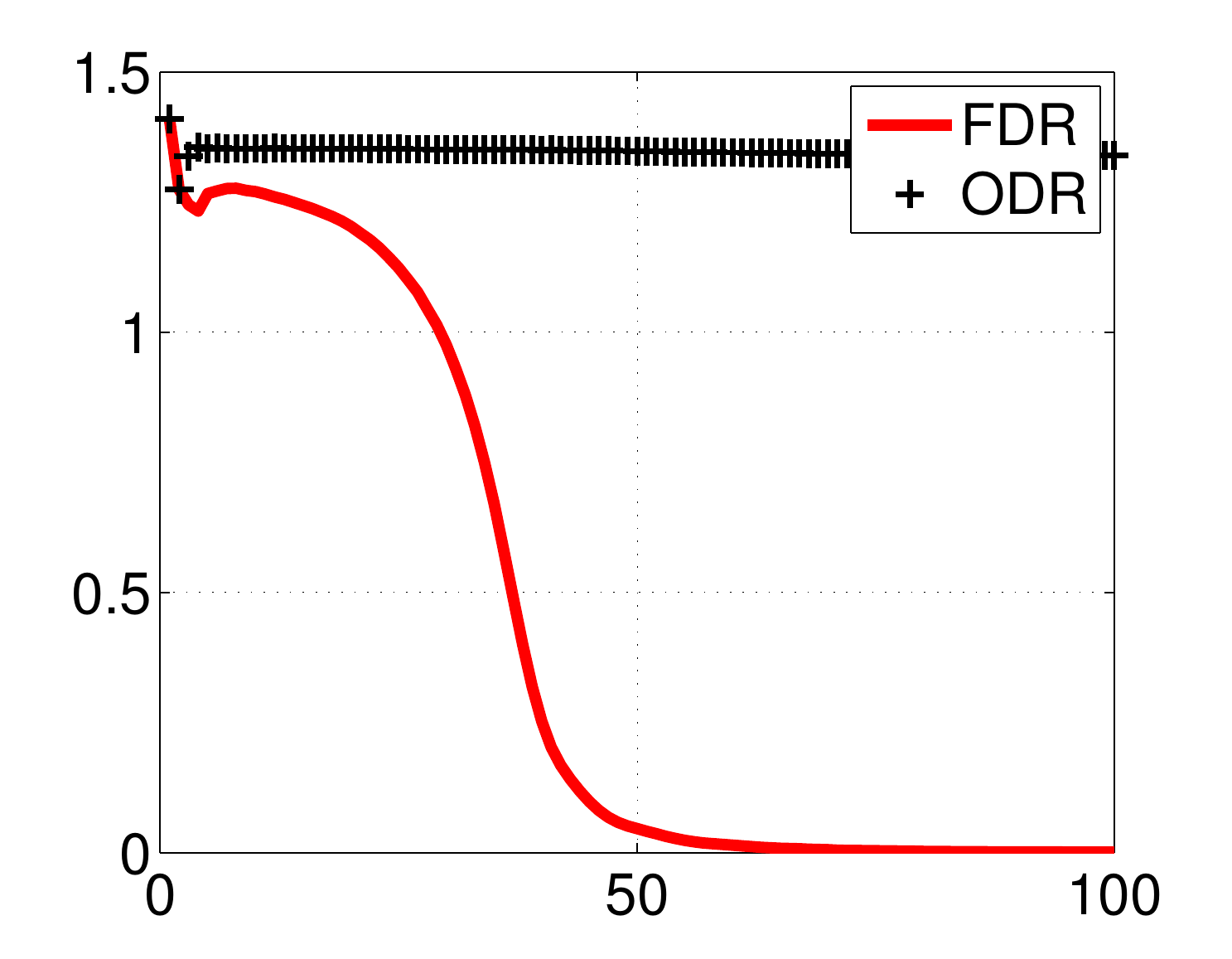}
}\hspace{-0.5cm}\subfigure[TCB + RI]
{
\includegraphics[width=4.2cm]{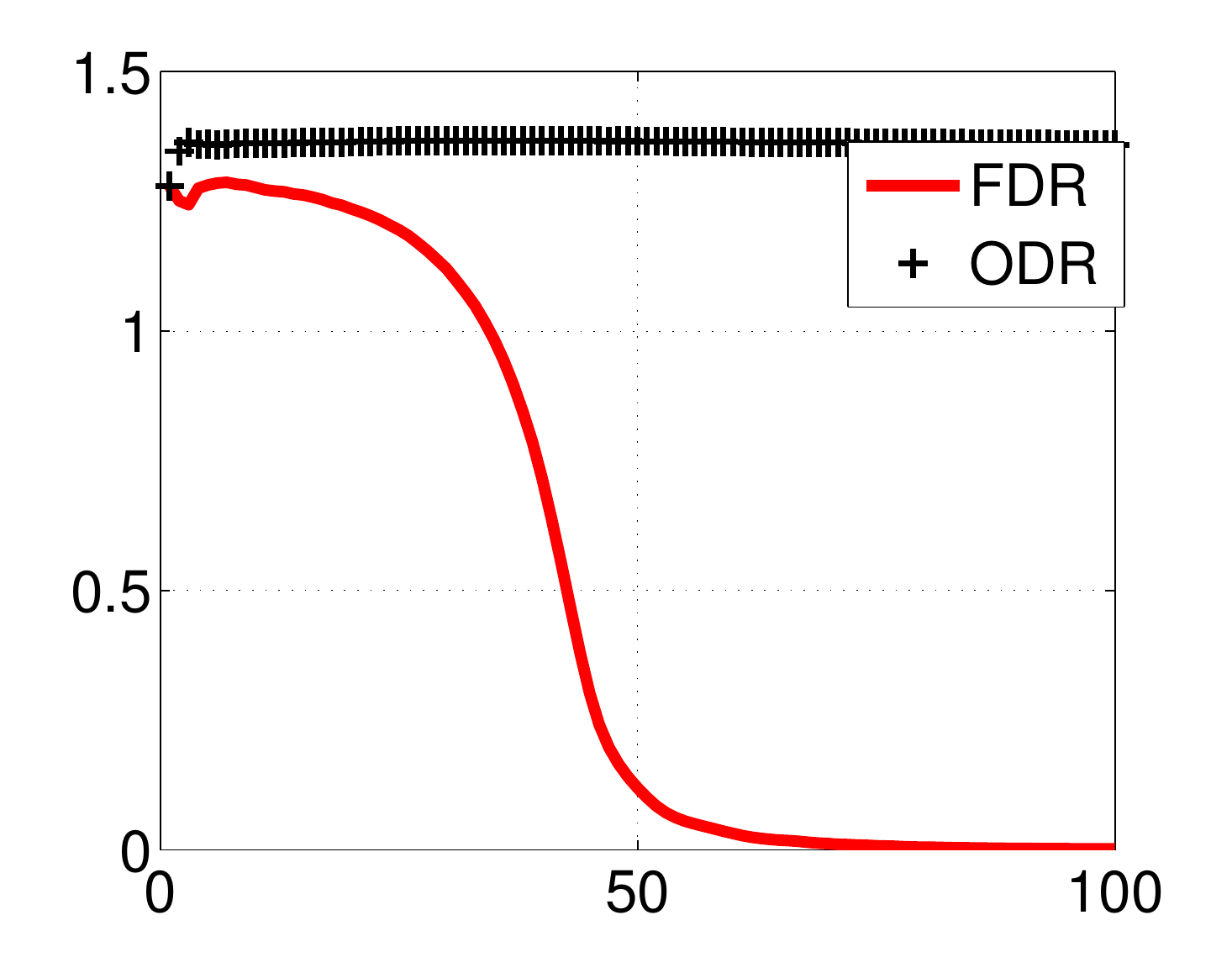}
}\hspace{-0.5cm}\subfigure[RPP + CI]
{
\includegraphics[width=4.2cm]{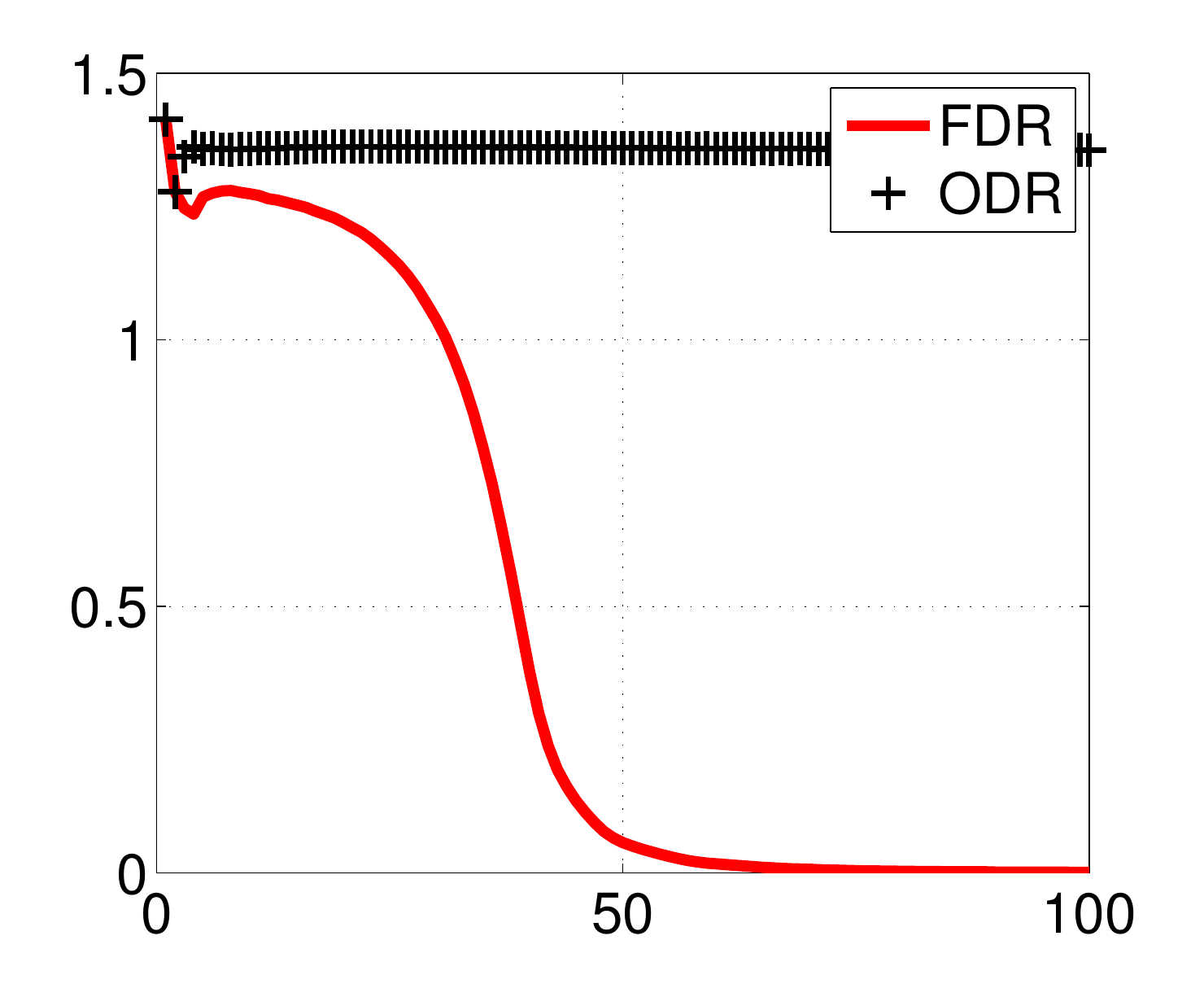}
}\hspace{-0.5cm}\subfigure[TCB + CI]
{
\includegraphics[width=4.3cm]{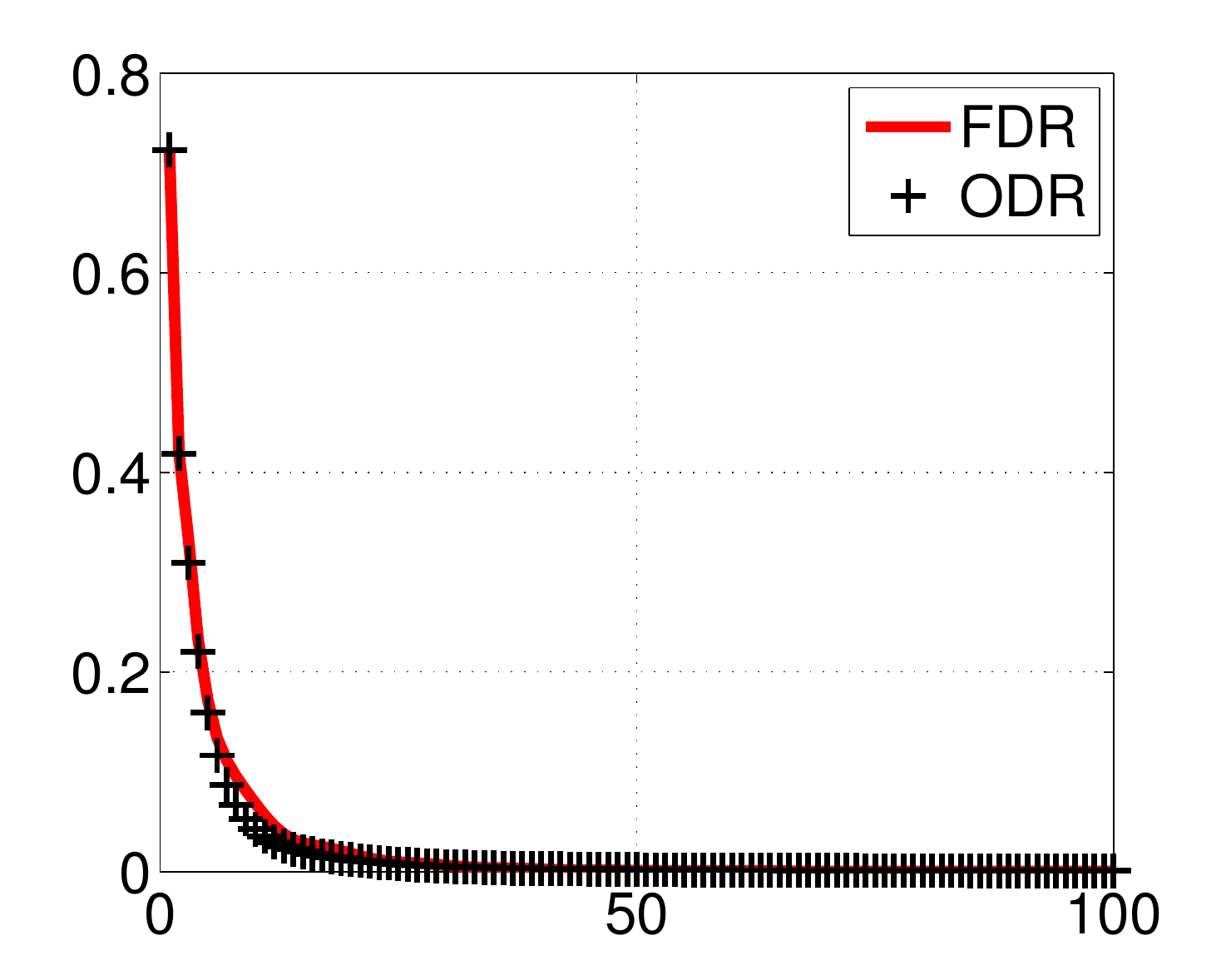}
}
\subfigure[RPP + RI]
{
\includegraphics[width=4.2cm]{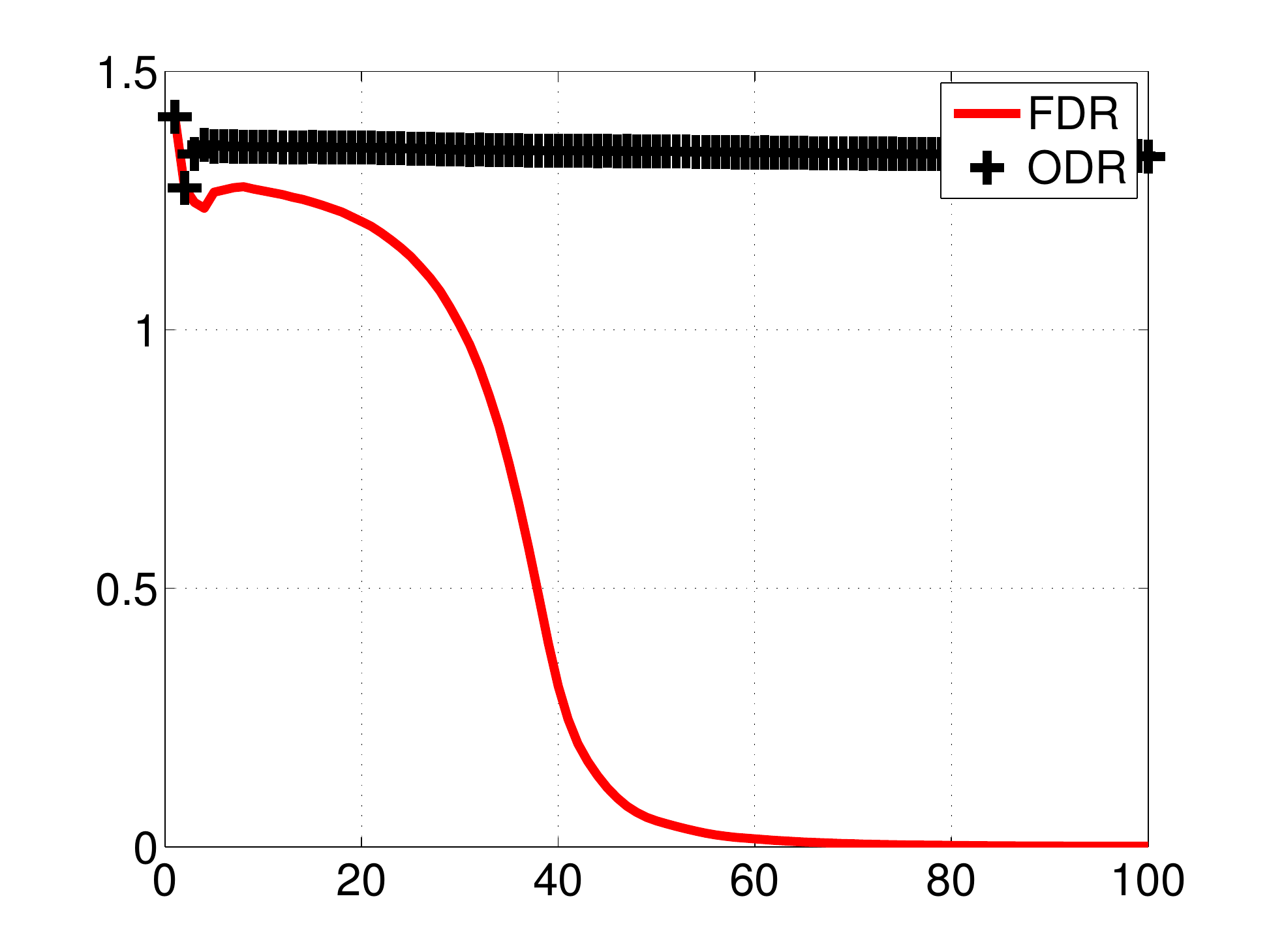}
}\hspace{-0.6cm}
\subfigure[RPP + CI]
{
\includegraphics[width=4.2cm]{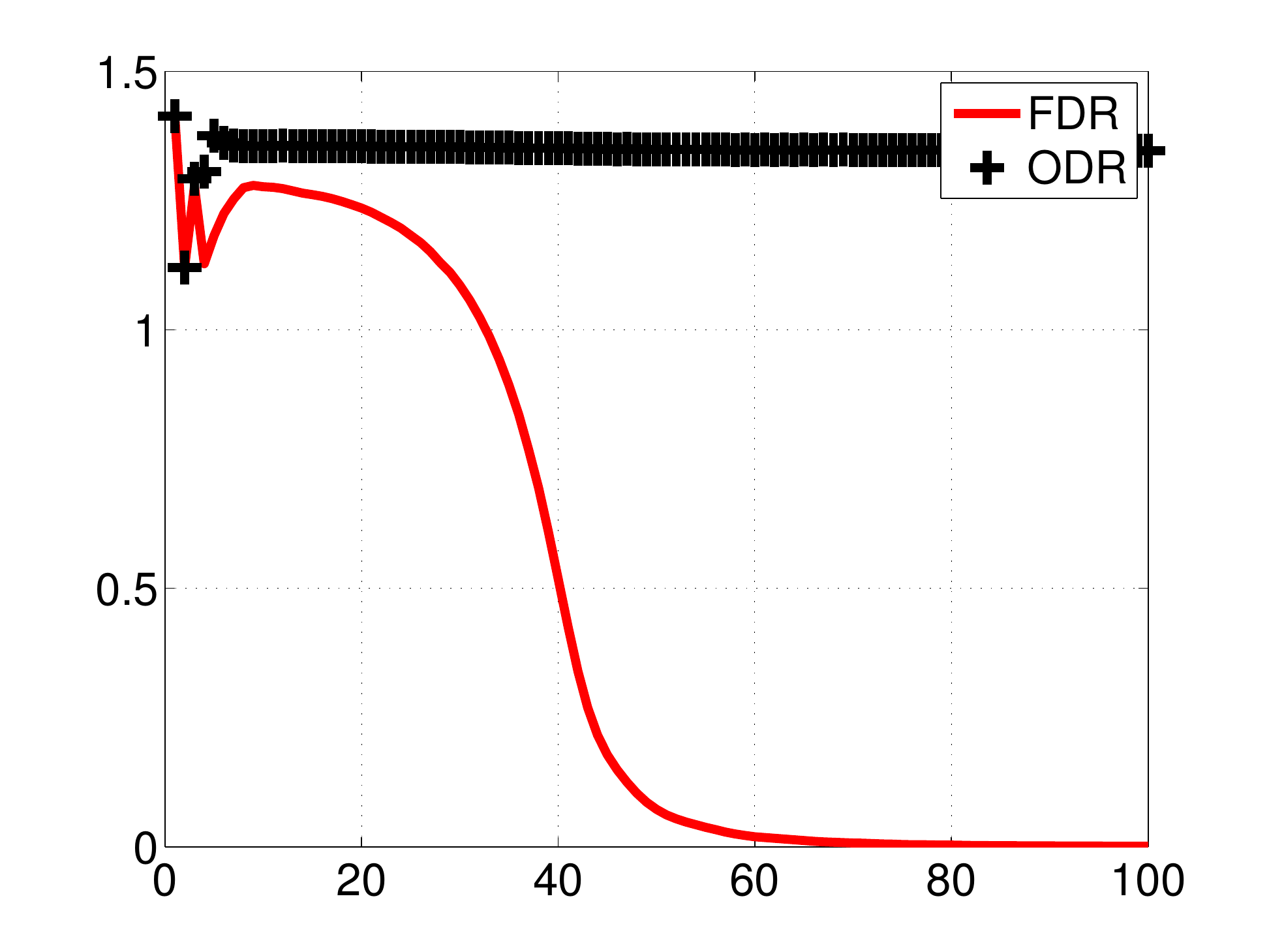}
}\hspace{-0.6cm}
\subfigure[TCB + RI]
{
\includegraphics[width=4.2cm]{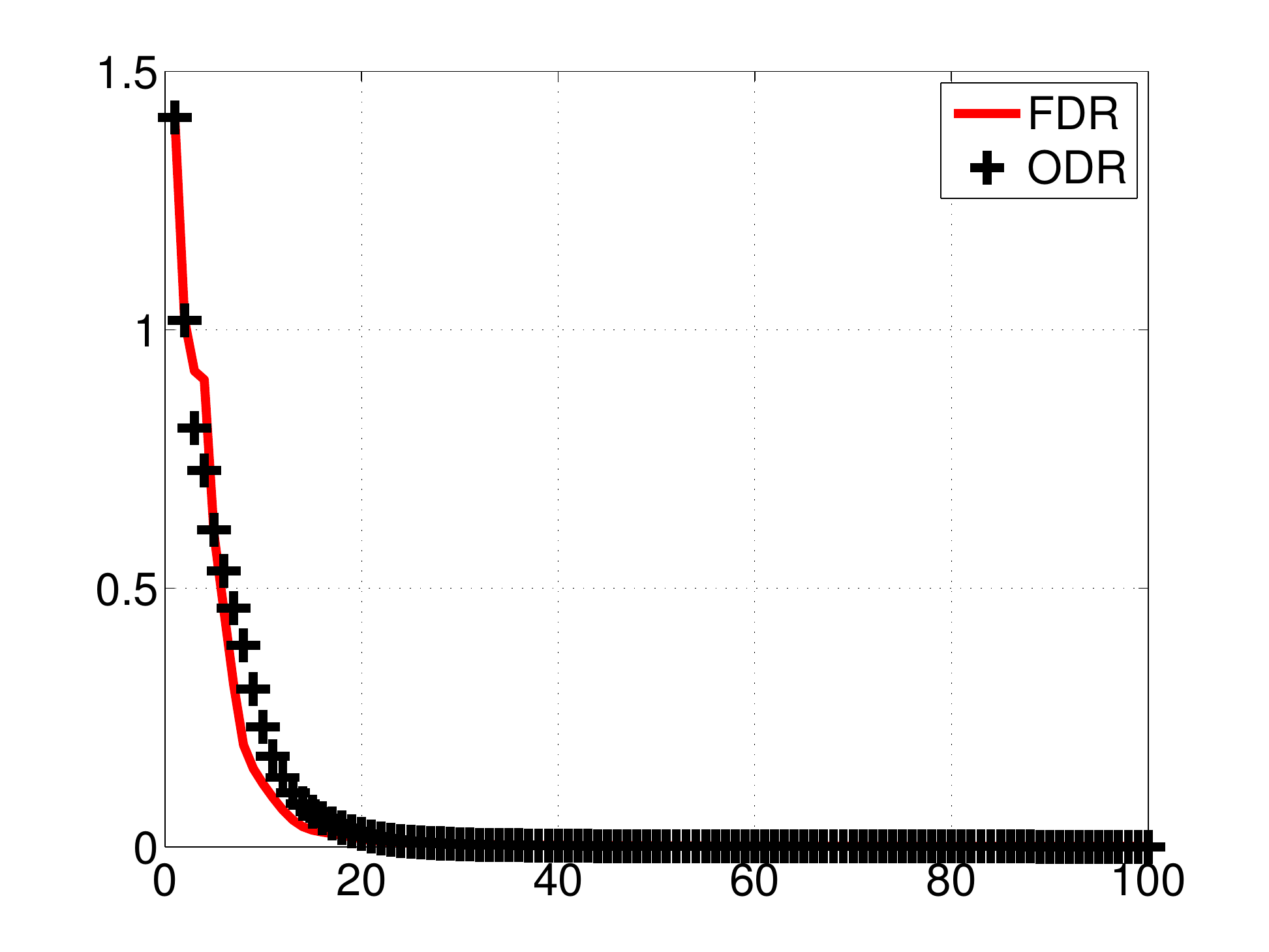}
}\hspace{-0.6cm}
\subfigure[TCB + CI]
{
\includegraphics[width=4.2cm]{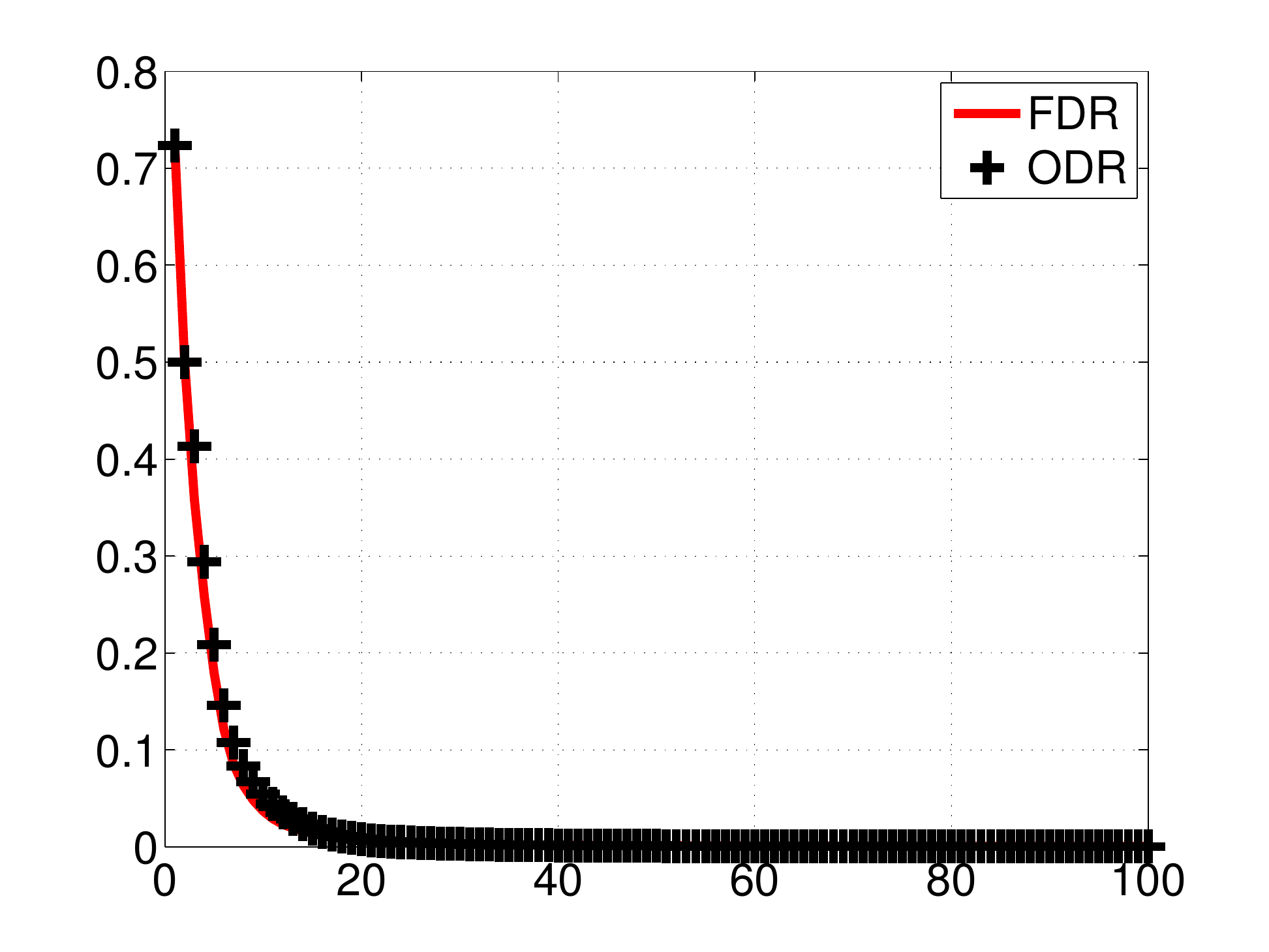}
}
}
\caption{Relative error  versus iteration in the 1-pattern case with
two different sector constraints: $[0,\pi/2]$ for (a)-(d)  and  $[0,\pi]$ for (e)-(h).}
\label{fig:one}
\end{center}
\end{figure}
Fig. \ref{fig:two} (a)-(d) shows the results of the 1-pattern case for which the sector condition is
imposed (ODR is equivalent to FDR as $N=\tilde n$). 

To test the effect of  the sector constraint, the phase of RPP is uniformly distributed in two different intervals: $[0,\pi/2]$ and $[0,\pi]$. While FDR/ODR global convergence  regardless the initialization is evident, the rate of convergence decreases as
the sector enlarges. When the sector constraint is absent, the iteration ceases to converge in general.

\subsection{$1\half$-mask case}

\begin{figure}[t!]
\subfigure[RPP + RI]
{
\includegraphics[width=4.2cm]{figs/PP1_5RI}
}\hspace{-0.6cm}
\subfigure[RPP + CI]
{
\includegraphics[width=4.2cm]{figs/PP1_5CI}
}\hspace{-0.6cm}
\subfigure[TCB + RI]
{
\includegraphics[width=4.2cm]{figs/TCB1_5RI}
}\hspace{-0.6cm}
\subfigure[TCB + CI]
{
\includegraphics[width=4.2cm]{figs/TCB1_5CI}
}
\caption{Relative error versus iteration in the $1\half$-mask case.}
\commentout{
\subfigure[FDR + RI]
{
\includegraphics[width=3.5cm]{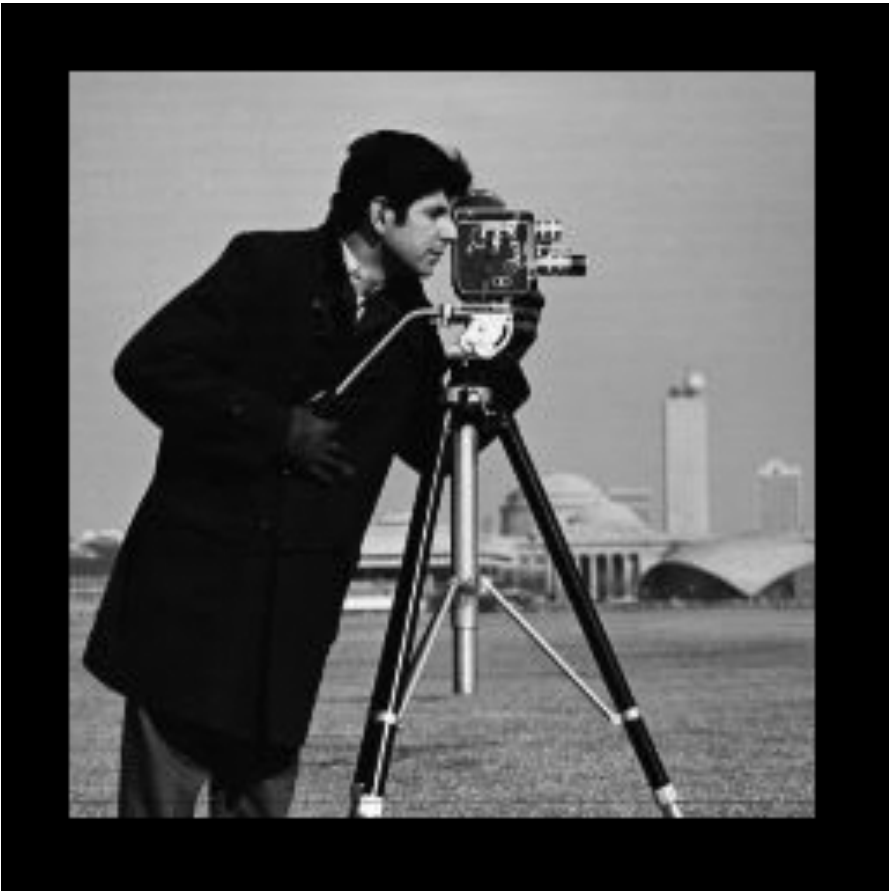}
\includegraphics[width=3.5cm]{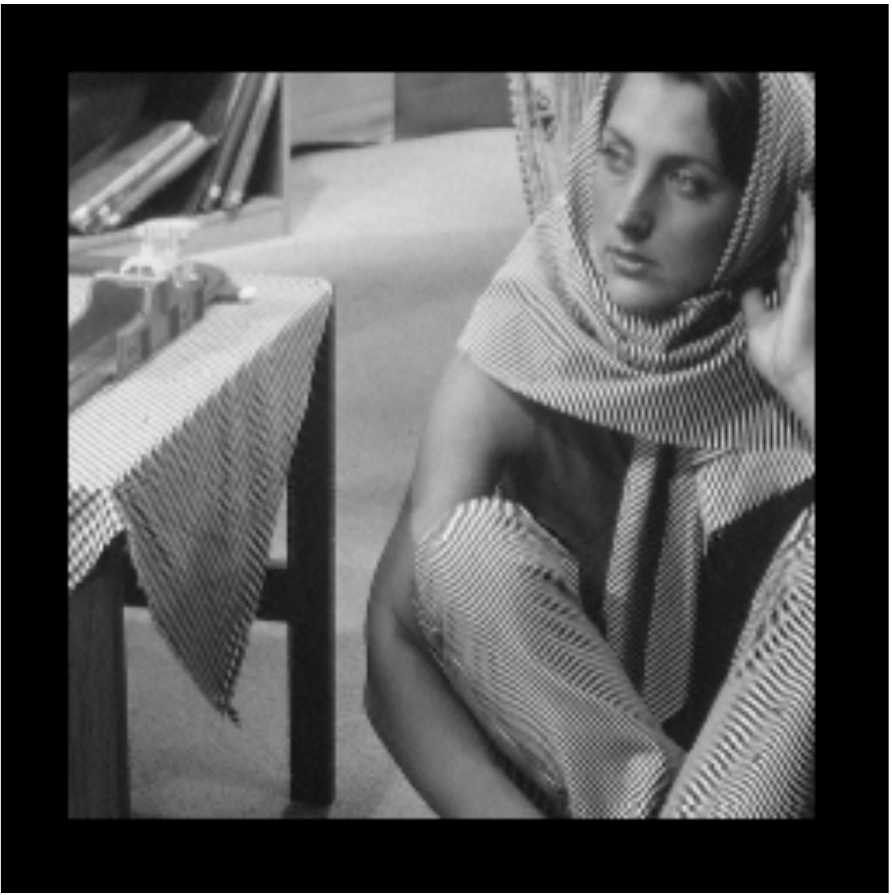}
}
\subfigure[FDR + CI]
{
\includegraphics[width=3.5cm]{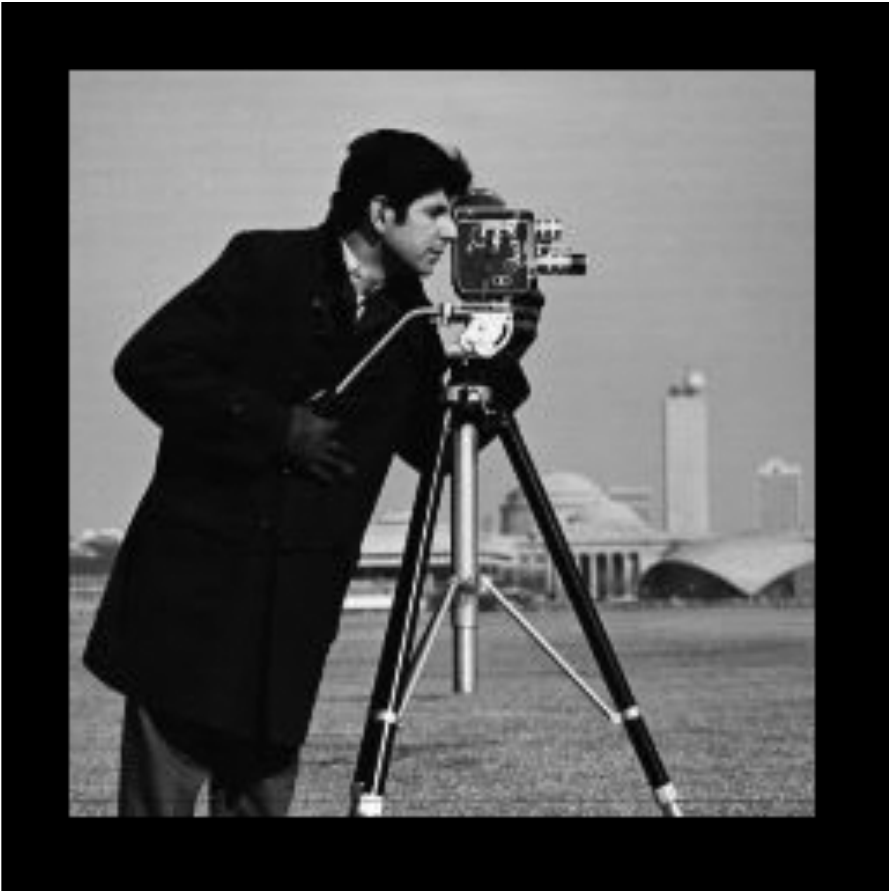}
\includegraphics[width=3.5cm]{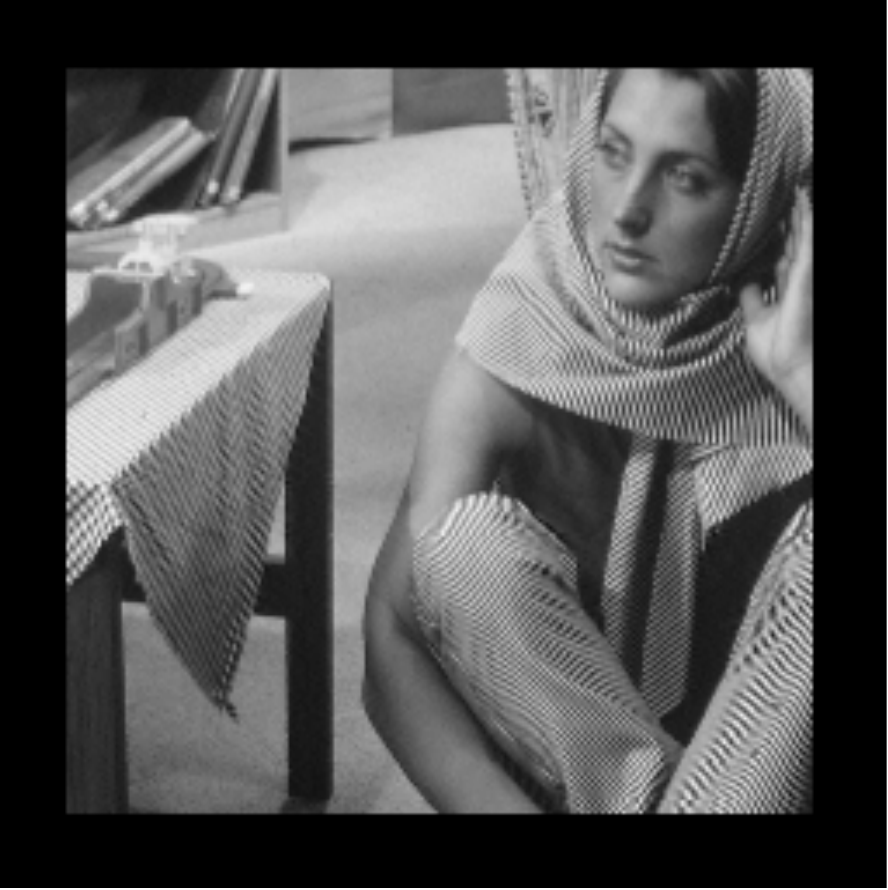}
}
\subfigure[ODR + RI]
{
\includegraphics[width=3.5cm]{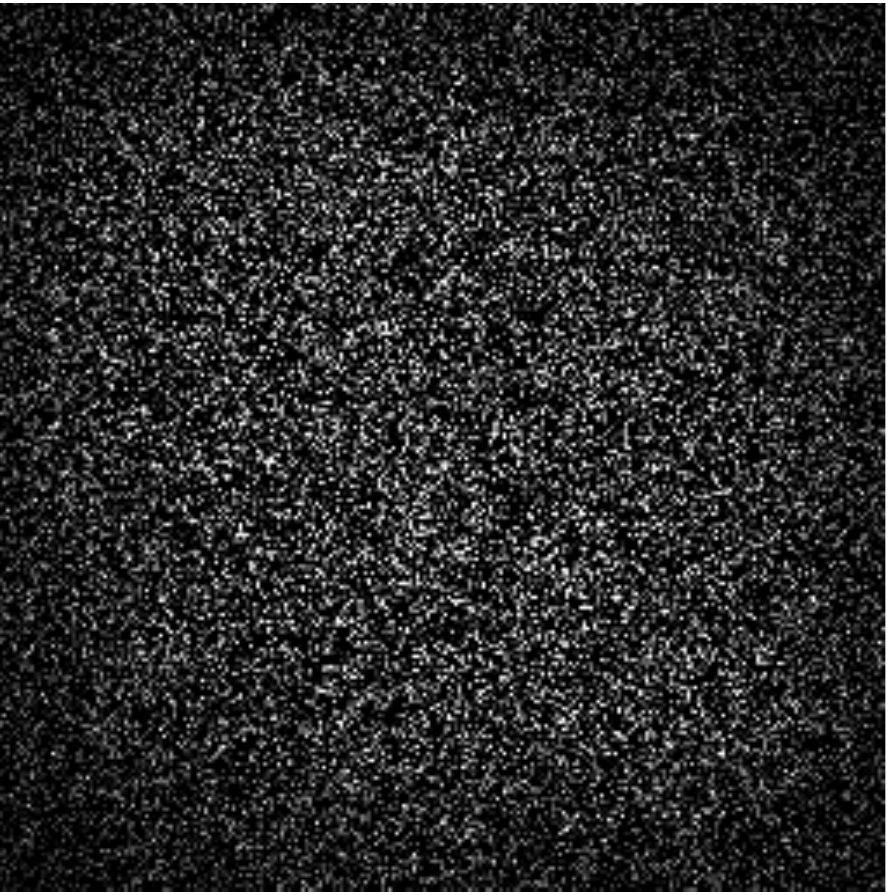}
\includegraphics[width=3.5cm]{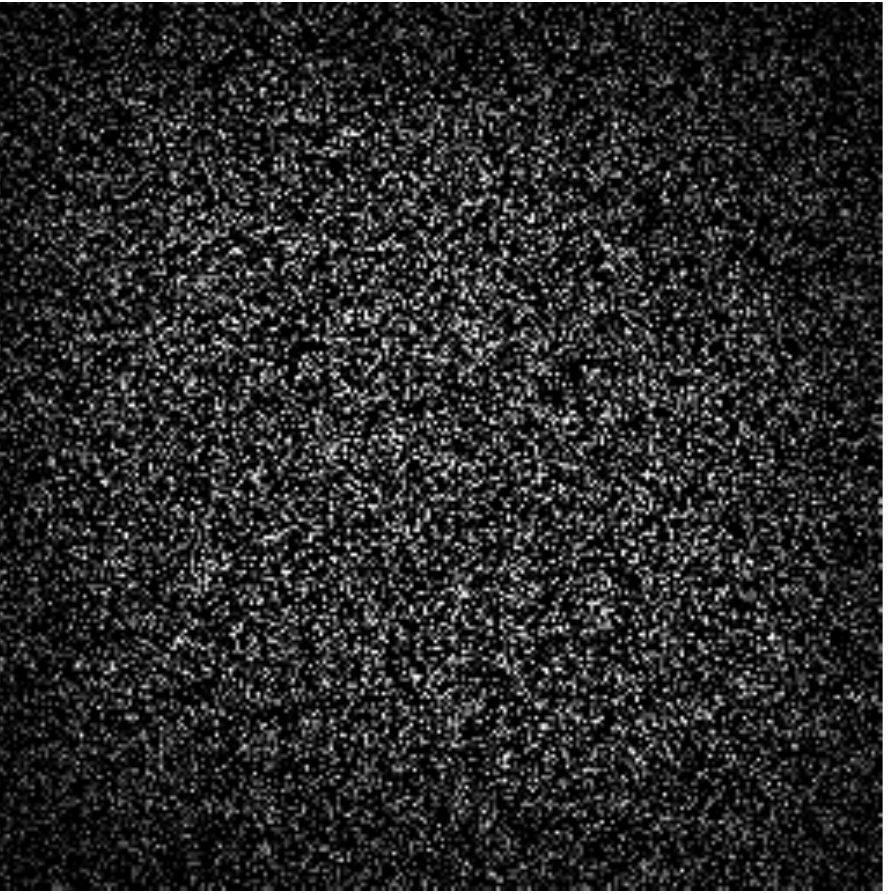}
}
\subfigure[ODR + CI]
{
\includegraphics[width=3.5cm]{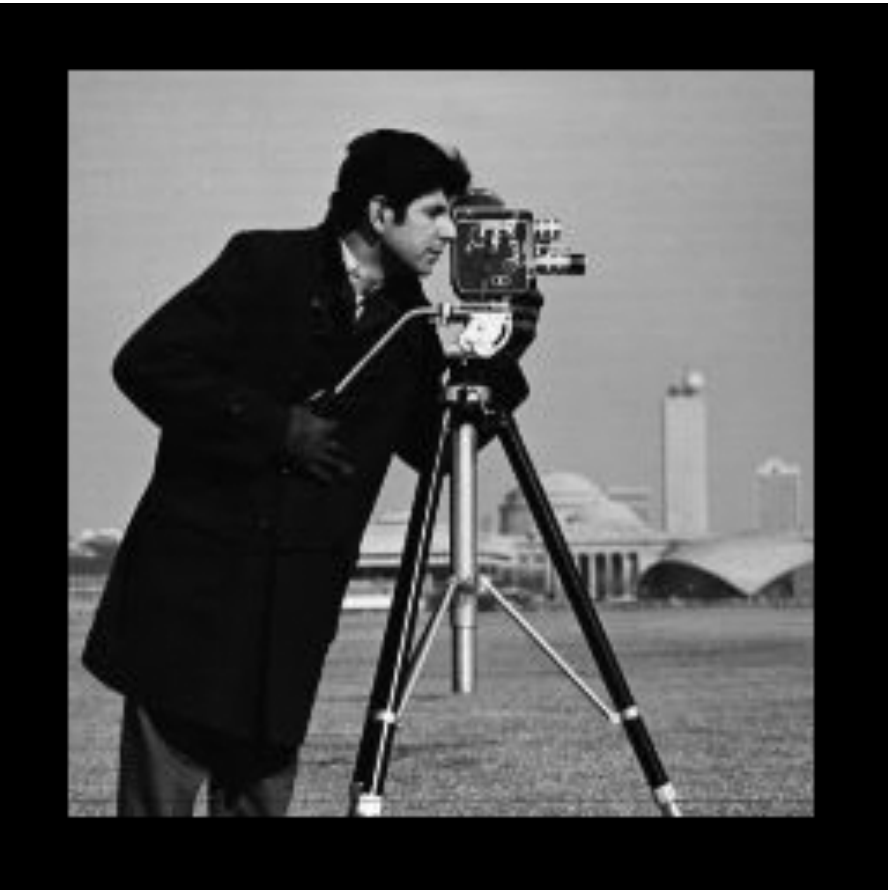}
\includegraphics[width=3.5cm]{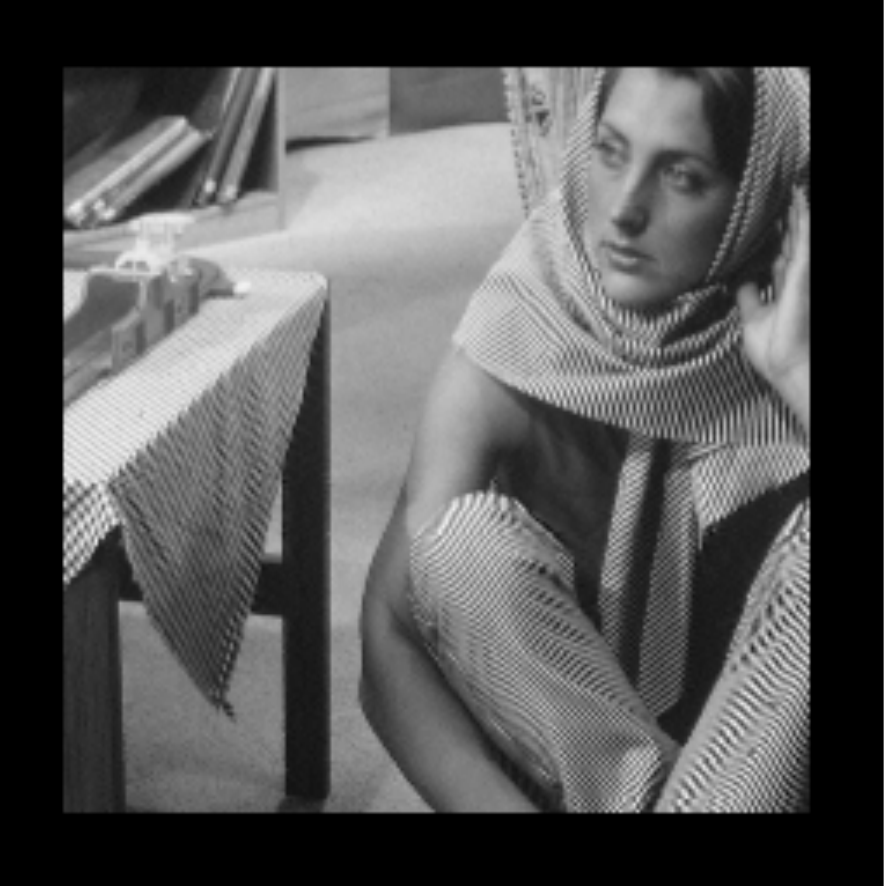}
}
\caption{Reconstructions of TCB in the $1\half$-mask case, c.f. Fig. \ref{fig:two} (f)\&(h)}
}
\label{fig5}\label{fig:two}
\end{figure}

For two- or multi-pattern case, we let the phase of RPP be uniformly distributed in $[0,2\pi]$ (i.e. no sector constraint). 
Fig. \ref{fig:two} (e)-(h) shows the results of the $1\half$-mask case for which ODR is implemented
with $\tilde n\approx 4n<N\approx 8n$ and hence not equivalent to FDR. We see that  the performances
of ODR and FDR are drastically different: While FDR converges to the true images
regardless the initialization within 100 iterations, ODR does so only for the deterministic image TCB. 



Fig. \ref{fig:noise} shows the relative error versus noise-to-signal ratio (NSR) when noise $\ep$ is present in the data where 
\[
\hbox{\rm NSR}\,\,={\|\ep\|\over \|A^*x_0\|}. 
\]
We note that for NSR $\in [0, 20\%]$ there is essentially no difference between the results
with the maximum number of iterations set to 100 and  200 and this segment  of error-noise curves is approximately the same straight line (slope $\approx 2.2$).  For a higher NSR, increasing
the maximum number of iterations reduces the error so with even greater number of 
iterations the straight line segment can be extended to NSR greater than $20\%$.

\begin{figure}[t]
\begin{center}
\subfigure[RPP]{
\includegraphics[width=5cm]{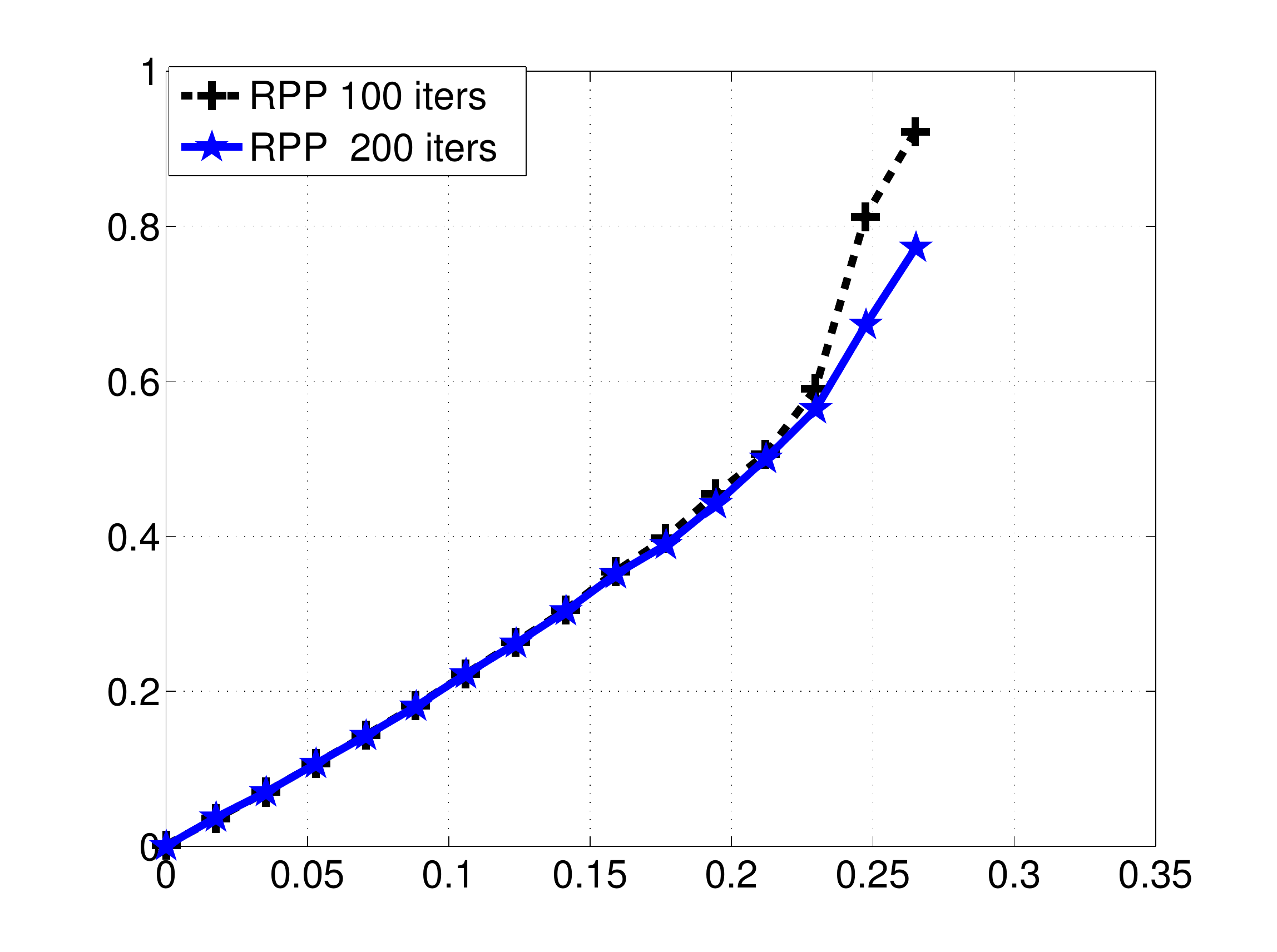}
}\quad 
\subfigure[TCB]{
\includegraphics[width=5cm]{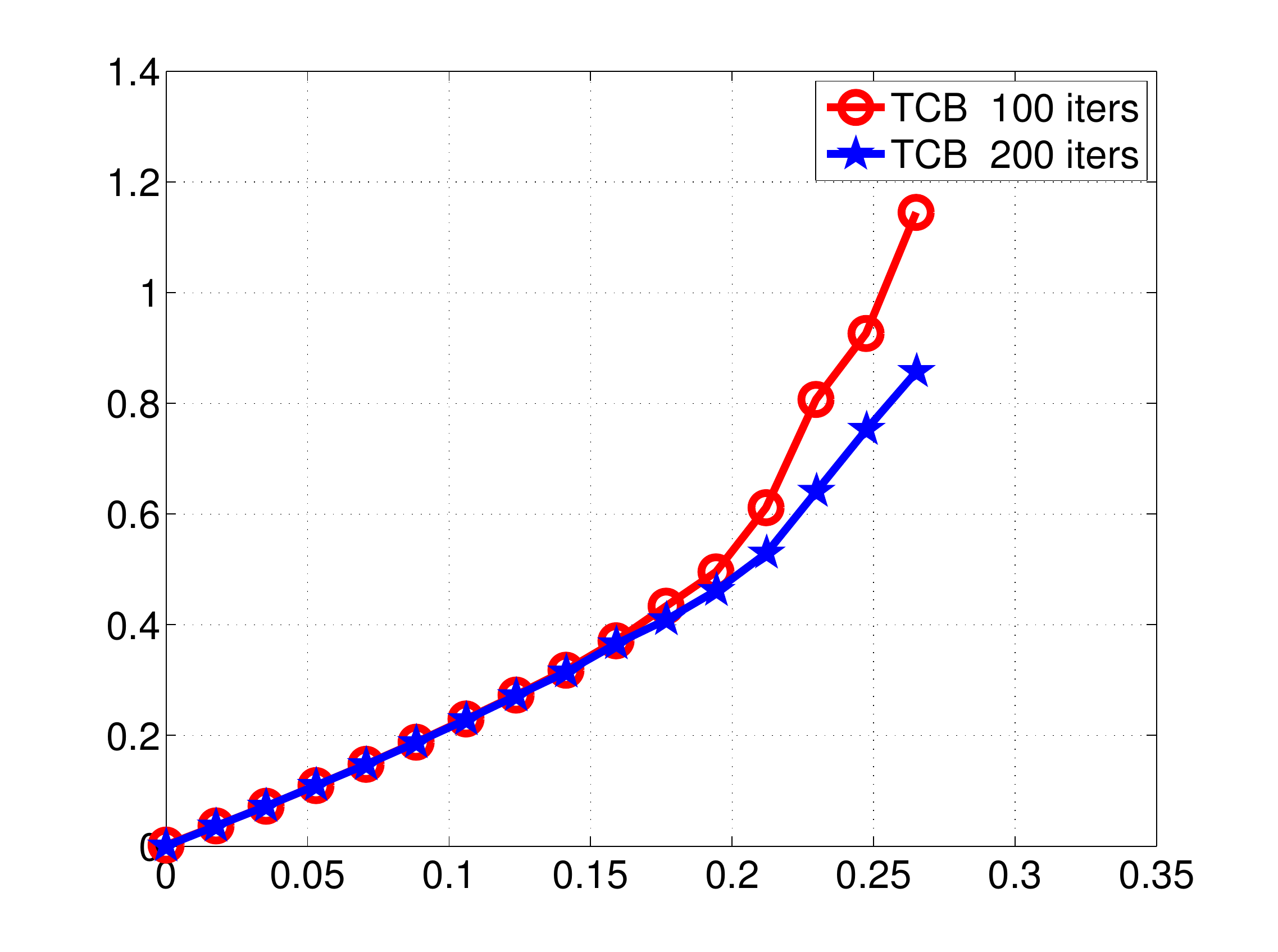}
}
\caption{Relative error versus NSR in the $1\half$-mask case
(both oversampled) with the maximum number of iteration set to 100 or 200. 
 }
\label{fig:noise}
\end{center}
\end{figure}

\subsection{Multi-mask case}

\begin{figure}[t]
\begin{center}
\subfigure[RPP + RI]
{
\includegraphics[width=4.2cm]{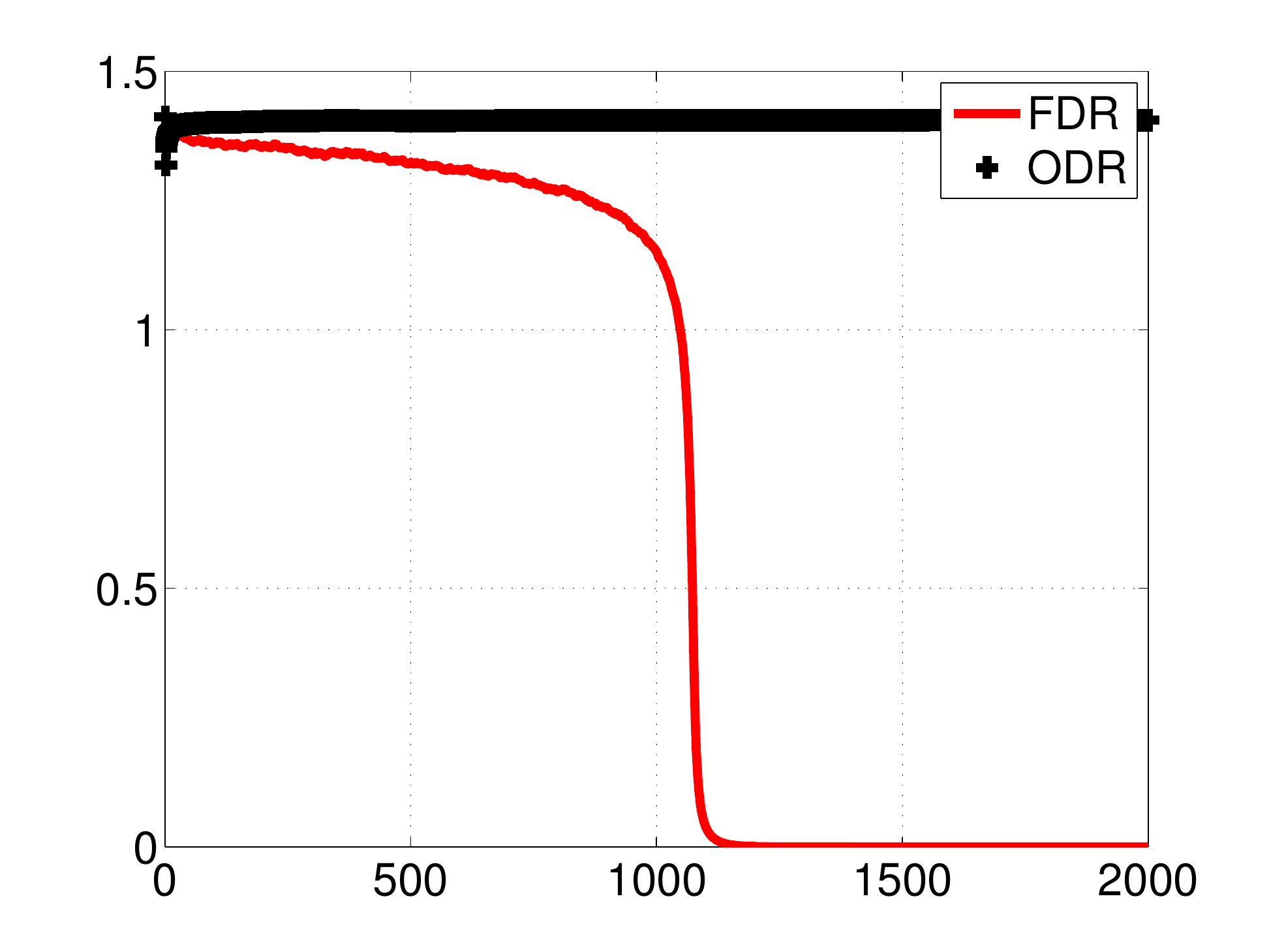}
}\hspace{-0.55cm}
\subfigure[RPP + CI]
{
\includegraphics[width=4.2cm]{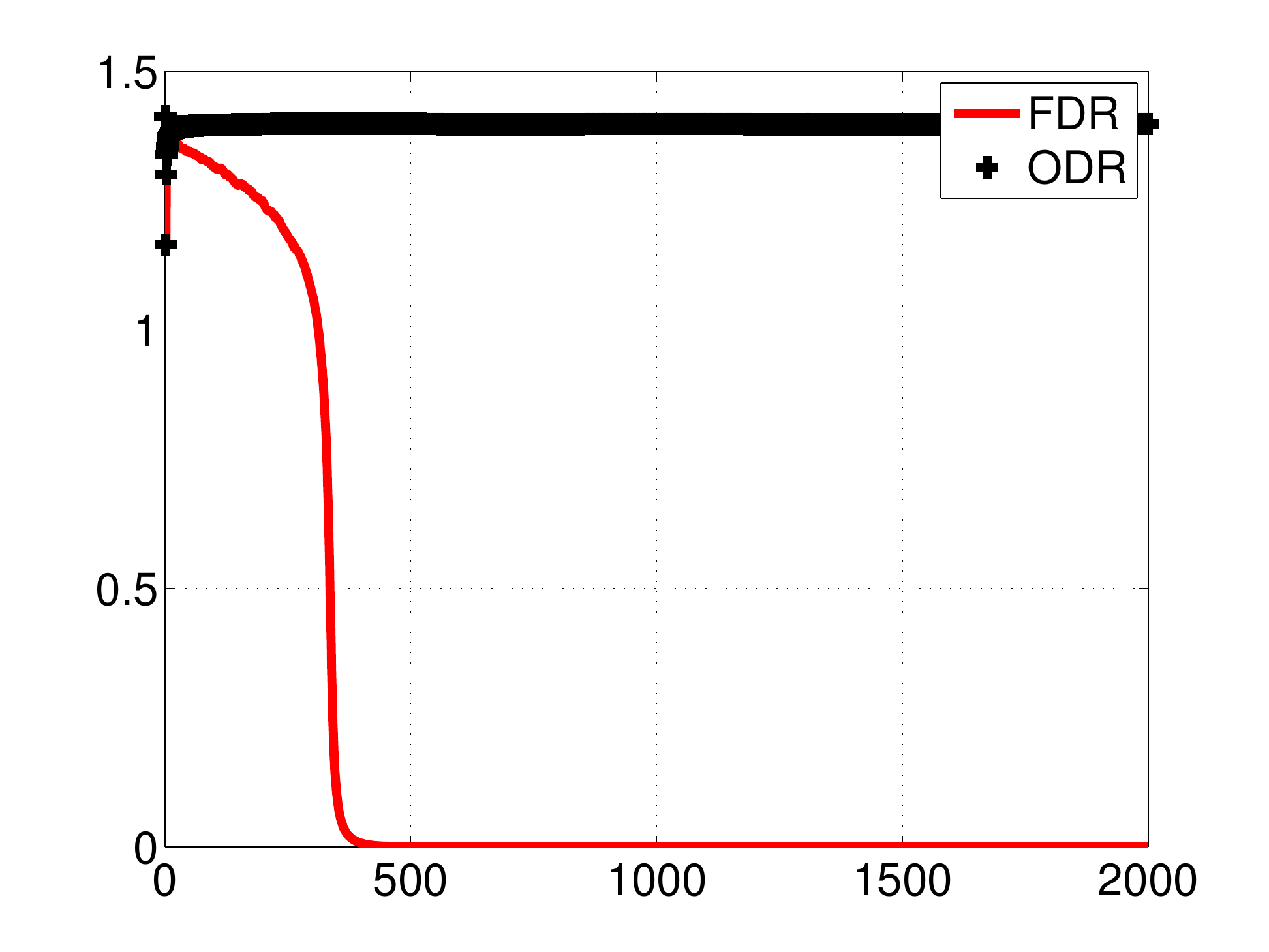}
}\hspace{-0.5cm}\subfigure[TCB + RI]
{
\includegraphics[width=4.2cm]{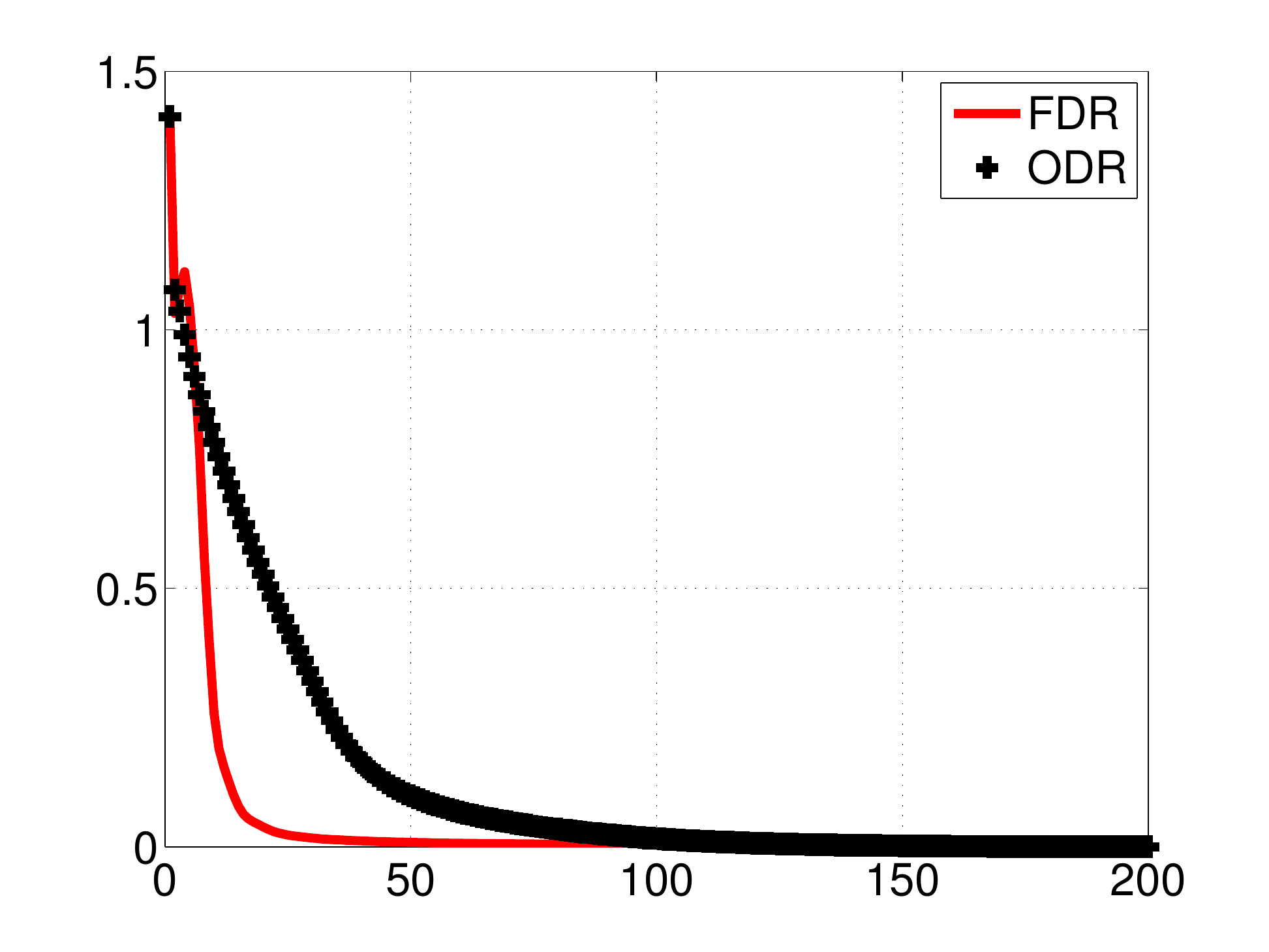}
}\hspace{-0.55cm}\subfigure[TCB + CI]
{
\includegraphics[width=4.2cm]{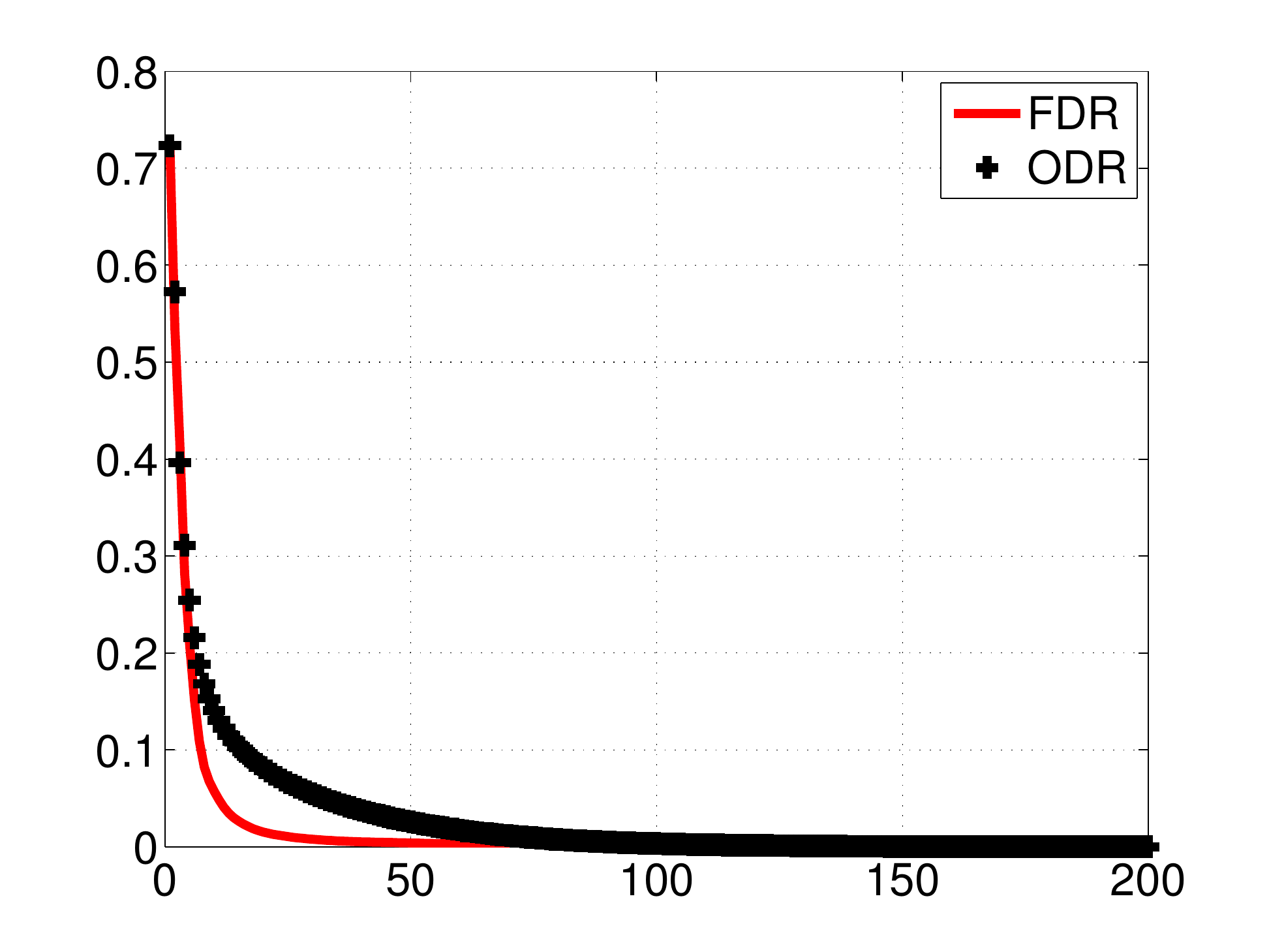}
}
\commentout{
\subfigure[RPP + RI]
{
\includegraphics[width=4.2cm]{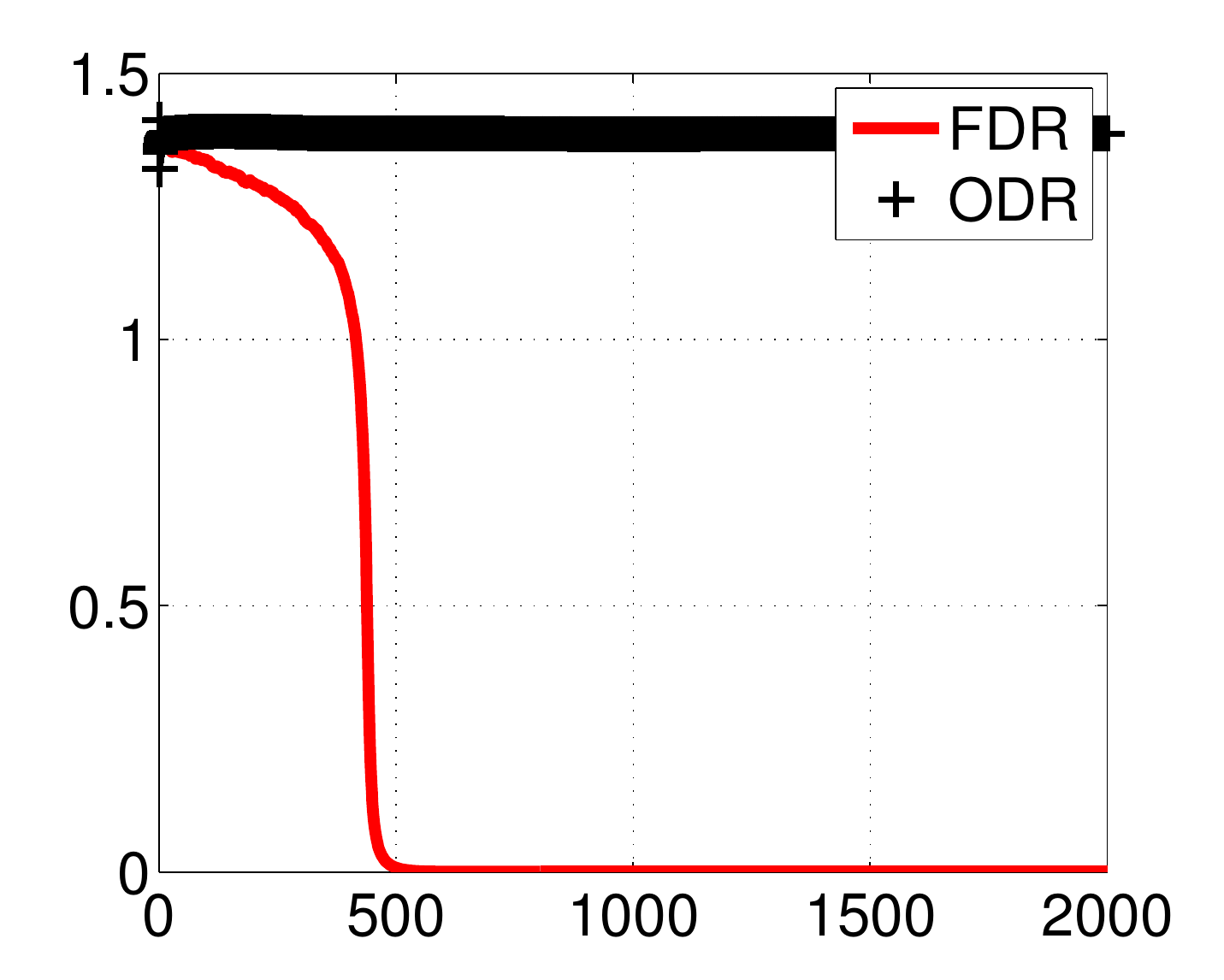}
}
\hspace{-0.7cm}
\subfigure[RPP + CI]
{
\includegraphics[width=4.1cm]{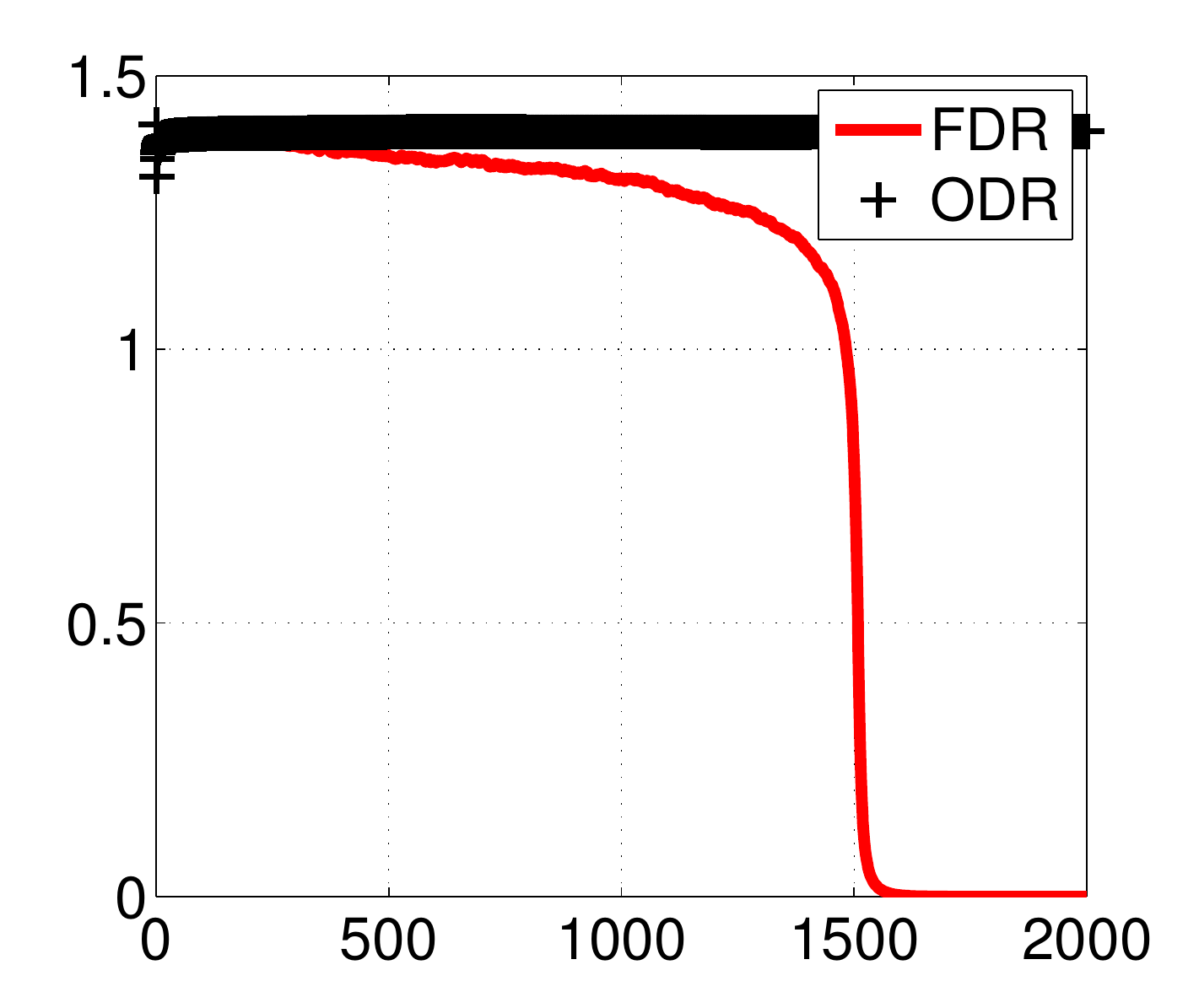}
}
\hspace{-.7cm}
\subfigure[TCB + RI]
{
\includegraphics[width=4.2cm]{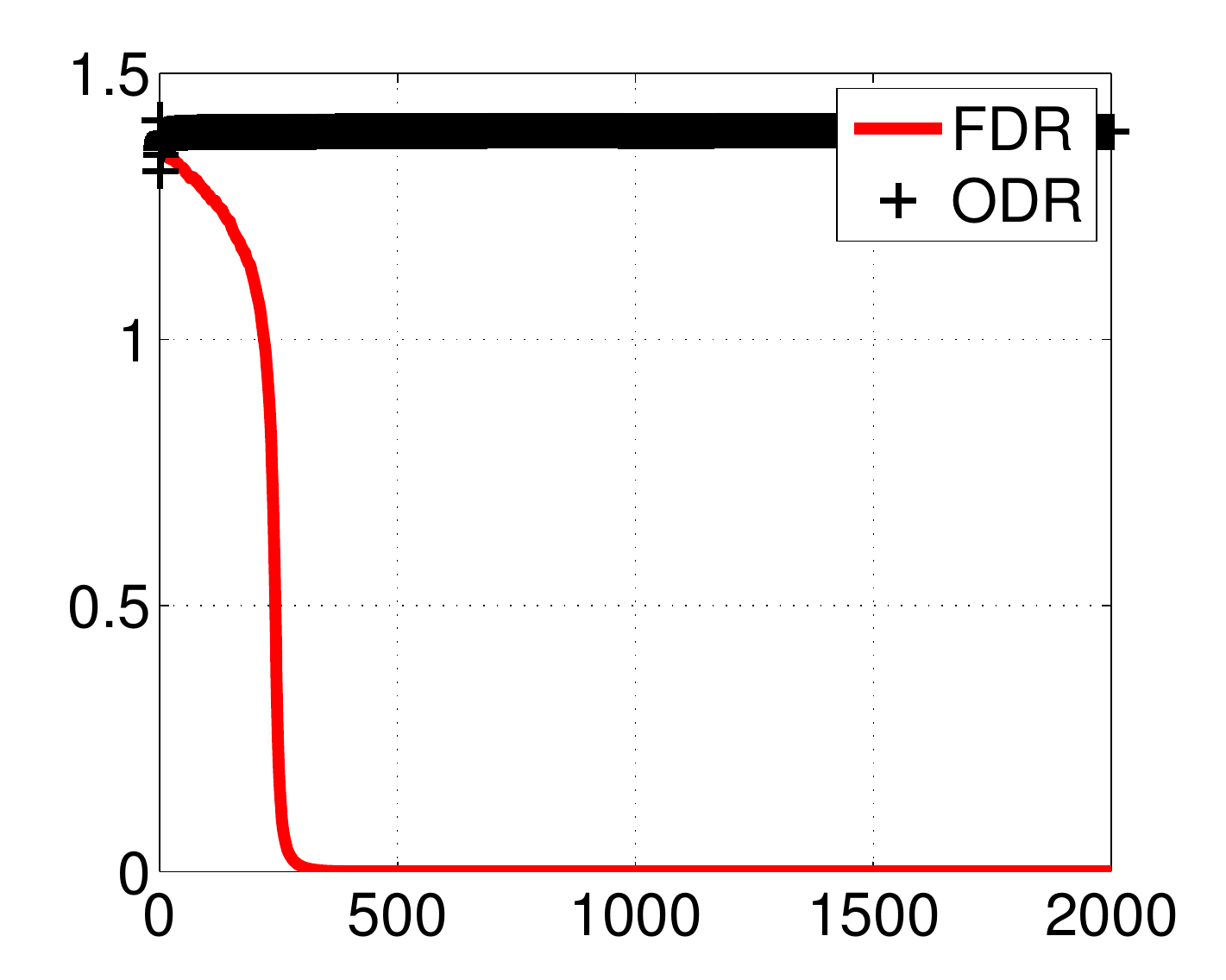}
}
\hspace{-.7cm}
\subfigure[TCB + CI]
{
\includegraphics[width=4.2cm]{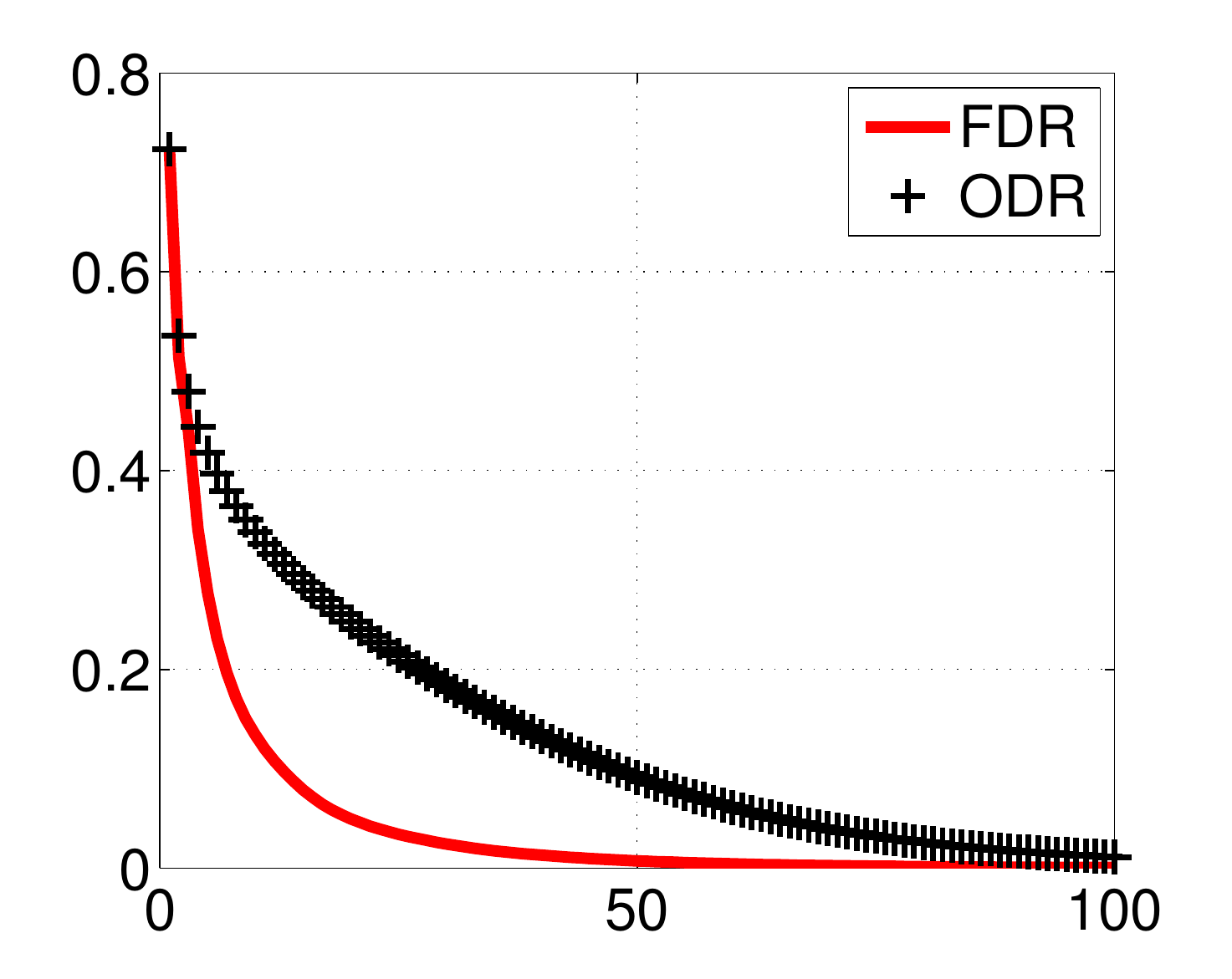}
}
}
\subfigure[RPP + RI]
{
\includegraphics[width=4.2cm]{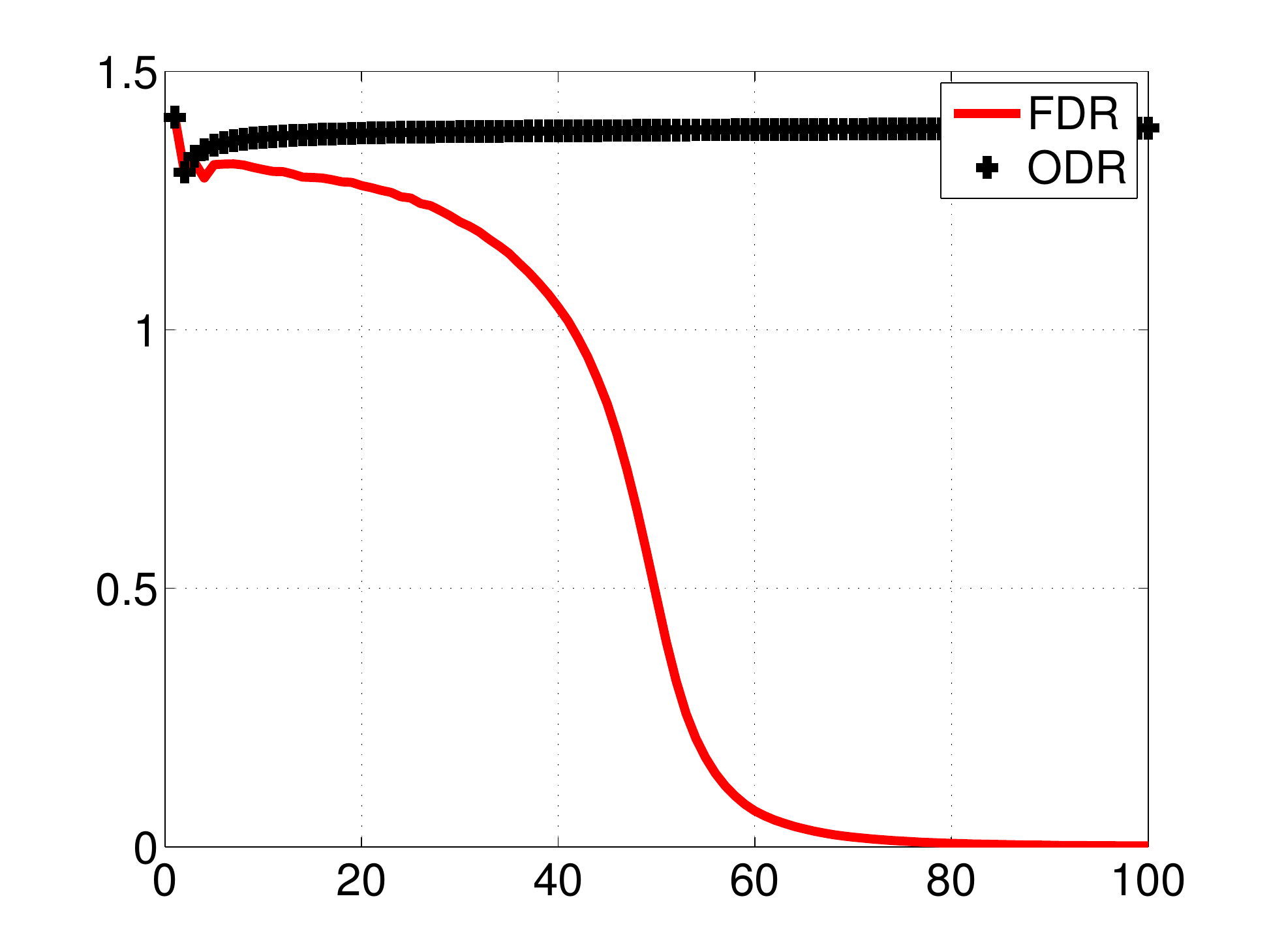}
}\hspace{-0.55cm}
\subfigure[RPP + CI]
{
\includegraphics[width=4.2cm]{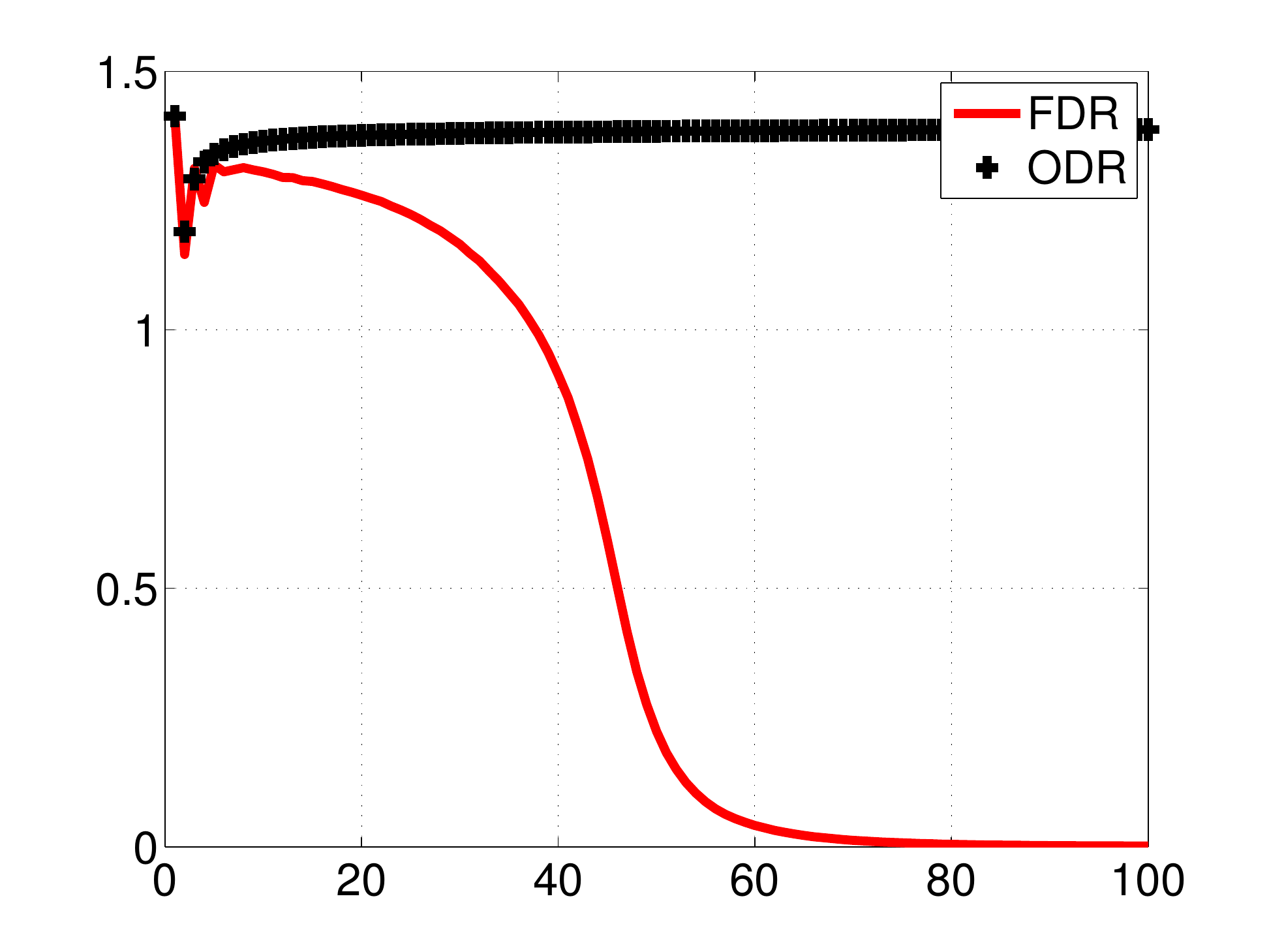}
}\hspace{-0.5cm}\subfigure[TCB + RI]
{
\includegraphics[width=4.2cm]{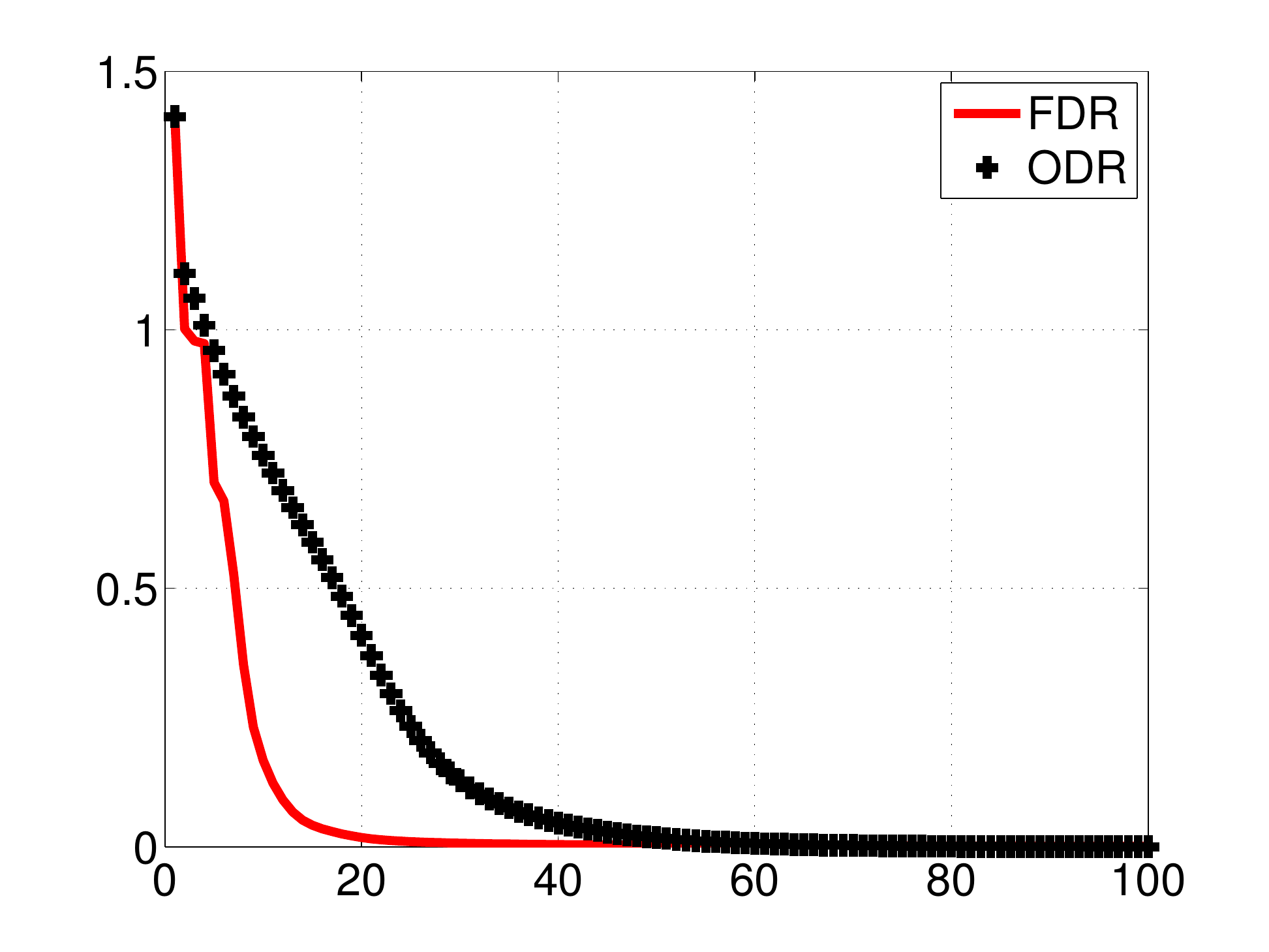}
}\hspace{-0.55cm}\subfigure[TCB + CI]
{
\includegraphics[width=4.2cm]{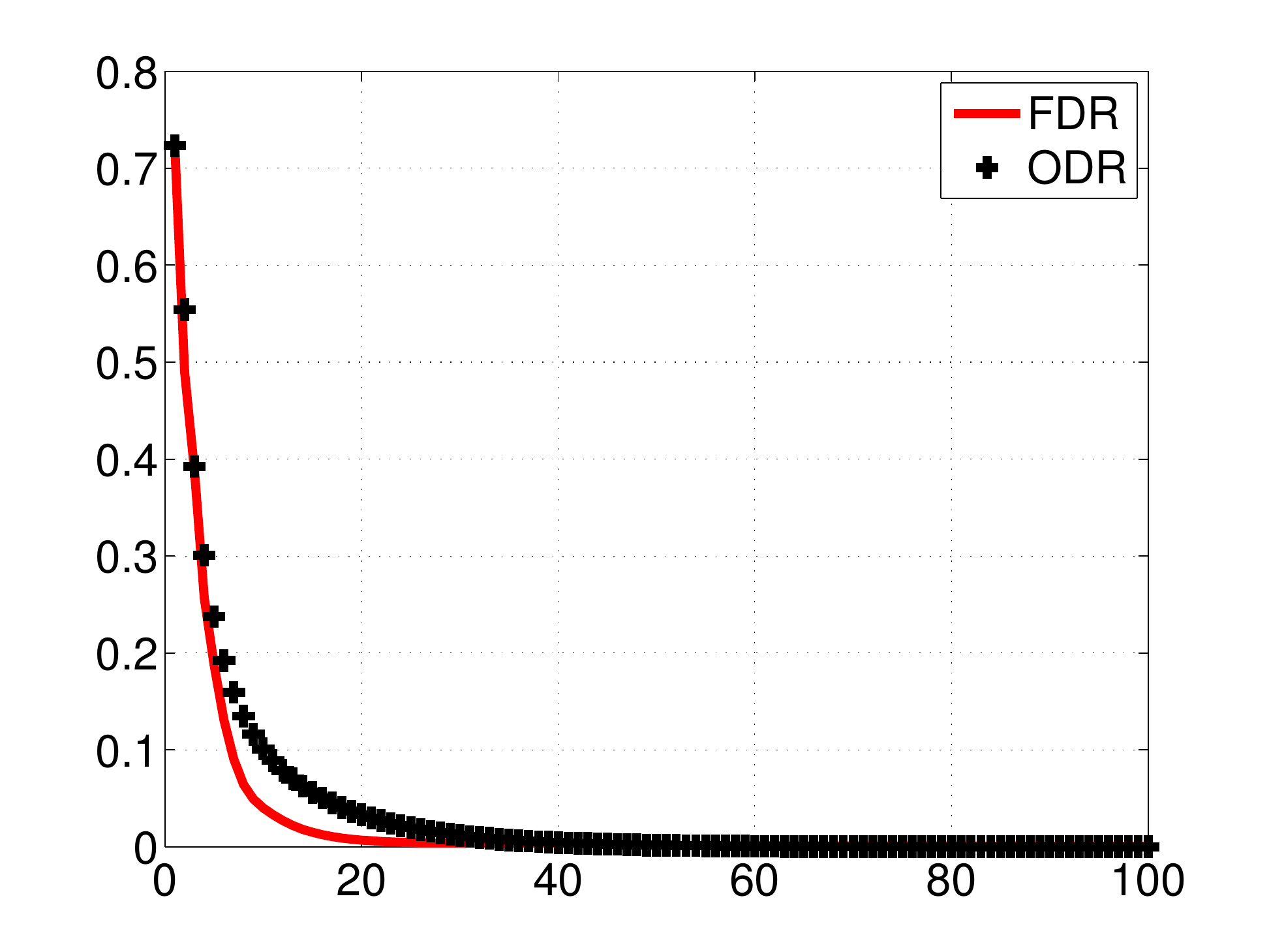}
}
\commentout{
\subfigure[RPP + RI]
{
\includegraphics[width=4.2cm]{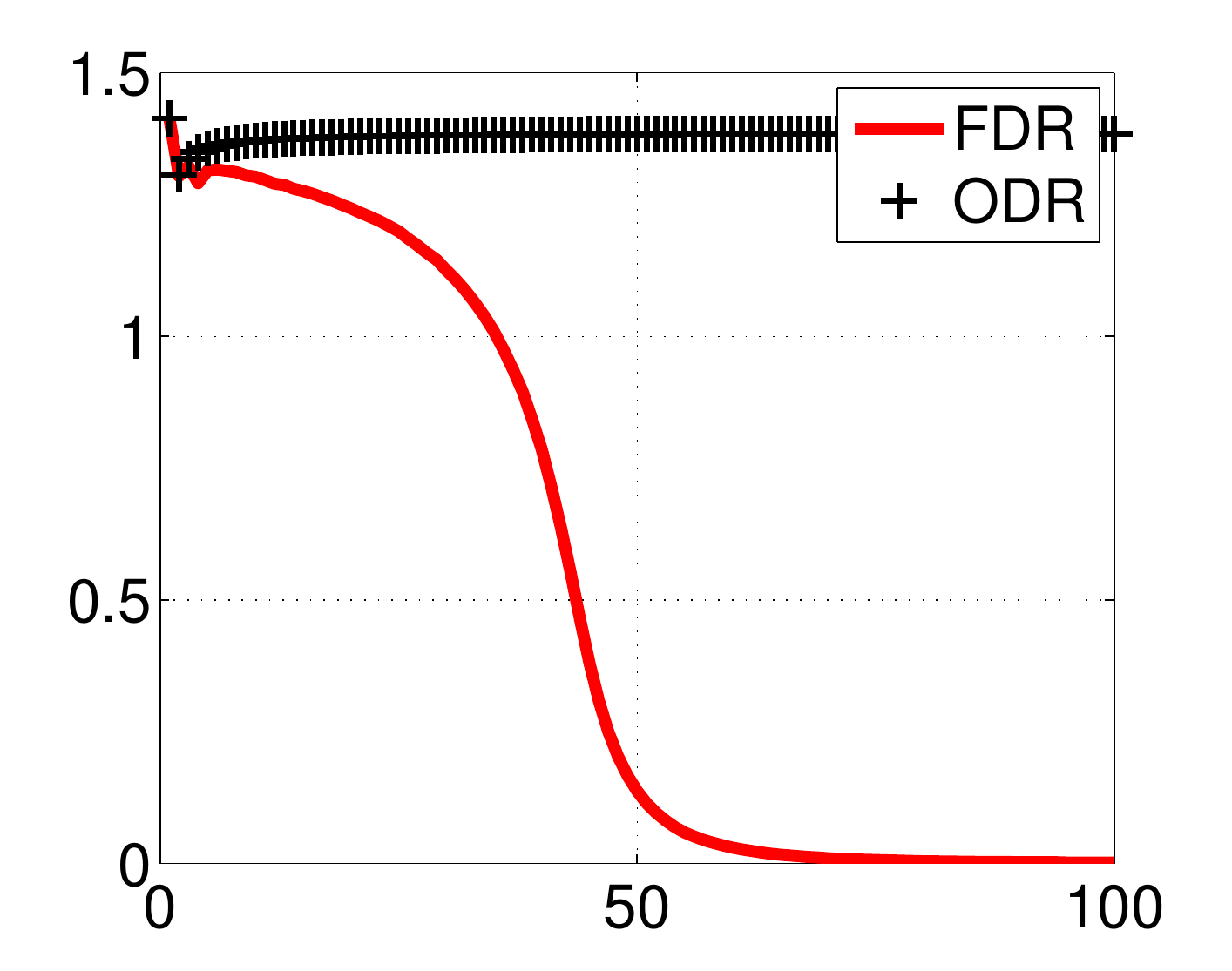}
}\hspace{-0.6cm}
\subfigure[RPP + CI]
{
\includegraphics[width=4.2cm]{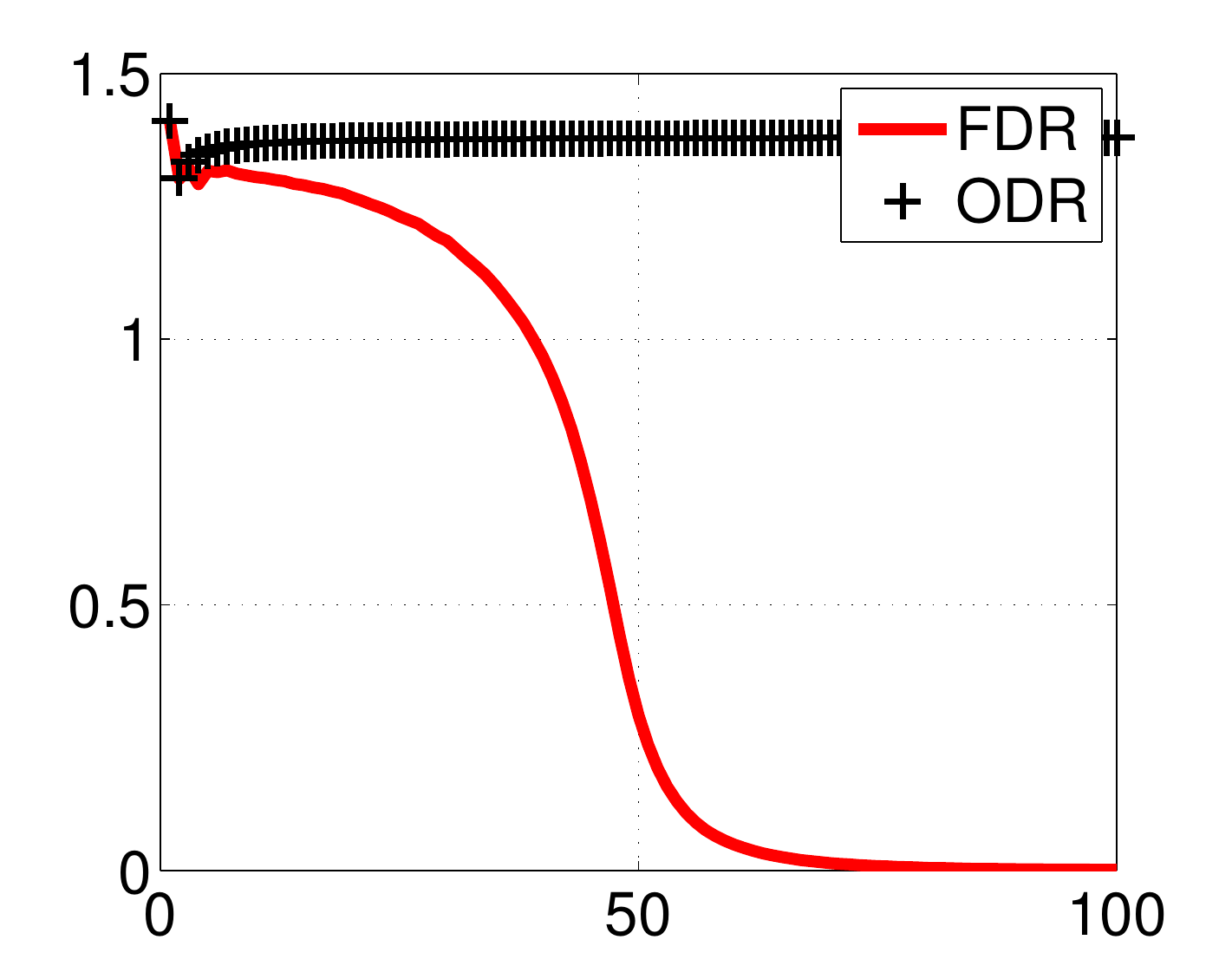}
}\hspace{-0.6cm}
\subfigure[TCB + RI]
{
\includegraphics[width=4.2cm]{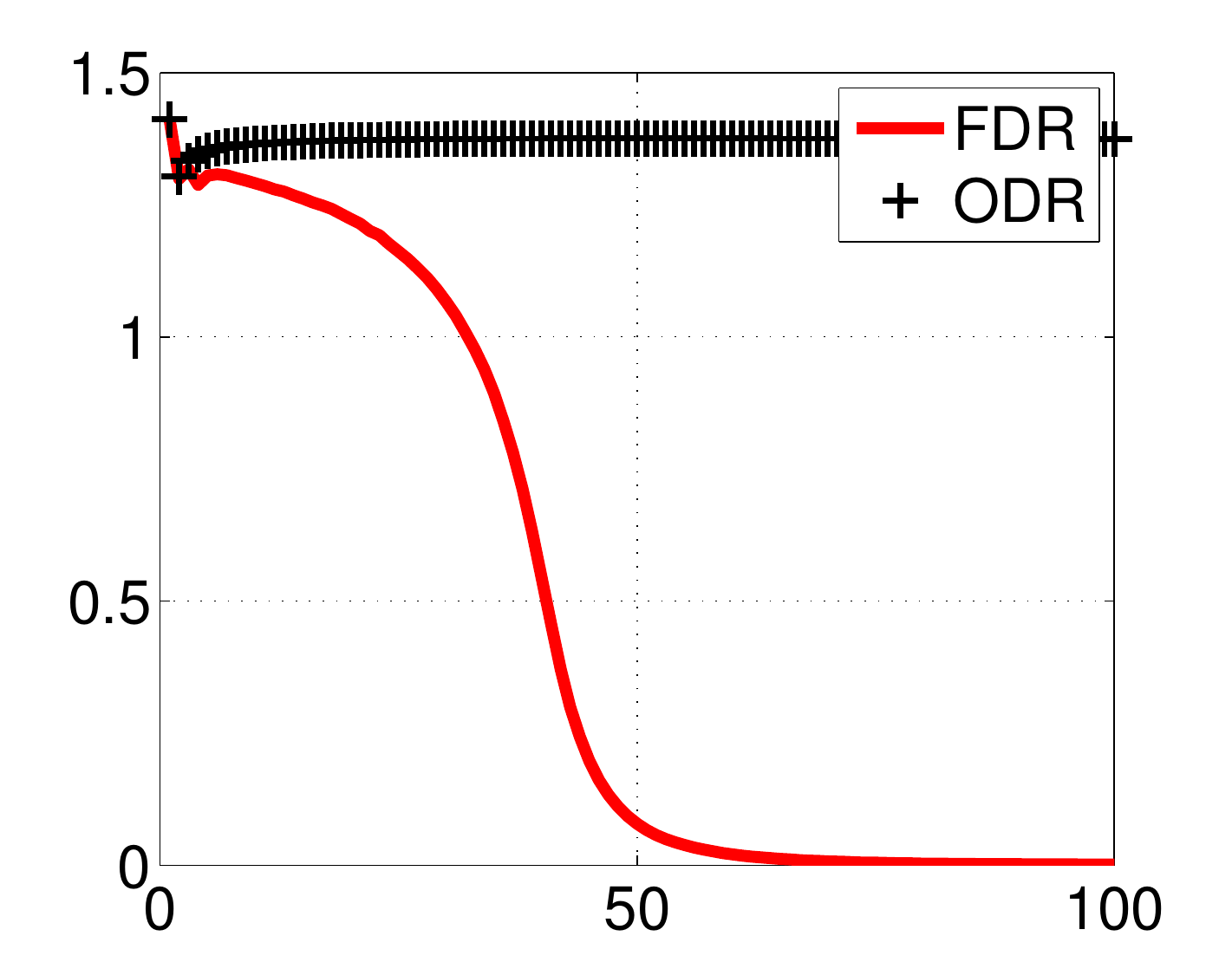}
}\hspace{-0.6cm}
\subfigure[TCB + CI]
{
\includegraphics[width=4.2cm]{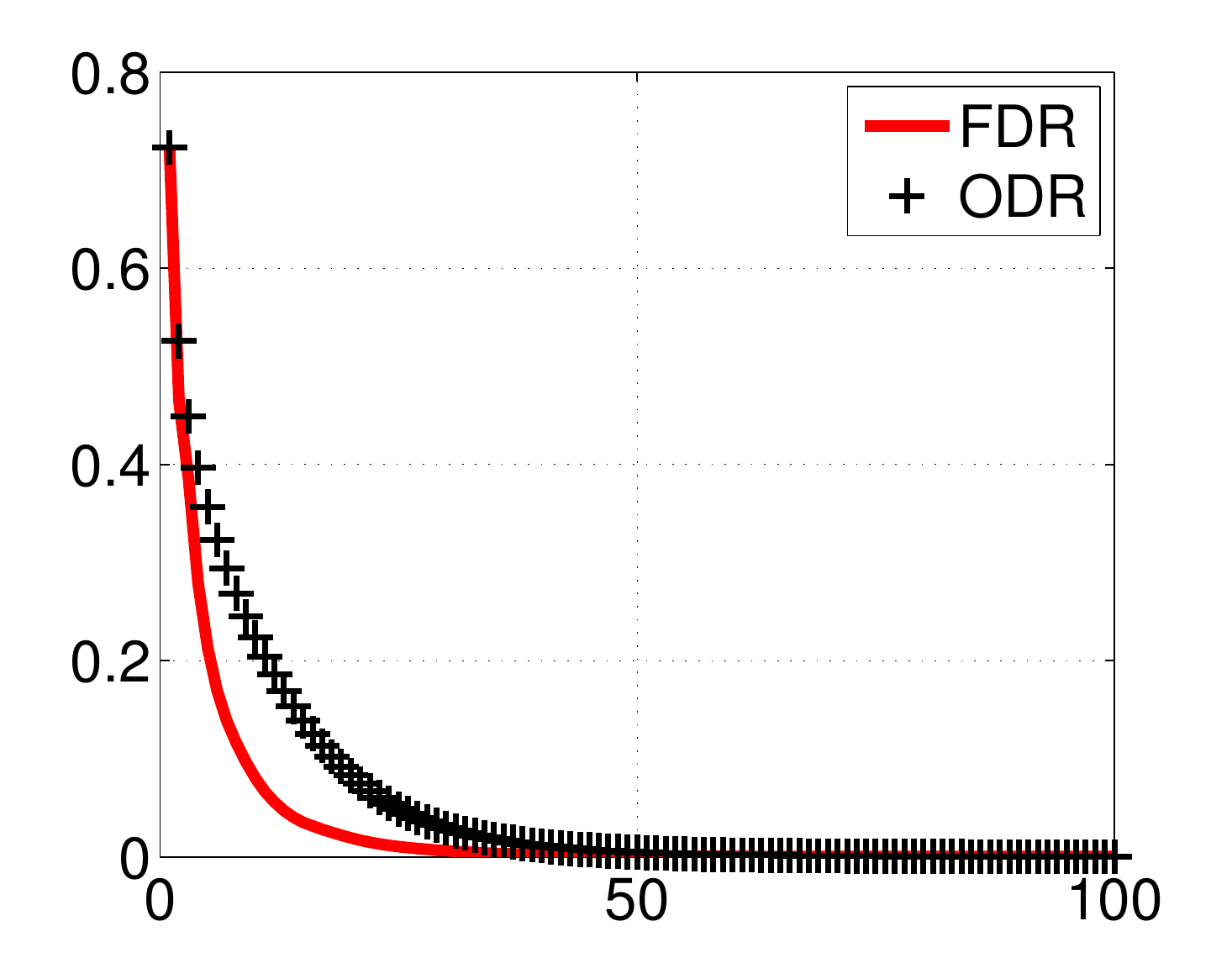}
}
}
\caption{Relative error  versus iteration with  3 patterns (a)-(d) and 
4 patterns  (e)-(h) (without oversampling in each pattern).}
\label{fig:three}
\label{fig:four}
\end{center}
\end{figure}

To test how DR performs in the setting of multiple patterns without
oversampling \cite{phaselift1, CLS1} we simulate the 3-pattern and
4-pattern cases with the propagation matrices \\

\beq \label{34'}
A^*=c \lt[\begin{matrix}
\Phi\,\, \diag\{\mu_1\}\\
...\\
\Phi\,\, \diag\{\mu_{\ell-1}\}\\
\Phi\,\, 
\end{matrix}\rt],\quad \ell=3,4,
\eeq
where $\Phi$ is
the standard (unoversampled)  discrete Fourier transform.  

Figure~\ref{fig:three} shows the result with  three patterns (a)-(d)
and four patterns (e)-(h), both 
without oversampling, i.e. $N=3|\cM|$ and
$N=4|\cM|$, respectively. Note that  the
number of data with four patterns is half of that with 2 oversampled patterns and yet the performance of the former
 is almost as good as that of the latter. 

Going from three patterns (Fig. \ref{fig:three} (a)-(b)) to four patterns 
(Fig. \ref{fig:three} (e)-(f))  reduces the number of iterations by almost 
an order of magnitude when RPP is the unknown image. The case with TCB has less room for improvement.

\commentout{
 \begin{figure}[htbp]
\begin{center}
\includegraphics[width=0.4\textwidth]{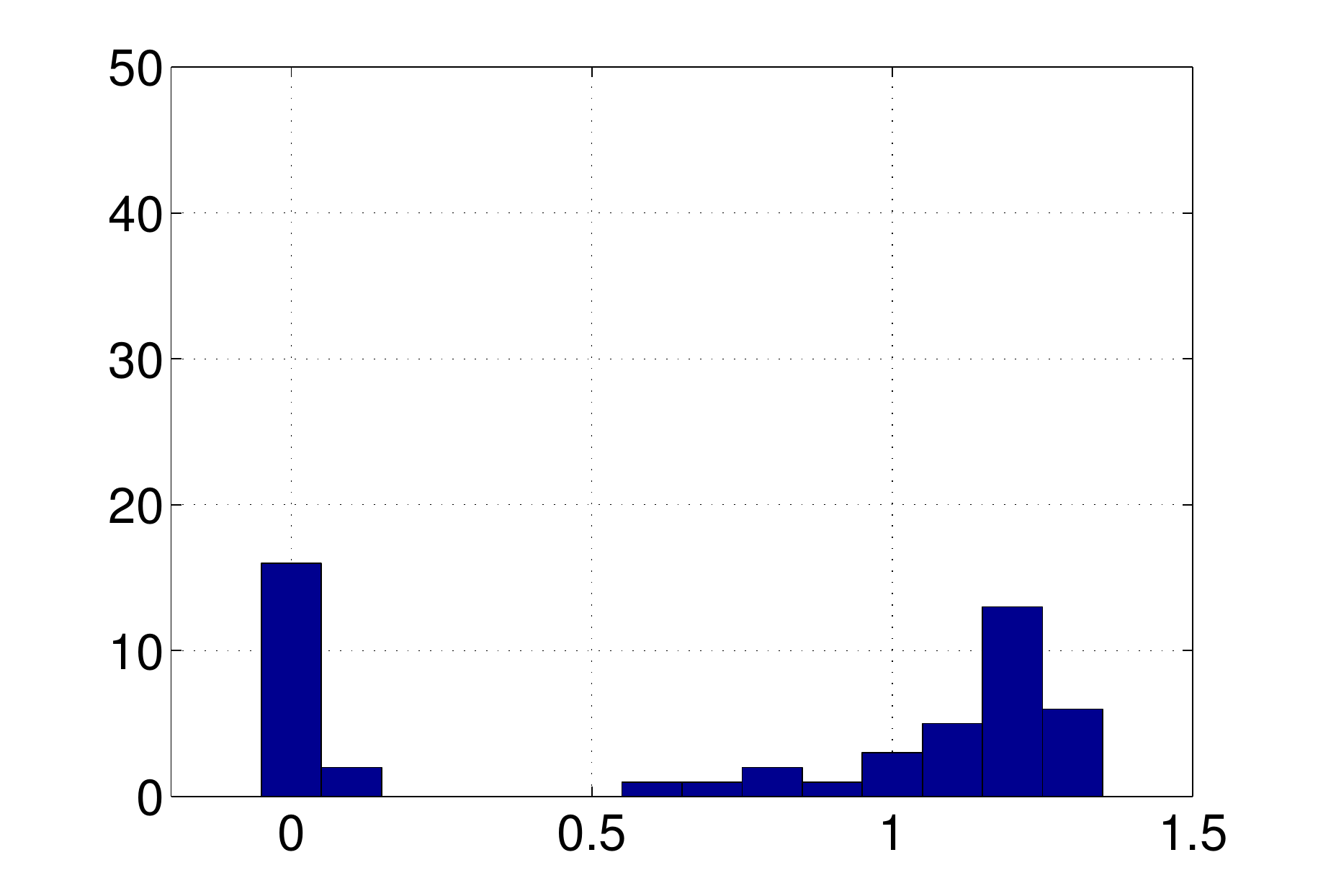}
\caption{
 Histogram of the relative error of ODR reconstruction of RPP with $\tilde n/n=6$ and
 50 independent random initializations. 
}
\label{HIOn=6}
\end{center}\label{HIOphase1}
\end{figure}
}

\subsection{Padding ratio}
 \begin{figure}[tbp]
\begin{center}
\includegraphics[width=6cm]{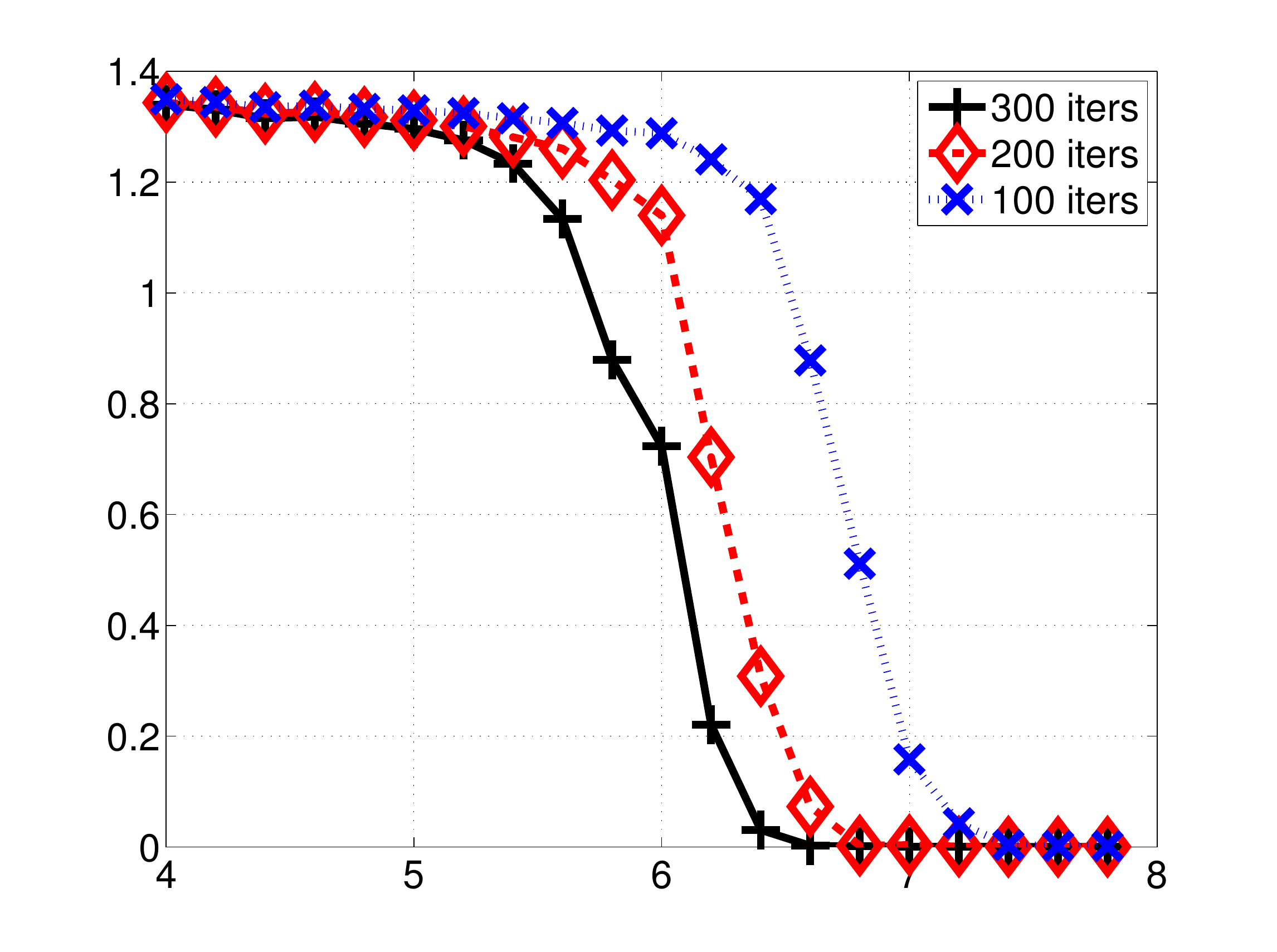}\quad
\includegraphics[width=5.5cm]{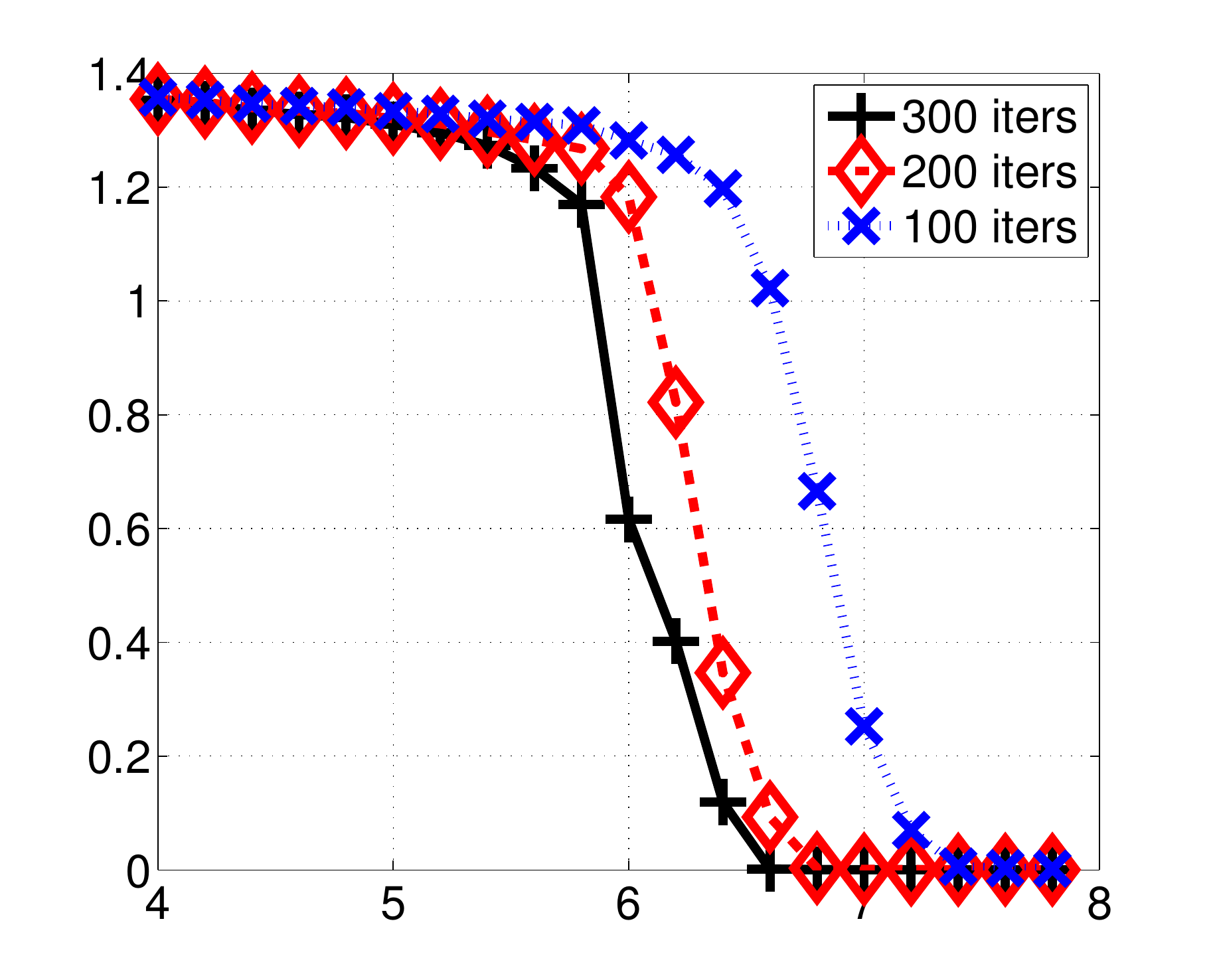}
\caption{
Relative error for RPP (left) and TCB (right) with various ratios $\tilde n/n$.
}
\label{HIOphase}
\label{fig:trans}
\end{center}
\end{figure}

Finally we test the effect of the padding ratio $\tilde n/n$ on the performance of ODR. 
For each $\tilde n/n \in [4,8]$, we conduct $50$ trials with independent, random initializations and average the relative errors. Recall that $\tilde n/n=4$ is the standard
padding  rate and at $\tilde n/n=8$ ODR is equivalent to FDR. 

Fig. \ref{fig:trans} shows the averaged  relative error versus the ratio $\tilde n/n$, demonstrating  an effect of phase transition which depends on the number of iterations. 
As the number of iterations increases, the threshold ratio decreases. The phase transition
accounts for  the sharp transition from stagnation at the standard  ratio $\tilde n/n=4$
to the rapid convergence at $\tilde n/n=8$ (i.e. FDR) seen in Fig. \ref{fig:two} (a)-(b)
and  Fig. \ref{fig:three} (a)-(b), (e)-(f) . 


\appendix
\section{Proof of Proposition \ref{thm3}}
In order to prove the uniqueness theorem for Fourier magnitude retrieval, we need to
take up the more elaborate notation in Section \ref{sec:coded}. 

\commentout{
 Let $\bn = (n_1,\ldots,n_d) \in \IZ^d$ and $\bz = (z_1,\ldots,z_d) \in \IC^d$. Define the multi-index notation $\bz^\bn = z_1^{n_1}z_2^{n_2},..., z_d^{n_d}$. Let 
  \begin{equation*}
 \cM = \{ \mathbf{0} \le \bn \le \bM\}, \quad \bM = (M_1,M_2,\ldots,M_d). 
 \end{equation*}
Here  $\mathbf{m} \le \bn$ if $m_j \le n_j, \forall j$. Denote $|\cM|  =\displaystyle \prod_{j=1}^d (M_j+1)$. 
}

Let 
$$F(\bz) = \sum_{\bn} f(\bn) \bz^{-\bn}$$
be the $z$-transform of $f$.
According to the fundamental theorem of algebra, $F(\bz)$ can   be written uniquely as
\beq
\label{temeq2}
F(\bz) = \alpha \bz^{-\bn_0} \prod_{k=1}^p F_k(\bz),
\eeq
where $\bn_0$ is a vector of nonnegative integers, $\alpha$ is a complex coefficient, and $F_k(\bz)$ are nontrivial irreducible  polynomials in $\bz^{-1}$.

\commentout{
From the calculation 
 \beq
  |F(e^{i2\pi\bw})|^2&=& \sum_{\bn =-\bM}^{\bM}\sum_{\mbm+\bn\in \cM} f(\mbm+\bn)\overline{f(\mbm)}
   e^{-\im 2\pi \bn\cdot \bw}\nn
   \eeq
   we see that the Fourier magnitude measurement
   is equivalent to the standard discrete Fourier measurement of
   the correlation function 
       \beqn
	  R_f(\bn)=\sum_{\mbm\in \cM} f(\mbm+\bn)\overline{f(\mbm)}
	  \label{autof}
	  \eeqn
if sampled at the lattice 
 \begin{equation}
\cL = \Big\{\bw=(w_1,...,w_d)\ | \ w_j = 0,\frac{1}{2 M_j + 1},\frac{2}{2M_j + 1},...,\frac{2M_j}{2M_j + 1}\Big\}
\end{equation}
which is $2^d$ times of the grid of the original image. 
Note that
$R_f$ is defined on the extended grid
 \begin{equation*}
 \widetilde \cM = \{ -\bM \le \bn \le \bM\}, \quad \bM = (M_1,M_2,\ldots,M_d). 
 \end{equation*}
}

 Define the shift 
 \[
 f_{\mathbf{m}+}(\cdot) = f(\mathbf{m}  + \cdot), \quad  f_{\mathbf{m}-}(\cdot) = f(\mathbf{m} - \cdot). 
 \]

\begin{definition}[Conjugate Symmetry]
A polynomial $X(\bz)$ in $\bz^{-1}$ is said to be conjugate symmetric if, for some vector $\bk$ of positive integers and some $\theta \in [0,2\pi)$, 
\begin{equation*}
X(\bz) = e^{i \theta} \bz^{-\bk} \overline{X(\bar{\bz}^{-1})}. 
\end{equation*}
In other words, the ratio between $X(\bz)$ and its conjugate inversion
is a monomial in $\bz^{-1}$ times a complex number of unit modulus. 
\end{definition}

A conjugate symmetric polynomial may be reducible, irreducible, trivial, or nontrivial.  Any  monomial $\bz^{-\bk}$ is conjugate symmetric.

\begin{prop}
Let $f$ be a finite array whose $z$-transform has no conjugate symmetric factors. If the $z$-transform $G$ of another array $g $ satisfies $\measuredangle{F(e^{2\pi i \bom})} - \measuredangle{G (e^{2\pi i \bom})}\in \{0, \pi\}$,  $
\forall\bom\in \cL$ 
 then $g = c f$ for some constant $c\in\IR$.
 

\label{thm:MR}
\end{prop}

The real-valued version of the above proposition is given in \cite{Hayes}.
For the reader's convenience, we provide the proof for the complex setting.

\begin{proof}
Consider the array $h$ defined by 
$$h = f\star \  \overline{g(-\cdot)} $$
whose $z$-transform is 
$$H(\bz) = F(\bz)\overline{G(\bar{\bz}^{-1})}.$$
Note that  $h$ is defined on  $\widetilde \cM$, instead of $\cM$,  so $H(\bz)$ is completely determined by sampling $H$ on $\cL$.

Since 
$$\measuredangle{H(e^{2\pi i\bom})} = \measuredangle{F(e^{2\pi i\bom})} - \measuredangle{G(e^{2\pi i\bom})}$$
it follows $H(e^{2\pi i\bom})$  is {real-valued}.  By analytic continuation, we have
$$H(\bz) = \overline{H(\bar{\bz}^{-1})}$$
and
\begin{equation}
F(\bz)\overline{G(\bar{\bz}^{-1})} = \overline{F(\bar{\bz}^{-1})}G(\bz).
\label{temeq1}
\end{equation}
Multiplying both sides of \eqref{temeq1} by $\bz^{-\bM}$ results in the following polynomial equation in $\bz^{-1}$:
\begin{equation}
F(\bz)\overline{G(\bar{\bz}^{-1})}\bz^{-\bM} =\bz^{-\bM} \overline{F(\bar{\bz}^{-1})}G(\bz).
\label{temeq0}
\end{equation}

 We observe $\bn_0=0$ in  view of  (\ref{temeq2}) and the assumption that $F(\bz)$ has no conjugate symmetric factor.  
 We also have
 \begin{equation}
\bz^{-\bM} \overline{F(\bar{\bz}^{-1})} = \tilde \alpha \bz^{-\bn_1}\prod_k \tilde{F}_k(\bz),
\label{temeq3}
\end{equation}
where $\tilde{F}_k(\bz)$ are the nontrivial irreducible non-conjugate symmetric  polynomials in $\bz^{-1}$ of the form $\tilde{F}_k(\bz) = \bz^{-\bM+\mathbf{p}_k}\overline{F_k(\bar{\bz}^{-1})}$ for some vector $\mathbf{p}_k$ of positive integers. 

Writing 
\begin{equation}
G(\bz) = \beta\bz^{-\mbm_0} \prod_{\ell} G_{\ell}(\bz),
\label{temeq4}
\end{equation}
where $G_\ell(\bz)$ are nontrivial irreducible  polynomials in $\bz^{-1}$, we have
\begin{equation}
\bz^{-\bM} \overline{G(\bar{\bz}^{-1})} = \tilde\beta \bz^{-\mbm_1} \prod_\ell \tilde{G}_\ell(\bz),
\label{temeq5}
\end{equation}
where $\tilde{G}_\ell(\bz)$ are the nontrivial irreducible  polynomials in $\bz^{-1}$ of the form  $\tilde{G}_\ell(\bz) = \bz^{-\bM + \mathbf{q}_\ell}\overline{G_\ell(\bar{\bz}^{-1})}$ for some vector $\mathbf{q}_\ell$ of positive integers.

Plugging \eqref{temeq2},\eqref{temeq3}, \eqref{temeq4} and \eqref{temeq5} in \eqref{temeq0} yields
\begin{equation}
\alpha\tilde\beta\bz^{-\mbm_1} \prod_k F_k(\bz) \prod_\ell \tilde{G}_\ell (\bz) = \tilde\alpha\beta \bz^{-\bn_1-\mbm_0} \prod_k \tilde{F}_k(\bz) \prod_\ell G_\ell(\bz).
\end{equation}
Each nontrivial irreducible factor $F_k(\bz)$ must be equal to some $\tilde{F}_{k'}(\bz)$ or some $G_{\ell '}(\bz)$. However, if $F_k(\bz) = \tilde{F}_{k}(\bz)$, then $F_k(\bz)$ is a conjugate symmetric factor. If, on the other hand, $F_k(\bz) = \tilde{F}_{k'}(\bz)$ for some $k' \neq k$, then $F_k(\bz)F_{k'}(\bz) = \tilde{F}_{k'}(\bz)\tilde F_{k}(\bz)$ is a conjugate symmetric factor. Both cases, however, are excluded by the assumption  that the $z$-transform of $f$ does not  have conjugate symmetric factors. 

Hence each $F_k$ (rest. $\tilde F_k$) must be equal to some $G_{\ell}$ (rest. $\tilde G_\ell$) and we can write 
\begin{equation}
G(\bz) = Q(\bz) F(\bz)
\label{temeq6}
\end{equation}
where $Q(\bz)$ is a polynomial in $\bz^{-1}$, i.e.
\[
Q(\bz)=\sum_{\bn\geq 0} c_{\bn} \bz^{-\bn}.
\]

By the assumption that  $\measuredangle{F(e^{2\pi i \bom})} - \measuredangle{G(e^{2\pi i \bom})}\in \{0,\pi\} $ we have $Q(e^{2\pi i\bom})\in \IR,\forall \bom\in \cL,$ and hence $\bar c_{\bn}=c_{-\bn}=0$ except for $\bn=0$ in which case $ c_0\in \IR$. Therefore, 
$Q=c_0\in \IR$ and this is what we start out to prove.
\end{proof}

\begin{prop} \cite{unique} 
Let $x_0$ have rank $\ge 2$. Let $\{\mu (\bn) \}$ be independent and continuous random variables on the unit circle of the complex plane. Then, the $z$-transform $F(\bz)$ of $f(\bn):=\mu (\bn)x_0(\bn)$ is irreducible up to a power of $\bz^{-1}$ with probability one.
\label{NonIR}
\end{prop}
For the proof of Proposition  \ref{NonIR} see Theorem 2 of \cite{unique}.

We next show  that  the $z$-transform of $\{\mu (\bn)x_0(\bn)\}$ is almost surely irreducible up to a power $\bz^{-1}$ and not conjugate symmetric.  

\begin{prop}
 Let $\{\mu (\bn)\}$ be independent and continuous random variables on the unit circle of the complex plane. Let $f(\bn):=\mu(\bn)x_0(\bn)$. Then  the $z$-transforms  of  both $f_{\mathbf{t}+} $ and $\overline{f_{\mathbf{t}-}}$ are 
almost surely not conjugate symmetric $\forall \ \mathbf{t}$.
\label{NonCS}
\end{prop}

\begin{proof}
The  $z$-transform 
\begin{equation}
{F}_{\bt+}(\bz) = \sum_{\bn} f(\bt+\bn)\bz^{-\bn}.
\label{temeq21}
\end{equation}
 is conjugate symmetric if 
\begin{equation}
{F}_{\bt+}(\bz)  = e^{i \theta} \bz^{-\bk} \overline{{F}_{\bt+}(\bar{\bz}^{-1})}
\label{temeq22}
\end{equation}
for some vector $\bk$ of positive integers and some $\theta \in [0,2\pi)$. Plugging \eqref{temeq21} in \eqref{temeq22} yields
$$\sum_{\bn} f(\bt+\bn)\bz^{-\bn} = e^{i\theta}\bz^{-\bk} \sum_{\bn'}\overline{f(\bt+\bn')}\bz^{\bn'},$$
which implies
\begin{equation}
f(\bt+\bn) = e^{i\theta} \overline{f(\bt+\bk-\bn)}, \ \forall \bn.
\label{temeq23}
\end{equation}
However, $x_0$ is deterministic, and $\{\mu (\bn)\}$ are independent and continuous random variables on $\mathbb{S}^1$, so \eqref{temeq23} fails with probability one for any $\bk$. There are finitely many choices of $\bk$, so the $z$-transform of $f_{\mathbf{t}+}$ is almost surely not conjugate symmetric.

Similarly, the $z$-transform of $\overline{f_{\mathbf{t}-}}$ is also almost surely {\em not} conjugate symmetric.
\end{proof}

{\bf Acknowledgements.} Research is supported in part by  US NSF grant DMS-1413373 and Simons Foundation grant 275037.


\end{document}